\documentclass[a4paper,reqno]{amsart}
\usepackage[reqno]{amsmath}
\usepackage{esint,latexsym,amssymb,amsthm}
\usepackage[abbrev,backrefs]{amsrefs}
\usepackage{a4wide}
\usepackage{tocvsec2}
\usepackage{enumerate}
\usepackage{enumitem}
\usepackage{wasysym}
\usepackage{pifont}
\usepackage{mathrsfs}
\usepackage{CJK}

\usepackage{tikz}
\usepackage{mathdots}
\usepackage{yhmath}
\usepackage{cancel}
\usepackage{color}
\usepackage{siunitx}
\usepackage{array}
\usepackage{multirow}
\usepackage{gensymb}
\usepackage{tabularx}
\usepackage{booktabs}
\usetikzlibrary{fadings}
\usetikzlibrary{patterns}

\newtheorem{theorem}{Theorem}[section]
\newtheorem{corollary}{Corollary}[section]

\newtheorem{lemma}{Lemma}[section]
\newtheorem{proposition}{Proposition}[section]
\theoremstyle{remark}
\newtheorem{remark}{Remark}[section]
\theoremstyle{definition}

\numberwithin{equation}{section}

\newcommand{\N}{{\mathbb N}}

\newcommand{\R}{{\mathbb R}}

\renewcommand{\O}{\mathcal{O}}
\renewcommand{\I}{\mathcal{I}}
\renewcommand{\H}{{\mathcal{H}_q}}
\newcommand{\Hrad}{{\mathcal{H}_{q,rad}}}
\newcommand{\D}{{\mathcal{D}_q}}
\newcommand{\Drad}{{\mathcal{D}_{q,rad}}}
\renewcommand{\P}{\mathcal{P}}
\newcommand{\eps}{{\varepsilon}}
\newcommand{\weakto}{\rightharpoonup}
\newcommand{\tV}{\widetilde{V}}
\newcommand{\abs}[1]{\lvert #1 \rvert}

\begin{document}
\title[Limit profiles for Choquard equations with local repulsion]
{Limit profiles for singularly perturbed\\ Choquard equations with local repulsion}%

\author{Zeng Liu}
\thanks{ZL was supported by NSFC Grant Numbers 11901418, 11771319, 12171470.}
\address{Department of Mathematics\\ Suzhou University of Science and Technology\\ Suzhou 215009\\ P.R. China}
\email{zliu@mail.usts.edu.cn (Zeng Liu)}

\author{Vitaly Moroz}
\address{Department of Mathematics\\ Swansea University\\ Fabian Way\\
Swansea SA1~8EN\\ Wales, United Kingdom}	
\email{v.moroz@swansea.ac.uk}



\date{\today}

\begin{abstract}
We study Choquard type equation of the form
$$-\Delta u +\varepsilon u-(I_{\alpha}*|u|^p)|u|^{p-2}u+|u|^{q-2}u=0\quad\text{in $\mathbb R^N$},\eqno{(P_\varepsilon)}$$
where $N\geq3$, $I_\alpha$ is the Riesz potential with $\alpha\in(0,N)$, $p>1$, $q>2$ and $\varepsilon\ge 0$.
Equations of this type describe collective behaviour of self-interacting many-body systems.
The nonlocal nonlinear term represents long-range attraction while the local nonlinear term represents short-range repulsion.
In the first part of the paper for a nearly optimal range of parameters we prove the existence and study regularity and qualitative properties of positive groundstates of $(P_0)$ and of $(P_\varepsilon)$ with $\varepsilon>0$. We also study the existence of a compactly supported groundstate for an integral Thomas--Fermi type equation associated to $(P_\eps)$.
In the second part of the paper,  for $\varepsilon\to 0$ we identify six different asymptotic regimes and provide a characterisation of the limit profiles of the groundstates of $(P_\varepsilon)$ in each of the regimes. We also outline three different asymptotic regimes in the case $\varepsilon\to\infty$. In one of the asymptotic regimes positive groundstates of $(P_\varepsilon)$ converge to a compactly supported Thomas--Fermi limit profile. This is a new and purely nonlocal phenomenon that can not be observed in the local prototype case of $(P_\varepsilon)$ with $\alpha=0$.  In particular, this provides a justification for the Thomas--Fermi approximation in astrophysical models of self--gravitating Bose--Einstein condensate.
\end{abstract}

\maketitle

\setcounter{tocdepth}{1}
\tableofcontents

\newpage
\section{Introduction}

\subsection{Background}
We are concerned with the asymptotic properties of positive groundstate solutions of the Choquard type equation
\begin{equation}\tag{$P_\eps$}\label{eqPeps}
 -\Delta u +\eps u-(I_{\alpha}*|u|^p)|u|^{p-2}u+|u|^{q-2}u=0\quad\text{in $\R^N$},
\end{equation}
where $N\geq 3$, $p>1$, $q>2$ and $\eps\ge 0$. Here $I_\alpha(x):=A_\alpha|x|^{-(N-\alpha)}$ is the Riesz potential with $\alpha\in(0,N)$ and $*$ denotes the standard convolution in $\R^N$. The choice of the normalisation constant $A_\alpha:=\frac{\Gamma((N-\alpha)/2)}{\pi^{N/2}2^{\alpha}\Gamma(\alpha/2)}$
ensures that $I_\alpha(x)$ could be interpreted as the Green function of $(-\Delta)^{\alpha/2}$ in $\R^N$, and that the semigroup property $I_{\alpha+\beta}=I_\alpha*I_\beta$ holds for all $\alpha,\beta\in(0,N)$ such that
$\alpha+\beta<N$, see for example \cite{DuPlessis}*{pp.\thinspace{}73--74}.

Equation
\begin{equation*}\tag{$\mathscr{C}$}\label{eqC}
-\Delta v+\eps v-(I_{\alpha}*|v|^p)|v|^{p-2}v=0\quad\text{in $\R^N$}
\end{equation*}
is often known as the Choquard equation and had been studied extensively during the last decade, see \cite{MS17} for a survey. In this work we are interested in the case when the standard Choquard equation is modified by including the local repulsive $|u|^{q-2}u$ term.

If $u_\eps$ is a solution of \eqref{eqPeps} with $N=3$, $\alpha=2$, $p=2$ and $q=4$ then $\psi(t,x):=e^{i\eps t}u_\eps(x)$ is a standing wave solution of the time-dependent equation
\begin{equation}\label{GPP}
i\partial_t\psi=-\Delta\psi-(I_2*|\psi|^2)\psi+|\psi|^2\psi, \quad (t,x)\in \R\times\R^3.
\end{equation}
Equation of this form models, in particular, self--gravitating Bose--Einstein condensates with repulsive short--range interactions, which describe astrophysical objects such as boson stars and, presumably, dark matter galactic halos. In this context, \eqref{GPP} was introduced and studied under the name of Gross--Pitaevskii--Poisson equation in \citelist{\cite{Wang}\cite{Bohmer-Harko}\cite{Chavanis-11}}, see a survey paper \cite{Chavanis-15}.

More generally, equation \eqref{eqPeps} can be seen as a stationary NLS with  an attractive long range interaction, represented by the nonlocal term, coupled with a repulsive short range interaction, represented by the local nonlinearity. While for the most of the relevant physical applications $p=2$, the values $p\neq 2$ may appear in several relativistic models of the density functional theory \citelist{\cite{BPO-2002}\cite{BLS-2008}\cite{BGT-2012}}.

In this work we are specifically interested in the case where $\eps>0$ is a small (or large) parameter and all other parameters are fixed.
Our main goal is to understand the behaviour of groundstate solutions of \eqref{eqPeps} when $\eps\to 0$.
We also discuss the case $\eps\to\infty$, which is to some extent dual to $\eps\to 0$.
The local prototype of \eqref{eqPeps} and a formal limit of \eqref{eqPeps} as $\alpha\to 0$ is the equation
\begin{equation}\label{eq01-loc}
 -\Delta u +\eps u-|u|^{2p-2}u+|u|^{q-2}u=0\quad\text{in $\R^N$}.
\end{equation}
It is well--known that this equation admits a unique positive solution in $H^1(\R^N)$ for any $1<2p<q<\infty$ provided that $\eps>0$ is sufficiently small, and has no finite energy solutions for large $\eps$. This result goes back to Strauss \cite{Strauss}*{Example 2} and Berestycki and Lions \cite{BL-I}*{Example 2} (see \cite{MM-14}*{Theorem A} for a precise existence statement and further references). A complete characterization of all possible asymptotic regimes in \eqref{eq01-loc} as $\eps\to 0$ was obtained in \cite{MM-14}, see also earlier work \cite{Muratov}. Essentially, {\em three} different limit regimes were identified in \cite{MM-14}, depending on whether $p$ is less, equal or bigger than the critical Sobolev exponent $p^\ast=\frac{N}{N-2}$.
Recently, \eqref{eq01-loc} had been revisited in \cite{Lewin} where nondegeneracy of ground--sates and sharp asymptotics of the $L^2$--norm of the ground states as $\eps\to 0$ had been described in connection with the uniqueness conjecture in the $L^2$--constraint minimization problem associated to \eqref{eq01-loc}.
See also \cite{MSVM}, where the same problem is studied with the opposite sign of the $|u|^{q-2}u$--term.

\subsection{Existence and properties of groundstates for \eqref{eqPeps}}
We are not aware of a systematic study of ground--sates of Choquard equation \eqref{eqPeps}. First existence results seem to appear in \cite{Mugnai} in the case $N=3$, $\alpha=2$, $p=2$. See also \citelist{\cite{Seok14}\cite{LiLiShi}\cite{Bhattarai}\cite{Feng-Chen-Ren}} and references therein for further results which however do not cover the optimal ranges of parameters. The planar case with the logarithmic convolution kernel was studied in \cite{Weth} but since the kernel is sign-changing this requires different techniques. Near optimal existence results for the Choquard equation of type \eqref{eqPeps} with an attractive local perturbation (the opposite sign of the local nonlinear term) were recently obtained in \citelist{\cite{Ma-Li-Zhang-19}\cite{Ma-Li-19}}.

Our first goal in this work is to establish the existence of ground-sate solutions of Choquard equation \eqref{eqPeps} for an optimal range of parameters. By a {\em groundstate} solution of \eqref{eqPeps} we understand a
weak solution $u \in H^1(\R^N)\cap L^q(\R^N)$ which has a minimal energy
\begin{equation*}\label{eq02}
\mathcal{I}_{\eps}(u):=\frac{1}{2}\int_{\R^N}|\nabla u|^2dx+\frac{\eps}{2}\int_{\R^N}|u|^2dx-\frac{1}{2p}\int_{\R^N}(I_{\alpha}*|u|^p)|u|^pdx+\frac{1}{q}\int_{\R^N}|u|^q dx
\end{equation*}
amongst all nontrivial finite energy solutions of \eqref{eqPeps}.
Remarkably, and in contrast with its local prototype \eqref{eq01-loc}, we prove that ground states for \eqref{eqPeps} exist for every $\eps>0$. We also establish some qualitative properties of the solutions \eqref{eqPeps} such as regularity and decay at infinity. These properties are similar to the standard Choquard equation \eqref{eqC}. Note that we do not study the uniqueness or non-degeneracy of the groundstates of \eqref{eqPeps} and we are not aware of any even partial results in this direction. We believe this is a very difficult open problem. Our results do not rely and do not require the uniqueness or non-degeneracy.

Essential tools to control the nonlocal term in $\mathcal{I}_\eps$ are  the Hardy--Littlewood--Sobolev (HLS) inequality
\begin{equation}\label{HLS}
\int_{\R^N}(I_{\alpha}*|u|^p)|u|^pdx\leq \mathcal{C}_{\alpha}\|u\|^{2p}_{\frac{2Np}{N+\alpha}}\qquad\forall u\in L^\frac{2Np}{N+\alpha}(\R^N),
\end{equation}
which is valid for any $p\ge 1$ (and $\mathcal{C}_\alpha$ is independent of $p$);
and the Sobolev inequality
\begin{equation}\label{Sobolev}
	\|\nabla u\|_2^2\ge \mathcal{S}_*\|u\|_{2^*}^2\qquad\forall u\in D^1(\R^N),
\end{equation}
where $2^*=\frac{2N}{N-2}$ is the critical Sobolev exponent and $D^1(\R^N)$ denotes the homogeneous Sobolev space with the norm $\|u\|_{D^1(\R^N)}=\|\nabla u\|_{L^2}$.
The values of the sharp constants $\mathcal{S}_*>0$ and $\mathcal{C}_{\alpha}>0$ are known explicitly \citelist{\cite{Lieb}\cite{Lieb-Loss}}.
HLS and Sobolev inequalities can be used to control the nonlocal term in the two cases:

\begin{itemize}
\item
if $\frac{N+\alpha}{N}\le p\le \frac{N+\alpha}{N-2}$ then $L^\frac{2Np}{N+\alpha}(\R^N)\subset L^2\cap L^{2^*}(\R^N)$\smallskip
\item
if $p\ge \frac{N+\alpha}{N}$ and $q\ge \frac{2Np}{N+\alpha}$ then $L^\frac{2Np}{N+\alpha}(\R^N)\subset L^2\cap L^q(\R^N)$
\end{itemize}

\noindent
The two cases have non-empty intersection but this is not significant for us at this moment. In each of these two cases,
$\mathcal I_\eps: H^1(\R^N)\cap L^q(\R^N)\to\R$ is well defined and critical points of $\mathcal{I}_{\eps}$
are solutions of \eqref{eqPeps}.

Our main existence result for \eqref{eqPeps} is the following.
\begin{theorem}\label{thm01}
Let $\frac{N+\alpha}{N}<p<\frac{N+\alpha}{N-2}$ and $q>2$, or $p\geq\frac{N+\alpha}{N-2}$ and $q>\frac{2Np}{N+\alpha}$.
Then for each $\eps>0$, equation \eqref{eqPeps} admits a positive spherically symmetric ground state solution $u_\eps \in H^1\cap L^1\cap C^2(\R^N)$ that is a monotone decreasing function of $|x|$.
Moreover, there exists $C_\eps>0$ such that
\begin{itemize}
	\item if \(p > 2\),
	\[
	\lim_{\abs{x} \to \infty} u_\eps(x) \abs{x}^{\frac{N - 1}{2}} e^{\sqrt{\eps}\abs{x}}=C_\eps,
	\]
	\item if \(p = 2\),
	\[
	\lim_{\abs{x} \to \infty} u_\eps(x) \abs{x}^{\frac{N - 1}{2}} \exp \int_{\nu}^{\abs{x}} \sqrt{\eps - \tfrac{\nu^{N - \alpha}}{s^{N - \alpha}}} \,ds=C_\eps,\quad\text{where $\nu:=\big(A_\alpha\|u_\eps\|_2^2\big)^\frac{1}{N - \alpha} $},
	\]
	\item if \(p < 2\),
	$$\lim_{x\to\infty}u_\eps(x)|x|^\frac{N-\alpha}{2-p}=\left(\eps^{-1}A_\alpha\|u_\eps\|_p^p\right)^{\frac1{2-p}}.$$
\end{itemize}
\end{theorem}

\begin{figure}[t]
	\centering

\tikzset{every picture/.style={line width=0.75pt}} 

\begin{tikzpicture}[x=0.75pt,y=0.75pt,yscale=-0.75,xscale=0.75]
	
	
	
	\tikzset{
		pattern size/.store in=\mcSize,
		pattern size = 5pt,
		pattern thickness/.store in=\mcThickness,
		pattern thickness = 0.3pt,
		pattern radius/.store in=\mcRadius,
		pattern radius = 1pt}
	\makeatletter
	\pgfutil@ifundefined{pgf@pattern@name@_ro2tk9cc0}{
		\pgfdeclarepatternformonly[\mcThickness,\mcSize]{_ro2tk9cc0}
		{\pgfqpoint{0pt}{-\mcThickness}}
		{\pgfpoint{\mcSize}{\mcSize}}
		{\pgfpoint{\mcSize}{\mcSize}}
		{
			\pgfsetcolor{\tikz@pattern@color}
			\pgfsetlinewidth{\mcThickness}
			\pgfpathmoveto{\pgfqpoint{0pt}{\mcSize}}
			\pgfpathlineto{\pgfpoint{\mcSize+\mcThickness}{-\mcThickness}}
			\pgfusepath{stroke}
	}}
	\makeatother
	\tikzset{every picture/.style={line width=0.75pt}} 
	
	
	
	
	\draw  [draw opacity=0][fill={rgb, 255:red, 248; green, 231; blue, 28 }  ,fill opacity=0.5 ][line width=0]  (310,310) -- (150,510) -- (150,470) -- cycle ;
	

	\draw  [draw opacity=0][fill={rgb, 255:red, 208; green, 2; blue, 27 }  ,fill opacity=0.15 ] (310,310) -- (460,310) -- (460,510) -- (310,510) -- cycle ;
	\draw  [draw opacity=0][fill={rgb, 255:red, 208; green, 2; blue, 27 }  ,fill opacity=0.15 ] (110,130) -- (150,130) -- (150,510) -- (110,510) -- cycle ;
	\draw  [draw opacity=0][fill={rgb, 255:red, 208; green, 2; blue, 27 }  ,fill opacity=0.15 ] (460,130) -- (310,310) -- (460,310) -- cycle ;
	\draw  [dash pattern={on 0.84pt off 2.51pt}]  (110,310) -- (460,310) ;
	
	\draw  [draw opacity=0][fill={rgb, 255:red, 126; green, 211; blue, 33 }  ,fill opacity=0.2 ][line width=0]  (310,310) -- (150,510) -- (150,130) -- (310,130)-- cycle ;
	
	\draw  [draw opacity=0][fill={rgb, 255:red, 248; green, 231; blue, 28 }  ,fill opacity=0.5 ][line width=0]  (310,310) -- (150,510) -- (150,470) -- cycle ;
	
	\draw [line width=1.5]  (71.07,510) -- (476.43,510)(110,112.97) -- (110,551.94) (469.43,505) -- (476.43,510) -- (469.43,515) (105,119.97) -- (110,112.97) -- (115,119.97)  ;
	\draw    (360,505) -- (359.92,513.61) ;


	\draw [color={rgb, 255:red, 126; green, 211; blue, 33 }  ,draw opacity=1 ][line width=1.0]  [dash pattern={on 5.63pt off 4.5pt}]  (110,510) -- (460,160) ;
	
	\draw [color={rgb, 255:red, 74; green, 144; blue, 226 }  ,draw opacity=1 ][line width=2.5]    (310,130) -- (310,310) ;

	\draw [color={rgb, 255:red, 208; green, 2; blue, 27 }  ,draw opacity=1 ][fill={rgb, 255:red, 234; green, 178; blue, 85 }  ,fill opacity=1 ][line width=0.75]    (150,130) -- (150,510) ;

	\draw [color={rgb, 255:red, 208; green, 2; blue, 27 }  ,draw opacity=1 ][line width=0.75]  [dash pattern={on 4.5pt off 4.5pt}]  (150,510) -- (310,510) ;

	\draw [color={rgb, 255:red, 208; green, 2; blue, 27 }  ,draw opacity=1 ][line width=0.75]    (310,310) -- (310,510) ;

	\draw [color={rgb, 255:red, 74; green, 144; blue, 226 }  ,draw opacity=1 ][line width=3]    (590,250) -- (470,250) ;

	\draw [color={rgb, 255:red, 144; green, 19; blue, 254 }  ,draw opacity=1 ][line width=3]    (470,370) -- (590,370) ;
	
	\draw [color={rgb, 255:red, 126; green, 211; blue, 33 }  ,draw opacity=1 ][line width=2.5]    (150,470) -- (310,310) ;

	\draw [color={rgb, 255:red, 126; green, 211; blue, 33 }  ,draw opacity=1 ][line width=3]    (590,310) -- (470,310) ;
	
	\draw [color={rgb, 255:red, 189; green, 16; blue, 224 }  ,draw opacity=1 ][line width=2.5]    (150,510) -- (310,310) ;

	\draw [color={rgb, 255:red, 208; green, 2; blue, 27 }  ,draw opacity=1 ][line width=0.75]    (460,130) -- (310,310) ;
	\draw [shift={(310,310)}, rotate = 129.81] [color={rgb, 255:red, 208; green, 2; blue, 27 }  ,draw opacity=1 ][fill={rgb, 255:red, 208; green, 2; blue, 27 }  ,fill opacity=1 ][line width=0.75]      (0, 0) circle [x radius= 3.35, y radius= 3.35]   ;
	

	\draw (145,525) node [scale=0.9]  {$\frac{N+\alpha }{N}$};
	\draw (102.5,111) node [scale=0.9]  {$q$};
	\draw (467.5,519) node [scale=0.9]  {$p$};
	\draw (117.5,519) node [scale=0.9]  {$1$};
	\draw (102.5,499) node [scale=0.9]  {$2$};
	\draw (102,300) node [scale=0.9]  {$2^{*}$};
	\draw (305,525) node [scale=0.9]  {$\frac{N+\alpha }{N-2}$};
	\draw (362,520) node [scale=0.9]  {$2^{*}$};
	\draw  [color={rgb, 255:red, 0; green, 0; blue, 0 }  ,draw opacity=1 ][fill={rgb, 255:red, 74; green, 144; blue, 226 }  ,fill opacity=1 ][line width=0.75]   (250.5, 468.5) circle [x radius= 18.44, y radius= 18.26]   ;
	\draw (250.5,468.5) node [scale=0.9]  {$P_{0}$};
	\draw  [fill={rgb, 255:red, 74; green, 144; blue, 226 }  ,fill opacity=1 ]  (350.5, 201.5) circle [x radius= 18.44, y radius= 18.26]   ;
	\draw (350.5,201.5) node [scale=0.9]  {$P_{0}$};
	\draw  [fill={rgb, 255:red, 126; green, 211; blue, 33 }  ,fill opacity=1 ]  (225.5, 311) circle [x radius= 55.96, y radius= 18.73]   ;
	\draw (225.5,311) node [scale=0.9]  {$Choquard$};
	\draw (499,155) node [scale=0.9]  {$q=2\frac{2p+\alpha }{2+\alpha}$};
	\draw (495,125) node [scale=0.9]  {$q=\frac{2Np}{N+\alpha}$};
	\draw (523.5,358.5) node  [align=left] {{\small \textit{Critical TF}}};
	\draw (528,238.5) node  [align=left] {{\small \textit{Critical Choquard}}};
	\draw (527,298) node [scale=0.9] [align=left] {\textit{Selfrescaling}};
	
	\draw (527,448) node [scale=0.9] [align=left] {\huge{$\boxed{\eps\to 0}$}};

	\draw  [fill={rgb, 255:red, 248; green, 231; blue, 28 }  ,fill opacity=1 ][line width=0.75]   (168, 469) circle [x radius= 14.26, y radius= 11.73]   ;
	\draw (168,469) node [scale=0.9]  {$T\!F$};


\end{tikzpicture}

\caption{Six limit regimes for $(P_\eps)$ as $\eps\to 0$ on the $(p,q)$--plane} \label{fig:M1}
\end{figure}

The existence range of Theorem \ref{thm01} is optimal.
This follows from the Poho\v zaev identity argument, see Corollary \ref{c-non}.
We emphasise that no upper restrictions on $p$ and $q$ are needed
and in particular, $q$ could take Sobolev supercritical values, i.e. $q>2^*$ (see Figure \ref{fig:M1}).
The decay rates of ground states at infinity are exactly the same as in the standard Choquard case, compare Theorem \ref{thmC} below or \cite{MS13}*{Theorem 4}. For a discussion of the implicit exponential decay in the case $p=2$ we refer to \cite{MS13}*{pp.157-158} or \cite{MVS-JDE}*{Section 6.1}.
\smallskip

Our main goal in this paper is to understand and classify the asymptotic profiles as $\eps\to 0$ and $\eps\to\infty$ of the groundstates $u_\eps$, constructed in Theorem \ref{thm01}. Remarkably, our study uncovers a novel and rather complicated limit structure of the problem, with six different limit equations (see Figure \ref{fig:M1}) as $\eps\to 0$:

\begin{itemize}[leftmargin=2em]
	\item {\em Formal limit} when the family of ground states $u_\eps$ converges to a groundstate of the {\em formal} limit equation
	\begin{equation}\tag{$P_0$}\label{eqP0}
	-\Delta u -(I_{\alpha}*|u|^p)|u|^{p-2}u+|u|^{q-2}u=0\quad\text{in $\R^N$}.
	\end{equation}
	The existence and qualitative properties of groundstate for \eqref{eqP0} for the optimal range of parameters is new and is studied in Section \ref{s5}, see Theorem \ref{thmP0}. The convergence of the groundstates to the limit profile is proved in Theorem \ref{thm02-lim}.
	\item {\em Choquard limit} when the rescaled family $$v_\eps(x):=\eps^{-\frac{2+\alpha}{4(p-1)}}u_\eps\big(\eps^{-\frac12}x\big)$$ converges to a groundstate of the standard Choquard equation
	\begin{equation*}\tag{$\mathscr{C}$}
	-\Delta v+v-(I_{\alpha}*|v|^p)|v|^{p-2}v=0\quad\text{in $\R^N$},
	\end{equation*}
	which was studied in \cite{MS13}. The convergence is proved in Theorem \ref{thmC-lim}.
	\item
	{\em Thomas--Fermi limit} when the rescaled family $$v_\eps(x):=\eps^{-\frac{1}{q-2}}u_\eps\big(\eps^{-\frac{4-q}{\alpha(q-2)}} x\big)$$
	converges to a groundstate of the {\em Thomas--Fermi} type integral equation
	\begin{equation*}\tag{$T\!F$}\label{eqTF}
	v-(I_{\alpha}*|v|^p)|v|^{p-2}v+|v|^{q-2}v=0\quad\text{in $\R^N$}.
	\end{equation*}
	The existence and qualitative properties of groundstate for \eqref{eqTF} for $p\neq 2$ will be studied in the forthcoming work \cite{TF}. In this paper we consider only the case $p=2$ which is well known in the literature when $\alpha=2$ \citelist{\cite{Auchmuty-Beals}\cite{Lions-81}\cite{Brezis-Benilan}} and was studied recently in \citelist{\cite{Carrillo-CalcVar}\cite{Carrillo-NA}} for the general $\alpha\in(0,N)$, yet for the range of powers $q$ which is incompatible with our assumptions. In Theorem \ref{thmTF} we prove the existence and some qualitative properties of a groundstate for \eqref{eqTF} with $p=2$ for the optimal range $q>\frac{4N}{N+\alpha}$. This extends some of the existence results in \citelist{\cite{Carrillo-CalcVar}\cite{Carrillo-NA}}.
	The convergence of $v_\eps$ to a groundstate of \eqref{eqTF} is proved in Theorem \ref{t-TF-0} for $p=2$ and $\alpha=2$ (the general case $p\neq 2$ and $\alpha\neq 2$ will be studied in \cite{TF}). Remarkably, for $p=2$ the limit groundstates of \eqref{eqTF} are compactly supported, so the rescaled groundstates $v_\eps$ develops a steep ``corner layer'' as $\eps\to 0$!

	\item
	{\em Critical Choquard regime}, when the family of ground states $u_\eps$ converges after an {\em implicit rescaling}
	$$v_\eps(x):=\lambda_\eps^\frac{N-2}{2}u_\eps(\lambda_\eps x)$$
	to a groundstate of the {\em critical Choquard equation}
	\begin{equation*}\tag{$\mathscr{C}_{HL}$}\label{eqC0}
	-\Delta v=(I_{\alpha}*|v|^\frac{N+\alpha}{N-2})|v|^{\frac{N+\alpha}{N-2}-2}v,\qquad v\in D^1(\R^N).
	\end{equation*}
	A detailed characterisation of the ground states of \eqref{eqC0} was recently obtained in \citelist{\cite{Chen}\cite{Minbo-Du}}. In Theorem \ref{t-CC-0} we derive a sharp two--sided asymptotic characterisation of the rescaling $\lambda_\eps$, following the ideas developed in the local case in \cite{MM-14}.

	\item
	{\em Self--similar regime} $q=2\frac{2p+\alpha}{2+\alpha}$,  when ground states $u_\eps$ are obtained as rescalings of the groundstate $u_1$, i.e.
	$$u_1(x)=\eps^{-\frac{2+\alpha}{4(p-1)}}u_\eps\big(\eps^{-\frac12}x\big)$$
	\item
	{\em the Critical Thomas--Fermi regime}, when the family of ground states $u_\eps$ converges  after an {\em implicit rescaling}
	$$v_\eps(x):=\lambda_\eps^\frac{N+\alpha}{2p}u_\eps(\lambda_\eps x)$$
	to a groundstate of the {\em critical Thomas--Fermi equation}
	\begin{equation}\tag{$T\!F_*$}\label{eqTF0}
	|v|^{q-2}v=(I_{\alpha}*|v|^p)|v|^{p-2}v,\qquad v\in L^q(\R^N).
	\end{equation}
	Groundstates of this equation correspond to the minimizers of the Hardy--Littlewood--Sobolev inequality and completely characterised by Lieb in \cite{Lieb}.
	In Theorem \ref{t-HLS-0} we derive a two--sided asymptotic characterisation of the rescaling $\lambda_\eps$.
\end{itemize}
Self--similar, Thomas--Fermi and critical Thomas--Fermi regimes are specific to the nonlocal case only. When $\alpha=0$ they all ``collapse'' into the case $p=q$, which is degenerate for the local prototype equation \eqref{eq01-loc}.
Three other regimes could be traced back to the local equation \eqref{eq01-loc} studied in \cite{MM-14}.

When $\eps\to\infty$ the limit structure is simpler. Only Choquard, Thomas--Fermi and self--similar regimes are relevant (see Figure \ref{fig:M2}) and there are no critical regimes. In particular, the Thomas--Fermi limit with $\eps\to \infty$ appears in the study of the stationary Gross--Pitaevskii--Poisson equation \eqref{GPP}, see Remark \ref{rGPP}.
\smallskip

The precise statements of our results for $\eps\to 0$ are given in Section \ref{s2}.
In Section \ref{s3} we outline the results for $\eps\to\infty$ and discuss the connection with astrophysical models of self--gravitating Bose--Einstein condensate.
In Section \ref{s4} we prove Theorem \ref{thm01}.
In Sections \ref{s5} and \ref{sTF} we establish the existence and basic properties of groundstates for the ``zero--mass'' limit equation $(P_0)$ and for the Thomas--Fermi equation \eqref{eqTF}.
In Sections \ref{s6} and \ref{s7} we study the asymptotic profiles of the groundstates of \eqref{eqPeps} in the {\em non-critical} and {\em critical} regimes respectively. Finally, in the Appendix we discuss a contraction inequality which was communicated to us by Augusto Ponce and which
we used as a key tool in several regularity proofs.

\subsubsection*{\bf Asymptotic notations}
For real valued functions $f(t), g(t) \geq 0$ defined on a subset of $\R_+$, we write:
\smallskip

$f(t)\lesssim g(t)$ if there exists $C>0$ independent of $t$
such that $f(t) \le C g(t)$;

$f(t)\gtrsim g(t)$ if $g(t)\lesssim f(t)$;

$f(t)\sim g(t)$ if $f(t)\lesssim g(t)$ and $f(t)\gtrsim g(t)$;

$f(t)\simeq g(t)$ if $f(t) \sim g(t)$ and $\lim_{t\to 	0}\frac{f(t)}{g(t)}=1$.

\noindent
Bearing in mind that $f(t),g(t) \geq 0$, we write $f(t)=\O(g(t))$ if $f(t)\sim g(t)$, and $f(t)=o(g(t))$ if $\lim\frac{f(t)}{g(t)}=0$.
As usual, $B_R=\{x\in\R^N:|x|<R\}$ and $C,c,c_1$ etc., denote generic positive constants.

\section{Asymptotic profiles as $\eps\to 0$}\label{s2}
Our main goal in this work is to understand the asymptotic behaviour of the constructed in Theorem \ref{thm01} groundstate solutions $u_\eps$ of \eqref{eqPeps} in the limits $\eps\to 0$ and $\eps\to\infty$.

\subsection{Formal limit $(P_0)$}
Loosely speaking, the elliptic regularity implies that $u_\eps$ converges as $\eps\to 0$ to a nonnegative radial solution of the {\em formal} limit equation
\begin{equation}\tag{$P_0$}
-\Delta u -(I_{\alpha}*|u|^p)|u|^{p-2}u+|u|^{q-2}u=0\quad\text{in $\R^N$}.
\end{equation}
However, \eqref{eqP0} becomes a meaningful limit equation for \eqref{eqPeps} only in the situation when \eqref{eqP0} admits a nontrivial nonnegative solution.
Otherwise the information that $u_\eps$ converges to zero (trivial solution of \eqref{eqP0}) does not reveal any information about the limit profile of $u_\eps$. We prove in this work the following existence result for \eqref{eqP0}.

\begin{theorem}\label{thmP0}
	Let $\frac{N+\alpha}{N}<p<\frac{N+\alpha}{N-2}$ and $2<q<\frac{2Np}{N+\alpha}$, or $p>\frac{N+\alpha}{N-2}$ and $q>\frac{2Np}{N+\alpha}$.
	Then equation \eqref{eqP0} admits a positive spherically symmetric groundstate solution $u_0 \in D^1\cap L^q\cap C^2(\R^N)$
	which is a monotone decreasing function of $|x|$.
	Moreover,
	\begin{itemize}
		\item if \(p < \frac{N+\alpha}{N-2}\) then $u\in L^1(\R^N)$,
		\item if \(p >  \frac{N+\alpha}{N-2}\) then
		$$u_0\gtrsim|x|^{-(N-2)}\quad\text{as $|x|\to\infty$},$$
		and if $p>\max\big\{\frac{N+\alpha}{N-2},\frac23\big(1+\frac{N+\alpha}{N-2}\big)\big\}$ then
		$$u_0\sim|x|^{-(N-2)}\quad\text{as $|x|\to\infty$}.$$
	\end{itemize}
\end{theorem}

The restrictions on $p$ and $q$ in the existence part of the theorem ensures that the energy $\mathcal{I}_0$  which corresponds to \eqref{eqP0} is well-defined on $D^1(\R^N)\cap L^q(\R^N)$, see \eqref{eq15} below.
A Poho\v zaev identity argument (see Remark \ref{r-Phz0}) confirms that the existence range in Theorem \ref{thmP0} is optimal,
with the exception of the ``double critical'' point $p=\frac{N+\alpha}{N-2}$ and $q=\frac{2Np}{N+\alpha}=2^*$ on the $(p,q)$--plane (see Remark \ref{r-2crit}).

Note that
$$\max\Big\{\tfrac{N+\alpha}{N-2},\tfrac23\big(1+\tfrac{N+\alpha}{N-2}\big)\Big\}=
\left\{
\begin{array}{cl}
\frac23\big(1+\frac{N+\alpha}{N-2}\big)&\text{if $\frac{N+\alpha}{N-2}<2$},\medskip\\
\frac{N+\alpha}{N-2}&\text{if $\frac{N+\alpha}{N-2}\ge 2$}.
\end{array}
\right.$$
The upper bound on $u_0$ in the case $\frac{N+\alpha}{N-2}<2$ and $\frac{N+\alpha}{N-2}<p\le\frac23\big(1+\frac{N+\alpha}{N-2}\big)$ remains open. We conjecture that our restriction on $p$ for the upper decay estimate is technical and that
$u_0\sim|x|^{-(N-2)}$ as $|x|\to\infty$ for all $p>\frac{N+\alpha}{N-2}$.

Observe that the energy $\I_0$ is well posed in the space $D^1(\R^N)\cap L^q(\R^N)$, while  $\I_\eps$ with $\eps>0$ is well--posed on $H^1(\R^N)\cap L^q(\R^N)$. Since $H^1(\R^N)\subsetneq D^1(\R^N)$, small perturbation arguments in the spirit of the Lyapunov--Schmidt reduction are not directly applicable to the family $\I_\eps$ in the limit $\eps\to 0$. Using direct variational analysis based on the comparison of the groundstate energy levels for two problems, we establish the following result.

\begin{theorem}[Formal limit \eqref{eqP0}]\label{thm02-lim}
	Let $\frac{N+\alpha}{N}<p<\frac{N+\alpha}{N-2}$ and $2<q<\frac{2Np}{N+\alpha}$, or
	$p>\max\big\{\frac{N+\alpha}{N-2},\frac23\big(1+\frac{N+\alpha}{N-2}\big)\big\}$
	and $q>\frac{2Np}{N+\alpha}$.
	
	Then as $\eps\to 0$, the family of ground states $u_\eps$ of $(P_\eps)$ converges in $D^1(\R^N)$ and $L^q(\R^N)$ to a positive spherically symmetric ground state solution $u_0\in D^1\cap L^q(\R^N)$ of the formal limit equation \eqref{eqP0}.
	Moreover, $\eps\|u_\eps\|_2^2\to 0$.
\end{theorem}

The restriction 	$p>\max\big\{\frac{N+\alpha}{N-2},\frac23\big(1+\frac{N+\alpha}{N-2}\big)\big\}$ is related only to the upper decay bound on $u_0$ in Theorem \ref{thmP0}, i.e. we could establish the convergence of $u_\eps$ to $u_0$ for $p>\frac{N+\alpha}{N-2}$ as soon as we know that $u_0\sim|x|^{-(N-2)}$ as $|x|\to\infty$.

\subsection{Rescaled limits}
When $\frac{N+\alpha}{N}<p\le\frac{N+\alpha}{N-2}$ and $q\ge\frac{2Np}{N+\alpha}$ the formal limit problem \eqref{eqP0} has no nontrivial sufficiently regular finite energy solutions (see Remark \ref{r-2crit}). As a consequence, $u_\eps$ converges uniformly on compact sets to zero. We are going to show that in these regimes $u_\eps$ converges to a positive solution of a limit equation {\em after a rescaling}
$$v(x):=\eps^s u(\eps^{t}x),$$
for specific choices of $s,t\in\R$. The rescaling transforms \eqref{eqPeps} into the equation
\begin{equation}\label{eq01st}
-\eps^{-s-2t}\Delta v+\eps^{1-s}v-\eps^{-(2p-1)s+\alpha t}(I_{\alpha}*|v|^p)|v|^{p-2}v+\eps^{-(q-1)s}|v|^{q-2}v=0\quad\text{in $\R^N$}.
\end{equation}
If $q\neq 2\frac{2p+\alpha}{2+\alpha}$ then there are three natural possibilities to choose $s$ and $t$, each achieving the balance of {\em three} different terms in \eqref{eq01st}. Note that the choice $\eps^{-s-2t}=\eps^{-(2p-1)s+\alpha t}=\eps^{-(q-1)s}$ leads to $s=t=0$, when \eqref{eq01st} reduces to the original equation $(P_\eps)$.

\smallskip
\paragraph{{\sl I. First rescaling.}}
The choice $\eps^{-s-2t}=\eps^{1-s}=\eps^{-(2p-1)s+\alpha t}$ leads to $s=-\frac{2+\alpha}{4(p-1)}$, $t=-1/2$ and rescaled equation
\begin{equation*}\tag{$\mathscr{C}_\eps$}\label{eqCeps}
-\Delta v+v-(I_{\alpha}*|v|^p)|v|^{p-2}v+\eps^{\frac{q(2+\alpha)-2(2p+\alpha)}{4(p-1)}}|v|^{q-2}v=0\quad\text{in $\R^N$}.
\end{equation*}
If $q>2\frac{2p+\alpha}{2+\alpha}$, we have $\lim_{\eps\to 0}\eps^{\frac{q(2+\alpha)-2(2p+\alpha)}{4(p-1)}}=0$,
and we obtain as a formal limit the {\em Choquard equation}
\begin{equation*}\tag{$\mathscr{C}$}
-\Delta v+v-(I_{\alpha}*|v|^p)|v|^{p-2}v=0\quad\text{in $\R^N$}.
\end{equation*}

\smallskip
\paragraph{{\sl II. Second rescaling.}}
The choice $\eps^{1-s}=\eps^{-(2p-1)s+\alpha t}=\eps^{-(q-1)s}$ leads to $s=-\frac{1}{q-2}$, $t=-\frac{2p-q}{\alpha(q-2)}$ and rescaled equation
\begin{equation*}\tag{$T\!F_\eps$}\label{eqTFeps}
\eps^{\frac{2(2p+\alpha)-q(2+\alpha)}{\alpha(q-2)}}(-\Delta) v+v-(I_{\alpha}*|v|^p)|v|^{p-2}v+|v|^{q-2}v=0\quad\text{in $\R^N$}.
\end{equation*}
If $2<q<2\frac{2p+\alpha}{2+\alpha}$ we have $\lim_{\eps\to 0}\eps^{\frac{2(2p+\alpha)-q(2+\alpha)}{\alpha(q-2)}}=0$
and we obtain as a formal limit the {\em Thomas--Fermi} type integral equation
\begin{equation*}\tag{$T\!F$}
v-(I_{\alpha}*|v|^p)|v|^{p-2}v+|v|^{q-2}v=0\quad\text{in $\R^N$}.
\end{equation*}

\smallskip
\paragraph{{\sl III. Third rescaling}}
The choice $\eps^{-s-2t}=\eps^{1-s}=\eps^{-(q-1)s}$ leads to $s=-\frac{1}{q-2}$, $t=-\frac{1}{2}$ and rescaled equation
\begin{equation*}\label{eq01iii}
-\Delta v+v-\eps^{\frac{2(2p+\alpha)-q(\alpha+2)}{2(q-2)}}(I_{\alpha}*|v|^p)|v|^{p-2}v+|v|^{q-2}v=0\quad\text{in $\R^N$}.
\end{equation*}
If $2<q<2\frac{2p+\alpha}{2+\alpha}$ we have $\lim_{\eps\to 0}\eps^{\frac{2(2p+\alpha)-q(\alpha+2)}{2(q-2)}}=0$,
and we obtain as a formal limit the nonlinear local equation
\begin{equation*}\label{eq01i+}
-\Delta v+v+|v|^{q-2}v=0\quad\text{in $\R^N$}.
\end{equation*}
Such equation has no nonzero finite energy solutions and we rule out the third rescaling as {\em trivial}.

\subsection{Self--similar regime $q=2\frac{2p+\alpha}{2+\alpha}$}
In this special case the choice $s=-\frac{2+\alpha}{4(p-1)}$ and $t=-1/2$ leads to the balance of all {\em four} terms in \eqref{eq01st},
i.e. $\eps^{-s-2t}=\eps^{1-s}=\eps^{-(2p-1)s+\alpha t}=\eps^{-(q-1)s}$. The rescaled equation in this case becomes
\begin{equation*}\label{eq01-0}
-\Delta v+v-(I_{\alpha}*|v|^p)|v|^{p-2}v+|v|^{q-2}v=0\quad\text{in $\R^N$},\eqno{(P_1)}
\end{equation*}
and any solution of $(P_\eps)$ is a rescaling of a solution of $(P_1)$.
For completeness, we include this obvious observation in the following statement.

\begin{theorem}[Self--similar regime]\label{thm02-0}
	Let $\frac{N+\alpha}{N}<p<\frac{N+\alpha}{N-2}$ and $q=2\frac{2p+\alpha}{2+\alpha}$.
	Then for any $\eps>0$,
	$$u_\eps(x)=\eps^{\frac{2+\alpha}{4(p-1)}} u_1(\sqrt{\eps}x),$$
	where $u_1$ is a ground state solution of $(P_1)$.
\end{theorem}

\subsection{First rescaling: Choquard limit}

The following result describes the existence region and some qualitative properties of the groundstates of \eqref{eqC}.

\begin{theorem}[\cite{MS13}*{Theorem 4}]\label{thmC}
	Let $\frac{N+\alpha}{N}<p<\frac{N+\alpha}{N-2}$.
	Then Choquard's equation \eqref{eqC}
	admits a positive spherically symmetric ground state solution $v\in H^1\cap L^1\cap C^2(\R^N)$ such that $v(x)$ is a monotone decreasing function of $|x|$.
	Moreover, there exists $C>0$ such that
	\begin{itemize}
		\item if \(p > 2\),
		\[
		\lim_{\abs{x} \to \infty} v(x) \abs{x}^{\frac{N - 1}{2}} e^{\abs{x}}=C,
		\]
		\item if \(p = 2\),
		\[
		\lim_{\abs{x} \to \infty} v(x) \abs{x}^{\frac{N - 1}{2}} \exp \int_{\nu}^{\abs{x}} \sqrt{1 - \tfrac{\nu^{N - \alpha}}{s^{N - \alpha}}} \,ds=C,\quad\text{where $\nu:=\big(A_\alpha\|v\|_2^2\big)^\frac{1}{N - \alpha} $},
		\]
		\item if \(p < 2\),
		$$\lim_{|x|\to\infty}v(x)|x|^\frac{N-\alpha}{2-p}=\left(A_\alpha\|v\|_p^p\right)^{\frac1{2-p}}.$$
	\end{itemize}
\end{theorem}

The existence interval in this theorem is sharp, in the sense that \eqref{eqC} does not have finite energy solutions for $p\not\in\Big(\frac{N+\alpha}{N},\frac{N+\alpha}{N-2}\Big)$. The uniqueness of the ground state solution is known only for $N=3$, $p=2$ and $\alpha=2$ \cite{Lieb77} and several other special cases \citelist{\cite{MaZhao}\cite{Xiang}\cite{Seok}}.

In this paper we prove that after the 1st rescaling, groundstates of \eqref{eqPeps} converge to the groundstates of the Choquard equation \eqref{eqC}, as soon as \eqref{eqC} admits a nontrivial groundstate.

\begin{theorem}[Choquard limit]\label{thmC-lim}
	Let $\frac{N+\alpha}{N}<p<\frac{N+\alpha}{N-2}$ and $q>2\frac{2p+\alpha}{2+\alpha}$.
	As $\eps\to 0$, the rescaled family of ground states
	\begin{equation}\label{eq1res}
	v_\eps(x):=\eps^{-\frac{2+\alpha}{4(p-1)}}u_\eps(\tfrac{x}{\sqrt{\eps}})
	\end{equation}
	converges in $D^1(\R^N)$ and $L^q(\R^N)$
	to a positive spherically symmetric ground state solution $v_0\in D^1\cap L^q(\R^N)$ of the Choquard equation \eqref{eqC}.
\end{theorem}

\subsection{2nd rescaling: Thomas--Fermi limit for $p=2$}
In this paper we consider the 2nd rescaling regime only in the case $p=2$. The general case $p\neq 2$ is studied in the forthcoming work \cite{TF}.

When $p=2$ the formal limit equation for \eqref{eqPeps} in the 2nd rescaling is the Thomas--Fermi type integral equation
\begin{equation*}\tag{$T\!F$}\label{eqTF2}
v-(I_{\alpha}*|v|^2)v+|v|^{q-2}v=0\quad\text{in $\R^N$}.
\end{equation*}
One of the possible ways to write the variational problem that leads to \eqref{eqTF2} after a rescaling is
\begin{equation}\label{eq-sTF}
s_{T\!F}:=\inf\big\{\|v\|_q^q+\|v\|_2^2: 0\le v\in L^2\cap L^{q}(\R^N), \int_{\R^N}(I_{\alpha}*v^2)v^2 dx=1\big\}.
\end{equation}
By a {\em groundstate} of \eqref{eqTF2} we understand a rescaling of a nonnegative minimizer for $s_{T\!F}$ that satisfies the limit equation \eqref{eqTF2}.

To study nonnegative minimizers of the $s_{T\!F}$ it is convenient to substitute
$$\rho:=v^2,\qquad m:=\frac{q}{2},$$
for an equivalent representation
$$
s_{T\!F}=\inf\Big\{\int_{\R^N}\rho^{m}dx+\int_{\R^N}\rho\, dx: 0\le \rho\in L^1\cap L^{m}(\R^N), \int_{\R^N}(I_{\alpha}*\rho)\rho\, dx=1\Big\}.
$$
For $m>m_c:=2-\alpha/N$ it is not difficult to see that, after a rescaling, minimizers for $s_{T\!F}$ are in the one-to-one correspondence with the minimizers of
$$
\sigma_{T\!F}:=\inf\Big\{\int_{\R^N}|\rho|^m\,dx-\int_{\R^N}(I_{\alpha}*\rho)\rho\, dx: 0\le \rho\in L^1\cap L^m(\R^N), \|\rho\|_1=1\Big\}.
$$
The existence and qualitative properties of minimizers for $\sigma_{T\!F}$ in the case $N=3$, $\alpha=2$ and for $m>4/3$ is classical and goes back to \citelist{\cite{Auchmuty-Beals}\cite{Lions-81}}. The case $N\ge 2$, $\alpha\in(0,N)$ and $m>m_c$ it is a recent study by J.~ Carrillo et al.~\cite{Carrillo-CalcVar}.
If $m<m_c$ then $\sigma_{T\!F}=-\infty$ by scaling, while $m=m_c$ is the $L^1$--critical exponent for $\sigma_{T\!F}$ (this case is studied in \cite{Carrillo-NA}). Note that $m_c>\frac{4+\alpha}{2+\alpha}$ so in the 2nd rescaling regime we always have $\sigma_{T\!F}=-\infty$  when $\eps\to 0$!

In the next theorem we show that, unlike for $\sigma_{T\!F}$, minimization for $s_{T\!F}$ is possible for all $m>\frac{2N}{N+\alpha}$. The existence and qualitative properties of the minimizers are summarised below.

\begin{theorem}[Thomas--Fermi groundstate]\label{thmTF}
	Let $m>\frac{2N}{N+\alpha}$.
	Then $s_{T\!F}>0$ and there exists a nonnegative spherically symmetric nonincreasing minimizer $\rho_*\in L^1\cap L^\infty(\R^N)$ for $s_{T\!F}$.
	The minimizer $\rho_*$ satisfies the virial identity
	\begin{equation}\label{TF-Virial}
	m\int_{\R^N}\rho^m dx+\int_{\R^N}\rho\, dx=s_{T\!F}\frac{2N}{N+\alpha}
	\end{equation}
	and the Thomas--Fermi equation
	\begin{equation}\label{TF-rho}
	m\rho^{m-1}=\Big(s_{T\!F}\frac{2N}{N+\alpha} I_{\alpha}*\rho-1\Big)_+\quad\text{a.e. in $\R^N$}.
	\end{equation}
	Moreover, $\mathrm{supp}(\rho_*)=\bar B_{R_*}$ for some $R_*\in(0,+\infty)$,
	$\rho_*$ is $C^\infty$ inside the support, and $\rho_*\in C^{0,\gamma}(\R^N)$ with $\gamma=\min\{1,\frac{1}{m-1}\}$ if $\alpha>\big(\frac{m-2}{m-1})_+$, or for any $\gamma<\frac{\alpha}{m-2}$ if $m>2$ and $\alpha\le\big(\frac{m-2}{m-1})_+$. Moreover, if $\alpha>\big(\frac{m-2}{m-1})_+$ then $I_\alpha*\rho_*\in C^{0,1}(\R^N)$ and $\rho_*^{m-1}\in C^{0,1}(\R^N)$. Meanwhile,
\begin{equation}\label{TF-resc-min-rho-1}
		v_0(x)=\Big(\frac{q}{2}\Big)^\frac{1}{q-2}\sqrt{\rho_*\Big((\tfrac{q}{2})^\frac{2}{\alpha(q-2)}\big(\tfrac{2N s_{T\!F}}{N+\alpha}\big)^{-1/\alpha} x\Big)},
	\end{equation}
is a nonnegative spherically symmetric nonincreasing ground state solution of the Thomas--Fermi  equation \eqref{eqTF2}.
\end{theorem}

Only the existence part of the theorem requires a proof. The Euler--Lagrange equation, regularity and qualitative properties of the minimizers could be obtained by adaptations of the arguments developed for $m>m_c$ in \citelist{\cite{Carrillo-CalcVar}\cite{Carrillo-NA}}. We outline the arguments in Section \ref{s-TF}.


In the case $m\ge m_c$ the uniqueness of the minimizer for $\sigma_{T\!F}$ was recently proved in \cite{Carrillo-NA-2021} for $\alpha<2$, see also \cite{Carrillo-Sugiyama} for $\alpha=2$ and a survey of earlier results in this direction. 
For the full range $m>\frac{2N}{N+\alpha}$ and for $\alpha<2$ the uniqueness of a bounded radially nonincreasing solution for the Euler--Lagrange equation \eqref{TF-rho} (and hence the uniqueness of the minimizer $\rho_*$ for $s_{T\!F}$) is the recent result in \cite[Theorem 1.1 and Proposition 5.4]{Volzone}. For $\alpha=2$ the same follows from \cite[Lemma 5]{Flucher}. For $\alpha\in(2,N)$ the uniqueness of the minimizer for $s_{T\!F}$ or for $\sigma_{T\!F}$ seems to be open at present.
\smallskip

Next we prove that in the special case $\alpha=2$, groundstates of \eqref{eqPeps} converge to a groundstate of the Thomas--Fermi equation \eqref{eqTF2}.

\begin{theorem}[Thomas--Fermi limit for $\alpha=2$]\label{t-TF-0}
	Let $N\le 5$, $p=2$, $\alpha=2$ and $\frac{4N}{N+2}<q<3$.
	As $\eps\to 0$, the rescaled family of ground states
	\begin{equation}\label{eq2res}
		v_\eps(x):=\eps^{-\frac{1}{q-2}}u_\eps\big(\eps^{-\frac{4-q}{2(q-2)}} x\big)
	\end{equation}
	converges in $L^2(\R^N)$ and $L^q(\R^N)$ to a nonnegative spherically symmetric compactly supported ground state solution $v_0\in L^2\cap L^q(\R^N)$ of the Thomas-Fermi equation \eqref{eqTF2}. 
\end{theorem} 

\begin{remark}
	While the uniqueness of the Thomas--Fermi groundstate $v_0$ for $\alpha>2$ is generally open, it is clear from the proof of Theorem \ref{thmTF} that every ground state of \eqref{eqTF2} must have the same regularity and compact support properties as stated in Theorem \ref{thmTF}.
In particular, $v_\eps$ always exhibits as $\eps\to 0$ a ``corner layer'' near the boundary of the support of the limit groundstate of \eqref{eqTF2}.
\end{remark}

\begin{remark}
The case $p\neq 2$ and $\alpha\neq 2$ is studied in our forthcoming paper \cite{TF}.
We are going to show that the minimization problem
$$
s_{T\!F}:=\inf\big\{\|v\|_q^q+\|v\|_2^2: 0\le v\in L^2\cap L^{q}(\R^N), \int_{\R^N}(I_{\alpha}*v^p)v^p dx=1\big\}.
$$
admits a nonnegative spherically symmetric nonincreasing minimizer $v_*\in L^1\cap L^\infty(\R^N)$ for any $p>\frac{N+\alpha}{N}$ and $q>\frac{2Np}{N+\alpha}$ and this range is optimal. Moreover,
\begin{itemize}
	\item[$(a)$]
	if $p\ge 2$ then $\mathrm{Supp}(\overline{v}_*)=B_{R_*}$ and $\overline{v}_*=\lambda\chi_{B_{R_*}}+\phi$, where $R_*>0$, $\lambda>0$ if $p>2$ or $\lambda=0$ if $p=2$, and $\phi:B_R\to\R$ is H\"older continuous radially nonincreasing, $\phi(0)>0$ and $\lim_{|x|\to R_*}\phi(|x|)=0$;
	\smallskip
	\item[$(b)$]
	if $p<2$ then $\overline{v}_*\in D^1(\R^N)$ and $\mathrm{Supp}(\overline{v}_*)=\R^N$.
\end{itemize}
We will prove in \cite{TF} that such a minimizer is the limit of the rescaled groundstates $v_\eps(x)$ as $\eps\to 0$ in the Thomas--Fermi regime.
\end{remark}

\subsection{Critical Choquard regime $p=\frac{N+\alpha}{N-2}$}
When $p=\frac{N+\alpha}{N-2}$ and $q>\frac{2Np}{N+\alpha}=2^*$, neither \eqref{eqC} nor $(P_0)$ have nontrivial solutions. We prove that in this case the limit equation for \eqref{eqPeps} is given by the critical Choquard equation
\begin{equation*}\tag{$\mathscr{C}_{HL}$}
-\Delta v=(I_{\alpha}*|v|^\frac{N+\alpha}{N-2})|v|^{\frac{N+\alpha}{N-2}-2}v,\qquad v\in D^1(\R^N).
\end{equation*}
A variational problem that leads to \eqref{eqC0} can be written as
\begin{equation}\label{eSHL}
\mathcal{S}_{HL}=\inf_{w\in D^{1}(\R^N)\setminus \{0\}}\frac{\|\nabla w\|_2^2}{\Big\{ \int_{\R^N}(I_{\alpha}*|w|^{\frac{N+\alpha}{N-2}})|w|^{\frac{N+\alpha}{N-2}}dx\Big\}^{\frac{N-2}{N+\alpha}}}.
\end{equation}
It is known \cite{Minbo-Du}*{Lemma 1.1} that
$$\mathcal{S}_{HL}=\mathcal{S}_*\mathcal{C}_\alpha^{-\frac{N-2}{N+\alpha}}$$
where $\mathcal{S}_*$ is the Sobolev constant in \eqref{Sobolev} and $\mathcal{C}_\alpha$ is the Hardy--Littlewood--Sobolv constant in \eqref{HLS} (with $p=\frac{N+\alpha}{N-2}$).

By a {\em groundstate} of \eqref{eqC0} we understand a rescaling of a positive minimizer for $\mathcal{S}_{HL}$ that satisfies equation \eqref{eqC0}. Denote
\begin{equation}\label{eEF}
U_*(x)=\left(\frac{\sqrt{N(N-2)}}{1+|x|^2}\right)^{\frac{N-2}{2}}
\end{equation}
a groundstate of the Emden--Fowler equation $-\Delta U_*=U_*^{2^*-1}$ in $\R^N$.
Then (see e.g. \cite{Minbo-Du}*{Lemma 1.1}) all radial groundstates of \eqref{eqC0} are given by the function
\begin{equation}\label{eV1}
V(x)=\big(\mathcal{S}_*^\alpha\mathcal{C}_\alpha^2\big)^{-\frac{N-2}{4(\alpha+2)}}U_*(x)
\end{equation}
and the family of its rescalings
\begin{equation}\label{erC0}
V_{\lambda}(x)=\lambda^{-\frac{N-2}{2}}V(x/\lambda)\qquad(\lambda>0).
\end{equation}
In fact, if $N=3,4$ or if $N\ge 5$ and $\alpha\ge N-4$ then all finite energy {\em solutions} of \eqref{eqC0} are given by the rescalings and translations of $U_*$, see \cite{Minbo-Du}*{Theorem 1.1}.

We prove that in the critical Choquard regime the family of ground states $u_\eps$ converge in a suitable sense to $V$ after an {\em implicit} rescaling $\lambda_\eps$. Note that $V\in L^2(\R^N)$ only if $N\geq 5$ and hence the lower dimensions should be handled differently, as the $L^2$--norm of $u_\eps$ must blow up when $N=3,4$. Our principal result  is a sharp two--sided asymptotic estimate on the rescaling $\lambda_\eps$ as $\eps\to 0$. Similar result in the local case $\alpha=0$ was first observed in \cite{Muratov} and then rigorously established in \cite{MM-14}

\begin{theorem}[Critical Choquard limit]\label{t-CC-0}
	Let $p=\frac{N+\alpha}{N-2}$ and $q>\frac{2Np}{N+2}=2^*$.
	There exists a rescaling $\lambda_\eps:(0,\infty)\to(0,\infty)$ such that as $\eps\to 0$, the rescaled family of ground states
	$$\overline{v}_\eps(x):=\lambda_\eps^\frac{N-2}{2}u_\eps(\lambda_\eps x)$$
	converges to $V$ in $D^1(\R^N)$. Moreover, as $\eps\to 0$,
	\begin{equation}
	\lambda_{\eps}\sim\left\{\aligned & \eps^{-\frac{1}{q-4}} &	\text{ if $N=3$},&\\
	&\Big(\eps\ln\frac{1}{\eps}\Big)^{-\frac{1}{q-2}} &	\text{ if $N=4$},&\\
	&\eps^{-\frac{2}{(q-2)(N-2)}} & \text{ if $N\geq 5$}. &\\
	\endaligned\right.
	\end{equation}
\end{theorem}

\subsection{Critical Thomas--Fermi regime $q=\frac{2Np}{N+\alpha}$}
When $\frac{N+\alpha}{N}<p<\frac{N+\alpha}{N-2}$ and $q=\frac{2Np}{N+\alpha}$, neither \eqref{eqTF} nor $(P_0)$ have nontrivial solutions.
We show that in this case the limit equation for \eqref{eqPeps} is given by the critical Thomas--Fermi type equation
\begin{equation}\tag{$T\!F_*$}
|v|^{q-2}v=(I_{\alpha}*|v|^p)|v|^{p-2}v,\qquad v\in L^q(\R^N).
\end{equation}
By a {\em groundstate} of \eqref{eqTF0} we understand a positive solution
of \eqref{eqTF0} which is a rescaling of a nonnegative minimizer for the Hardy--Littlewood--Sobolev minimization problem
\begin{equation}\label{eqSTF}
	\mathcal{S}_{T\!F}=\inf_{w\in L^{\frac{2Np}{N+\alpha}}(\R^N)\setminus \{0\}}\frac{\int_{\R^N}|w|^\frac{2Np}{N+\alpha}dx}{\Big\{ \int_{\R^N}(I_{\alpha}*|w|^p)|w|^p dx\Big\}^{\frac{N}{N+\alpha}}}=\mathcal C_{\alpha}^{-\frac{N}{N+\alpha}},
	\end{equation}
	where $\mathcal C_{\alpha}$ is the optimal constant in \eqref{HLS}.
It is known \cite{Lieb}*{Theorem 4.3} that all radial groundstates of \eqref{eqTF0} are given by
\begin{equation}\label{eqUtilde}
\widetilde{U}(x)=\left(\frac{\sigma_{\alpha,N}}{1+|x|^2}\right)^{\frac{N+\alpha}{2p}},
\end{equation}
for a constant $\sigma_{\alpha,N}>0$, and the family of rescalings
\begin{equation}\label{eq-R-HLS}
\widetilde{U}_\lambda(x):=\lambda^{-\frac{N+\alpha}{2p}}\widetilde{U}(x/\lambda)=\lambda^{-\frac{N}{q}}\widetilde{U}(x/\lambda)\qquad(\lambda>0).
\end{equation}
We prove that, similarly to the critical Choquard regime, in the critical Thomas--Fermi regime the family of ground states $u_\eps$ converge in a suitable sense to $\widetilde{U}$ after an {\em implicit} rescaling $\lambda_\eps$.
Note that $\widetilde{U}\in L^2(\R^N)$ and $\widetilde{U}\in D^{1}(\R^N)$ only if $N\geq 4$, so $N=3$ is now the only special dimension. Our main result in the critical Thomas--Fermi regime is the following.

\begin{theorem}[Critical Thomas--Fermi limit]\label{t-HLS-0}
	Let $\frac{N+\alpha}{N}<p<\frac{N+\alpha}{N-2}$ and $q=\frac{2Np}{N+\alpha}$.
	There exists a rescaling $\lambda_\eps:(0,\infty)\to(0,\infty)$ such that as $\eps\to 0$, the rescaled family of ground states
	$$\overline{v}_\eps(x):=\lambda_\eps^\frac{N+\alpha}{2p}u_\eps(\lambda_\eps x)$$
	converges to $\widetilde{V}$ in $L^q(\R^N)$, where $\widetilde{V}$ is defined by \eqref{eqV1tilde}. Moreover,
	if $N\geq 4$ then as $\eps\to 0$,
	\begin{equation}\label{eq-HLS-0asy4}
	\lambda_{\eps}\sim\eps^{-\frac{1}{2}},
	\end{equation}
	while if $N=3$ then
	\begin{equation}\label{eq-HLS-0asy}
	\left\{\begin{array}{rcll}
	&\lambda_{\eps}&\sim\eps^{-\frac{1}{2}},&\qquad p\in \left(\tfrac13(3+\alpha), \tfrac23(3+\alpha)\right),\medskip\\
	\eps^{-\frac{1}{2}}(\ln\tfrac{1}{\eps})^{-\frac{1}{2}}\lesssim &\lambda_{\eps}&\lesssim \eps^{-\frac{1}{2}}(\ln\tfrac{1}{\eps})^{\frac{1}{6}},&\qquad p=\tfrac23(3+\alpha),\medskip\\
	\eps^{\frac{p-(3+\alpha)}{p}}\lesssim &\lambda_{\eps}&\lesssim\eps^{\frac{(3+\alpha)(3+\alpha-2p)}{p(3p-(3+\alpha))}},&\qquad p\in \left(\tfrac23(3+\alpha), 3+\alpha\right).
	\end{array}
	\right.
	\end{equation}
\end{theorem}

\begin{remark}
	We expect that the upper asymptotic bounds \eqref{eq-HLS-0asy} with $p\ge\tfrac{2(3+\alpha)}{3}$ could be refined to match the lower bounds, but this remains open at the moment.
\end{remark}

\section[Asymptotic profiles as $\eps\to \infty$]{Asymptotic profiles as $\eps\to \infty$ and Gross--Pitaevskii--Poisson model}\label{s3}

\begin{figure}[t]
	\centering

	\tikzset{every picture/.style={line width=0.75pt}} 
	
	\begin{tikzpicture}[x=0.75pt,y=0.75pt,yscale=-0.75,xscale=0.75]
		
		
		
		\tikzset{
			pattern size/.store in=\mcSize,
			pattern size = 5pt,
			pattern thickness/.store in=\mcThickness,
			pattern thickness = 0.3pt,
			pattern radius/.store in=\mcRadius,
			pattern radius = 1pt}
		\makeatletter
		\pgfutil@ifundefined{pgf@pattern@name@_ro2tk9cc0}{
			\pgfdeclarepatternformonly[\mcThickness,\mcSize]{_ro2tk9cc0}
			{\pgfqpoint{0pt}{-\mcThickness}}
			{\pgfpoint{\mcSize}{\mcSize}}
			{\pgfpoint{\mcSize}{\mcSize}}
			{
				\pgfsetcolor{\tikz@pattern@color}
				\pgfsetlinewidth{\mcThickness}
				\pgfpathmoveto{\pgfqpoint{0pt}{\mcSize}}
				\pgfpathlineto{\pgfpoint{\mcSize+\mcThickness}{-\mcThickness}}
				\pgfusepath{stroke}
		}}
		\makeatother
		\tikzset{every picture/.style={line width=0.75pt}} 
		

		

		\draw  [draw opacity=0][fill={rgb, 255:red, 208; green, 2; blue, 27 }  ,fill opacity=0.15 ] (310,310) -- (460,310) -- (460,510) -- (310,510) -- cycle ;
		\draw  [draw opacity=0][fill={rgb, 255:red, 208; green, 2; blue, 27 }  ,fill opacity=0.15 ] (110,130) -- (150,130) -- (150,510) -- (110,510) -- cycle ;
		\draw  [draw opacity=0][fill={rgb, 255:red, 208; green, 2; blue, 27 }  ,fill opacity=0.15 ] (460,130) -- (310,310) -- (460,310) -- cycle ;
		\draw  [dash pattern={on 0.84pt off 2.51pt}]  (110,310) -- (460,310) ;

		\draw [line width=1.5]  (71.07,510) -- (476.43,510)(110,112.97) -- (110,551.94) (469.43,505) -- (476.43,510) -- (469.43,515) (105,119.97) -- (110,112.97) -- (115,119.97)  ;
		\draw    (360,505) -- (359.92,513.61) ;
		
		\draw  [draw opacity=0][fill={rgb, 255:red, 248; green, 231; blue, 28 }  ,fill opacity=0.5 ][line width=0]  (310,310) -- (150,470) -- (150,130) -- (460,130)-- cycle ;
		
		\draw  [draw opacity=0][fill={rgb, 255:red, 126; green, 211; blue, 33 }  ,fill opacity=0.2][line width=0]  (310,310) --  (150,470) -- (150,510) -- (310,510) -- cycle ;
		

		\draw [color={rgb, 255:red, 126; green, 211; blue, 33 }  ,draw opacity=1 ][line width=1.5]  [dash pattern={on 5.63pt off 4.5pt}]  (110,510) -- (460,160) ;
		
		\draw [dash pattern={on 0.84pt off 2.51pt}] [line width=0.75]    (310,130) -- (310,310) ;

		\draw [color={rgb, 255:red, 208; green, 2; blue, 27 }  ,draw opacity=1 ][fill={rgb, 255:red, 234; green, 178; blue, 85 }  ,fill opacity=1 ][line width=0.75]    (150,130) -- (150,510) ;

		\draw [color={rgb, 255:red, 208; green, 2; blue, 27 }  ,draw opacity=1 ][line width=0.75]  [dash pattern={on 4.5pt off 4.5pt}]  (150,510) -- (310,510) ;

		\draw [color={rgb, 255:red, 208; green, 2; blue, 27 }  ,draw opacity=1 ][line width=0.75]    (310,310) -- (310,510) ;



		\draw [color={rgb, 255:red, 126; green, 211; blue, 33 }  ,draw opacity=1 ][line width=3]    (150,470) -- (310,310) ;

		\draw [color={rgb, 255:red, 126; green, 211; blue, 33 }  ,draw opacity=1 ][line width=3]    (590,410) -- (470,410) ;
		

		\draw [color={rgb, 255:red, 208; green, 2; blue, 27 }  ,draw opacity=1 ][line width=0.75]    (460,130) -- (310,310) ;
		
		\draw [shift={(310,310)}, rotate = 129.81] [color={rgb, 255:red, 208; green, 2; blue, 27 }  ,draw opacity=1 ][fill={rgb, 255:red, 208; green, 2; blue, 27 }  ,fill opacity=1 ][line width=0.75]      (0, 0) circle [x radius= 3.35, y radius= 3.35]   ;
		

		\draw (145,525) node [scale=0.9]  {$\frac{N+\alpha }{N}$};
		\draw (102.5,111) node [scale=0.9]  {$q$};
		\draw (467.5,519) node [scale=0.9]  {$p$};
		\draw (117.5,519) node [scale=0.9]  {$1$};
		\draw (102.5,499) node [scale=0.9]  {$2$};
		\draw (102,300) node [scale=0.9]  {$2^{*}$};
		\draw (305,525) node [scale=0.9]  {$\frac{N+\alpha }{N-2}$};
		\draw (362,520) node [scale=0.9]  {$2^{*}$};
		
		\draw  [fill={rgb, 255:red, 126; green, 211; blue, 33 }  ,fill opacity=1 ]  (230.5, 468.5) circle [x radius= 55.96, y radius= 15.73]   ;
		\draw (230.5,468.5) node [scale=0.9]  {$Choquard$};
		\draw (499,155) node [scale=0.9]  {$q=2\frac{2p+\alpha }{2+\alpha}$};
		\draw (495,125) node [scale=0.9]  {$q=\frac{2Np}{N+\alpha}$};
		\draw (527,398) node [scale=0.9] [align=left] {\textit{Selfrescaling}};
		\draw (527,258) node [scale=0.9] [align=left] {\huge{$\boxed{\eps\to\infty}$}};

		
		
		\draw  [fill={rgb, 255:red, 248; green, 231; blue, 28 }  ,fill opacity=1 ][line width=0.75] (230.5, 201.5) circle [x radius= 30, y radius= 30]   ;
		\draw (230.5,201.5) node [scale=0.9]  {$T\!F$};
		
		
	\end{tikzpicture}

	\caption{Three limit regimes for $(P_\eps)$ as $\eps\to \infty$ on the $(p,q)$--plane.} \label{fig:M2}
\end{figure}


The behaviour of ground states $u_\eps$ as $\eps\to\infty$ is less complex than in the case $\eps\to 0$.
Only the 1st and the 2nd rescalings are meaningful, separated by the $q=2\frac{2p+\alpha}{2+\alpha}$ line, however the limit equations ``switch'' compared to the case $\eps\to 0$.
There are no critical regimes.

\begin{theorem}[Choquard limit]\label{thmC-lim-inf}
	Let $\frac{N+\alpha}{N}<p<\frac{N+\alpha}{N-2}$ and $2<q<2\frac{2p+\alpha}{2+\alpha}$.
		As $\eps\to\infty$, the rescaled family of ground states
	\begin{equation}\label{eq1res-inf}
	v_\eps(x):=\eps^{-\frac{2+\alpha}{4(p-1)}}u_\eps(\tfrac{x}{\sqrt{\eps}})
	\end{equation}
	converges in $D^1(\R^N)$ and $L^q(\R^N)$
	to a positive spherically symmetric ground state solution $v_0\in D^1\cap L^q(\R^N)$ of the Choquard equation \eqref{eqC}.
\end{theorem}

Clearly, for $q=\frac{2p+\alpha}{2+\alpha}$ the self--similar regime of Theorem \ref{thm02-0} is valid also as $\eps\to\infty$.
For $p=2$ and $q>\frac{2p+\alpha}{2+\alpha}$ we have the following.

\begin{theorem}[Thomas--Fermi limit for $\alpha=2$]\label{t-TF-0-inf}
	Assume that $p=2$ and $\alpha=2$. Let $N\le 5$ and $q>3$,
	or $N\ge 6$ and $q>\frac{4N}{N+2}$.
	As $\eps\to \infty$, the rescaled family of ground states
	\begin{equation}\label{eq2res-inf}
	v_\eps(x):=\eps^{-\frac{1}{q-2}}u_\eps(\eps^{-\frac{4-q}{2(q-2)}} x)
	\end{equation}
	converges in $L^2(\R^N)$ and $L^q(\R^N)$ to a nonnegative spherically symmetric compactly supported ground state solution $v_0\in L^2\cap L^q(\R^N)$ of the Thomas-Fermi equation \eqref{eqTF2}. 
\end{theorem}

The proofs of Theorems \ref{t-TF-0-inf} and \ref{thmC-lim-inf} are very similar to the proofs of Theorems \ref{t-TF-0} and \ref{thmC-lim}.
We only note that the proof on Theorem~\ref{t-TF-0-inf} will involve the estimate \eqref{eq1203} with $q\ge 4$ when the the right hand side of \eqref{eq1203} blows-up. However the rate of the blow-up in \eqref{eq1203} isn't strong enough and all quantities involved in the proof remain under control. We leave the details to the interested readers.

\begin{remark}\label{rGPP}
	Note that the nature of rescaling \eqref{eq2res-inf} changes when $q=4$: for $q>4$ the mass of $u_\eps$ concentrates near the origin, while for $q<4$ it ``escapes'' to infinity.
	In particular, the stationary version of the Gross--Pitaevskii--Poisson equation \eqref{GPP} ($q=4$, $N=3$, $\alpha=2$) fits into the Thomas--Fermi regime as $\eps\to\infty$.
	The rescaling \eqref{eq2res-inf} in this case takes the simple form $v_\eps(x)=\eps^{-1/2}u_\eps(x)$ and $u_\eps(x)\approx\sqrt{\eps}v_0(x)$, or we can say that $u_\eps$ concentrates towards the compactly supported $v_0$. This is precisely the phenomenon which was already observed in \citelist{\cite{Wang}\cite{Bohmer-Harko}}, where the radius of the support of $v_0$ has the meaning of the radius of self-gravitating Bose--Einstein condensate, see \cite{Chavanis-11}.
	The limit minimization problem $s_{T\!F}$ in the Gross--Pitaevskii--Poisson equation \eqref{GPP} case becomes
	$$
	s_{T\!F}=\inf\Big\{\int_{\R^3}\rho^2dx+\int_{\R^3}\rho\, dx: 0\le \rho\in L^1\cap L^2(\R^3), \int_{\R^3}(I_2*\rho)\rho\, dx=1\Big\},
	$$
	and the Euler--Lagrange equation \eqref{TF-rho} in this case is linear inside the support of $\rho$:
	\begin{equation}\label{TF-rho-3d4}
	2\rho=\big(\tfrac{6}{5}s_{T\!F} I_{2}*\rho-1\big)_+\quad\text{a.e. in $\R^3$}.
	\end{equation}
	To find explicitly the solution of \eqref{TF-rho-3d4} constructed in Theorem \ref{thmTF}, we use the $\frac{\sin(|x|)}{|x|}$ ansatz as in \citelist{\cite{Chandrasekhar}*{p.92}\cite{Wang}\cite{Bohmer-Harko}}.

For $\lambda>0$ and $|x|\le \pi/\lambda$ consider the family
$$\rho_\lambda(|x|)=
\begin{cases}
\frac{\lambda^\frac{5}{2}}{\sqrt{2}\pi}\cdot\frac{\sin(|\lambda x|)}{\lambda|x|},&|x|<\pi/\lambda,\smallskip\\
0,&|x|\ge\pi/\lambda.
\end{cases}
$$
Then $-\Delta\rho_\lambda=\lambda^2\rho_\lambda$ in $B_{\pi/\lambda}$, and
$$I_2*\rho_\lambda(x)=
\begin{cases}
\lambda^{-2}\big(\rho_\lambda(x)+\frac{\lambda^\frac{5}{2}}{\sqrt{2}\pi}\big),&|x|<\pi/\lambda,\smallskip\\
\frac{\lambda^\frac{1}{2}}{\sqrt{2}\pi}\frac{I_2(|x|)}{I_2(\pi/\lambda)},&|x|\ge\pi/\lambda.
\end{cases}
$$
We compute
\begin{equation}
\int_{\R^3} I_2*\rho_\lambda(x)\rho_\lambda(x) dx=1.
\end{equation}
Optimising in $\lambda>0$, we find that
$$
s_{T\!F}=\inf_{\lambda>0}\Big(\int_{\R^3}\rho_\lambda^2(x)dx+\int_{\R^3}\rho_\lambda(x)dx\Big)=\inf_{\lambda>0}\Big(\frac{\lambda^2}{3}+\frac{2\sqrt{2}\pi}{3\sqrt{\lambda}}\Big)=\frac{5}{6}\cdot2^{\frac{3}{5}}\pi^{\frac{4}{5}}
$$
and the minimum occurs at the optimal $\lambda_*=\big(\frac{\pi^2}{2}\big)^{1/5}$.
Taking into account the uniqueness of the spherically symmetric nonincreasing minimizer for $s_{T\!F}$ in the case $\alpha=2$, which follows from \cite[Theorem 1.2]{Carrillo-Sugiyama} (see also \cite[Lemma 5.2]{Volzone}), the function
$$\rho_*(|x|)=\rho_{\lambda_*}(|x|)
=
\begin{cases}
\frac{\sin\big(2^{-\frac{1}{5}}\pi^\frac{2}{5}|x|\big)}{2^{\frac{4}{5}}\pi^\frac{2}{5}|x|},&|x|<2^\frac{1}{5}\pi^\frac{3}{5},\smallskip\\
0,&|x|\ge 2^\frac{1}{5}\pi^\frac{3}{5}.
\end{cases}
$$
is the unique spherically symmetric nonincreasing minimizer for $s_{T\!F}$ and a solution of \eqref{TF-rho-3d4}.

The solution of the limit Thomas--Fermi equation \eqref{eqTF}, which is written in this case as
\begin{equation}\label{TF-rho-3d4v0}
v+v^2=(I_{2}*v^2)v\quad\text{a.e. in $\R^3$}
\end{equation}
is given by the rescaled function in \eqref{TF-resc-min-rho-1},
\begin{equation}\label{eq-v03}
v_0(x)=\sqrt{2\rho_*\big(\sqrt{2}\big(\tfrac{6}{5}s_{T\!F}\big)^{-\frac{1}{2}} |x|\big)}
=
\begin{cases}
\sqrt{\frac{\sin\big(|x|\big)}{|x|}},&|x|<\pi,\smallskip\\
0,&|x|\ge \pi.
\end{cases}
\end{equation}
This is (up to the physical constants) the Thomas--Fermi approximation solution for self--gravitating BEC observed in \citelist{\cite{Wang}\cite{Bohmer-Harko}\cite{Chavanis-11}} and the support radius $R_0=\pi$ is the approximate radius of the BEC star. Our Theorem \ref{t-TF-0-inf} provides a rigorous justification for the convergence of the Thomas--Fermi approximation.
\end{remark}

\section{Existence and properties of groundstates for \eqref{eqPeps}}\label{s4}

\subsection{Variational setup}
It is a standard consequence of Sobolev and Hardy-Littlewood-Sobolev (HLS) inequalities \cite{Lieb-Loss}*{Theorems 4.3 and 8.3}
that for $\frac{N+\alpha}{N}\le p\le \frac{N+\alpha}{N-2}$ and $2<q\le\frac{2N}{N-2}$ the energy functional $\mathcal{I}_\eps\in C^1(H^1(\R^N),\R)$, cf.~\cite{MS17}.
For $p>\frac{N+\alpha}{N-2}$ the energy $\mathcal{I}_\eps$ is not well-defined on $H^1(\R^N)$. In this case an additional assumption $q>\frac{2Np}{N+\alpha}$ ensures the control of the nonlocal term by the $L^q$ and $L^2$--norm via the HLS inequality and interpolation, i.e.
\begin{equation}\label{HLSq}
\int_{\R^N}(I_{\alpha}*|u|^p)|u|^p dx\le C\|u\|_{\frac{2Np}{N+\alpha}}^{2p}\le C\|u\|_{2}^{2p\theta}\|u\|_{q}^{2p(1-\theta)},
\end{equation}
for a $\theta\in(0,1)$.
As a consequence, for $p>\frac{N+\alpha}{N-2}$ and $q>\frac{2Np}{N+\alpha}$ the energy $\mathcal{I}_\eps$ is well-defined on the space
$$
\H:=H^1(\R^N)\cap L^{q^*}(\R^N),\qquad q^*:=\max\{q,2^*\}.
$$
Clearly, $\H$ endowed with the norm
$$\|u\|_{\H}:=\|\nabla u\|_{L^2}+\|u\|_{L^2}+(q^*\!-\!2^*)\|u\|_{L^{q^*}}$$
is a Banach space, $\H\hookrightarrow L^{\frac{2Np}{N+\alpha}}(\R^{N})$ for any $q>2$ and $\H=H^1(\R^N)$ when $2<q\le 2^*$.
It is easy to check that $\mathcal{I}_\eps\in C^1(\H, \R)$ and the problem \eqref{eqPeps} is variationaly well-posed,
in the sense that weak solutions $u\in\H$ of \eqref{eqPeps} are critical points of $\mathcal{I}_\eps$,
i.e.
$$
\langle\mathcal{I}_{\eps}^\prime(u),\varphi\rangle_{\H}=\int_{\R^N}\nabla u\cdot\nabla \varphi dx+\eps \int_{\R^N}u\varphi dx-\int_{\R^N}(I_{\alpha}*|u|^p)|u|^{p-2}\varphi dx+\int_{\R^N}|u|^{q-2}u\varphi dx=0,
$$
for all $\varphi\in \H$. In particular, weak solutions $u\in\H$ of \eqref{eqPeps} satisfy the {\em Nehari identity}
\begin{equation}\label{e-Nehari}
\int_{\R^N}|\nabla u|^2 dx+\eps\int_{\R^N}|u|^2 dx+\int_{\R^N}|u|^q dx-\int_{\R^N}(I_{\alpha}*|u|^p)|u|^p dx=0.
\end{equation}
It is standard to see that under minor regularity assumptions weak solutions of \eqref{eqPeps} also satisfy the {\em Poho\v{z}aev identity}.

\begin{proposition}[Poho\v{z}aev identity]\label{p-Phz}
	Assume $p>1$ and $q>2$. Let $u\in\H\cap L^\frac{2Np}{N+\alpha}(\R^N)$ be a weak solution of \eqref{eqPeps}.
	If $\nabla u\in L^{\frac{2Np}{N+\alpha}}_{loc}(\R^N)\cap H^{1}_{loc}(\R^N)$ then
	$$
	\frac{N-2}{2}\int_{\R^N}|\nabla u|^2dx+\frac{\eps N}{2}\int_{\R^N}|u|^2dx+\frac{N}{q}\int_{\R^N}|u|^qdx-\frac{N+\alpha}{2p}\int_{\R^N}(I_{\alpha}*|u|^p)|u|^pdx=0.
	$$
\end{proposition}

\proof
The proof is an adaptation of \cite{MS13}*{Proposition 3.1}, we omit the details.
\qed

\smallskip
As a consequence, we conclude that the existence range stated in Theorem~\ref{thm01} is optimal.

\begin{corollary}[Nonexistence]\label{c-non}
Let $1<p\le\frac{N+\alpha}{N}$ and $q>2$, or $p\geq\frac{N+\alpha}{N-2}$ and $2<q\le\frac{2Np}{N+\alpha}$. Then \eqref{eqPeps} has no nontrivial weak solutions
$u\in\H\cap L^\frac{2Np}{N+\alpha}(\R^N)\cap W^{1,\frac{2Np}{N+\alpha}}_{loc}(\R^N)\cap W^{2,2}_{loc}(\R^N)$.
\end{corollary}

\proof
Follows from Poho\v{z}aev and Nehari identities.
\qed

\subsection{Apriori regularity and decay at infinity}

We show that all weak nonnegative solutions of \eqref{eqPeps} are in fact bounded classical solutions with an $L^1$--decay at infinity. We first prove a partial results which relies on the maximum principle for the Laplacian.

\begin{lemma}\label{l-infty}
	Assume $p>\frac{N+2}{N}$ and $q>\max\{p,2\}$. Let $s>\frac{Np}{\alpha}$ and $u\in\H\cap L^s(\R^N)$ be a nonnegative weak solution of \eqref{eqPeps}.
	Then $u\in L^\infty(\R^N)$.
\end{lemma}

\proof
The assumption $s>\frac{Np}{\alpha}$ imply that $I_{\alpha}*|u|^p\in L^\infty(\R^N)$.
Then $u\ge 0$ weakly satisfies
\begin{equation}\label{eq-maxpr}
-\Delta u\leq C u^{p-1}-u^{q-1}\quad\text{in $\R^N$,}
\end{equation}
where $C=C(u)=\|I_{\alpha}*|u|^p\|_\infty$. Choose $m=m(u)>0$ such that $Cm^{p-1}-m^{q-1}=0$.
Testing against $u_m=(u-m)_+\in H^1(\R^N)$, we obtain
\begin{equation}\label{eq-maxpr+}
\int_{\R^N}|\nabla u_m|^2dx=\int_{\R^N}\nabla u\cdot\nabla u_m dx\le\int_{\R^N}\big(C u^{p-1}-u^{q-1}\big)u_m dx\le 0,
\end{equation}
so $\|u\|_\infty\le m$.
\qed

The proof of the next statement in the case  $p<\frac{N+\alpha}{N-2}$ is an adaptation of the iteration arguments in \cite{MS13}*{Proposition 4.1}.
We only outline the main steps of the proof.
The case  $p\ge\frac{N+\alpha}{N-2}$ is new and relies heavily on the contraction inequality \eqref{e-Ponce-sub}, which is discussed in the appendix.

\begin{proposition}[Regularity and positivity]\label{p-reg}
	Let $\frac{N+\alpha}{N}<p<\frac{N+\alpha}{N-2}$ and $q>2$, or $p\ge\frac{N+\alpha}{N-2}$
	and $q>\frac{2Np}{N+\alpha}$.
	If $0\le u\in\H$ is a nontrivial weak solution of \eqref{eqPeps} then
	$u\in L^1\cap C^2(\R^N)$ and $u(x)>0$ for all $x\in\R^N$.
\end{proposition}

\proof
Since $u\in\H$ we know that $u\in L^s(\R^N)$ for all $s\in[2,q^*]$.

\medskip\noindent
{\bf Step 1.} {\em $u\in L^1(\R^N)$.}

\proof
Note that $u\ge 0$ weakly satisfies the inequality
\begin{equation}\label{eq-sniter}
-\Delta u +\eps u\le (I_{\alpha}*u^p)u^{p-1}\quad\text{in $\R^N$}.
\end{equation}
Since
$(-\Delta+\eps)^{-1}:L^s(\R^N)\mapsto L^s(\R^N)$ is a bounded order preserving linear mapping for any $s\ge 1$, we have
\begin{equation}\label{eq-sniter1}
u\le (-\Delta+\eps)^{-1}\big((I_{\alpha}*u^p)u^{p-1}\big)\quad\text{in $\R^N$.}
\end{equation}
Then, by the HLS and H\"older inequalities, $u\in L^{\underline{s}_n}(\R^N)$ with $0<\frac{2p-1}{\underline{s}_{n}}-\frac{\alpha}{N}<1$ implies $u\in L^{\underline{s}_{n+1}}(\R^N)$, where
$$\frac{1}{\underline{s}_{n+1}}=\frac{2p-1}{\underline{s}_{n}}-\frac{\alpha}{N}.$$
Since $p>\frac{N+\alpha}{N}$, we start the $\underline{s}_n$--iteration with $\underline{s}_0=\frac{2Np}{N+\alpha}<\frac{2N(p-1)}{\alpha}$, as in \cite{MS13}.
Then we achieve $\underline{s}_{n+1}\ge 1$ after a finite number of steps.
This implies $u\in L^1(\R^N)$.
\qed

\medskip\noindent
{\bf Step 2.} {\em $u\in L^\infty(\R^N)$.}

\proof
Assume that  $q\le \frac{Np}{\alpha}$, otherwise we conclude by Lemma \ref{l-infty}.
We consider separately the cases $p<\frac{N+\alpha}{N-2}$ and $p\ge\frac{N+\alpha}{N-2}$, which use different structures within the equation \eqref{eqPeps}.

\smallskip\noindent
{\sl A. Case $p<\frac{N+\alpha}{N-2}$.}
Note that $u\ge 0$ weakly satisfies the inequality
\begin{equation}\label{eq-sniter-sub}
-\Delta u\le (I_{\alpha}*u^p)u^{p-1}\quad\text{in $\R^N$}.
\end{equation}
Since $u\in L^1(\R^N)$ ,
we have $(I_{\alpha}*u^p)u^{p-1}\in L^t(\R^N)$ with $\frac1t:=2p-1-\frac{\alpha}{N}>\frac{2}{N}$.
We conclude that
\begin{equation}\label{eq-sniter2}
u\le I_2*\big((I_{\alpha}*u^p)u^{p-1}\big)\quad\text{in $\R^N$}.
\end{equation}
Then, by the HLS and H\"older inequalities, $u\in L^{\overline{s}_n}(\R^N)$ with $0<\frac{2p-1}{\overline{s}_{n}}-\frac{\alpha+2}{N}<1$ implies $u\in L^{\overline{s}_{n+1}}(\R^N)$, where
$$\frac{1}{\overline{s}_{n+1}}=\frac{2p-1}{\overline{s}_{n}}-\frac{\alpha+2}{N}.$$
Since $p<\frac{N+\alpha}{N-2}$, we start the $\overline{s}_n$--iteration with $\overline{s}_0=\frac{2Np}{N+\alpha}>\frac{2N(p-1)}{\alpha+2}$, as in \cite{MS13}.
Then we achieve $\overline{s}_{n+1}>\frac{Np}{\alpha}$ after a finite number of steps.
(Or if $\overline{s}_{n+1}=\frac{Np}{\alpha}$ we readjust $\overline{s}_0$.)

\smallskip\noindent
{\sl B. Case $p\ge\frac{N+\alpha}{N-2}$ and $q>\frac{2Np}{N+\alpha}$.}
Note that $u\ge 0$ weakly satisfies the inequality
\begin{equation}\label{eq-sniter-super}
-\Delta u+u^{q-1}\le (I_{\alpha}*u^p)u^{p-1}\quad\text{in $\R^N$}.
\end{equation}
Then, by the HLS and H\"older inequalities, and by the contraction inequality \ref{e-Ponce-sub}, $u\in L^{\overline{s}_n}(\R^N)$ with $0<\frac{2p-1}{\overline{s}_{n}}-\frac{\alpha}{N}<1$ implies $u\in L^{\overline{s}_{n+1}}(\R^N)$, where
$$\frac{q-1}{\overline{s}_{n+1}}=\frac{2p-1}{\overline{s}_{n}}-\frac{\alpha}{N}.$$
We start the $\overline{s}_n$--iteration with $\overline{s}_0=q$. If $q\ge 2p$ we achieve $\overline{s}_{n+1}>\frac{Np}{\alpha}$ after a finite number of steps.
If $q<2p$ we note that since $p\ge\frac{N+\alpha}{N-2}$, we have $\overline{s}_0=q>\frac{2Np}{N+\alpha}$.
Then we again achieve $\overline{s}_{n+1}>\frac{Np}{\alpha}$ after a finite number of steps. (Or if $\overline{s}_{n+1}=\frac{Np}{\alpha}$ we readjust $\overline{s}_0$.)

\medskip\noindent
{\bf Step 3.} {\em $u\in W^{2,r}(\R^N)$ for every $r>1$ and $u\in C^2(\R^N)$.}

\proof
Since $u\in L^1\cap L^\infty(\R^N)$, we have
\begin{equation*}
-\Delta u +\eps u=F\quad\text{in $\R^N$},
\end{equation*}
where $F:=(I_{\alpha}*u^p)u^{p-1}-u^{q-1}\in L^\infty(\R^N)$.
Then the conclusion follows by the standard Schauder estimates, see \cite{MS13}*{p.168} for details.
\qed

\medskip\noindent
{\bf Step 4.} {\em $u(x)>0$ for all $x\in\R^N$.}

\proof
We simply note that $u\ge 0$ satisfies
\begin{equation}\label{eq-wHarnack}
-\Delta u+V(x)u\ge 0\quad\text{in $\R^N$},
\end{equation}
where $V:=\eps+u^{q-2}\in C(\R^N)$. Then $u(x)>0$ for all $x\in\R^N$, e.g.~by the weak Harnack inequality.
\qed

\begin{proposition}[Decay asymptotics]
	Let $\frac{N+\alpha}{N}<p<\frac{N+\alpha}{N-2}$ and $q>2$ or $p\geq\frac{N+\alpha}{N-2}$ and $q>\frac{2Np}{N+\alpha}$.
	Let $0<u_\eps\in L^1\cap C^2(\R^N)$ be a radially symmetric solution of \eqref{eqPeps}.
	Then
		\begin{equation}\label{eq-localize}
		\lim_{\abs{x} \to \infty}(I_\alpha \ast u_\eps^p) (x)|x|^{N-\alpha} =A_\alpha \|u_\eps\|_p^p.
		\end{equation}
	Moreover,  there exists $C_\eps>0$ such that
	\begin{itemize}
		\item if \(p > 2\),
		\[
		\lim_{\abs{x} \to \infty} u_\eps(x) \abs{x}^{\frac{N - 1}{2}} e^{\sqrt{\eps}\abs{x}}=C_\eps,
		\]
		\item if \(p = 2\),
		\[
		\lim_{\abs{x} \to \infty} u_\eps(x) \abs{x}^{\frac{N - 1}{2}} \exp \int_{\nu}^{\abs{x}} \sqrt{\eps - \tfrac{\nu^{N - \alpha}}{s^{N - \alpha}}} \,ds=C_\eps,\quad\text{where $\nu:=\big(A_\alpha\|u_\eps\|_2^2\big)^\frac{1}{N - \alpha} $},
		\]
		\item if \(p < 2\),
		$$\lim_{x\to\infty}u_\eps(x)|x|^\frac{N-\alpha}{2-p}=\left(\eps^{-1}A_\alpha\|u_\eps\|_p^p\right)^{\frac1{2-p}}.$$
	\end{itemize}
\end{proposition}

\proof
To simplify the notation, we drop the subscript $\eps$ for $u_\eps$ in this proof.

Let $u\in L^1\cap C^2(\R^N)$ be a positive radially symmetric solution of \eqref{eqPeps}.
By the Strauss' radial $L^1$--bound \cite{BL-I}*{Lemma A.4}, $u^p(|x|)\le C\|u\|_1|x|^{-Np}$ for all $x\in\R^N$. Then by \cite{MS13}*{Propositions 6.1}, there exists \(\mu \in \R\) such that
\begin{equation}
\label{eqSublinearAsymptoticsRiesz}
\Big|I_\alpha \ast u^p (x) - I_\alpha(x) \|u\|_p^p\Big|
\le \frac{\mu}{\abs{x}^{N - \alpha + \delta}}\quad\text{for $x \in \R^N$},
\end{equation}
with \(0 < \delta \le \min (1, N (p - 1))\).
In particular, this implies \eqref{eq-localize}.

\medskip\noindent
{\bf Case $p\ge 2$.}
The derivation of the decay asymptotic of $u$ in the case(s) $p\ge 2$ requires minimal modifications of the proofs of \cite{MS13}*{Propositions 6.3, 6.5}. Indeed, \eqref{eq-localize} implies $(I_\alpha*|u|^p)u^{p-2}(x)\le\frac34 \eps$ for all $|x|>\rho$, for some $\rho>0$. Therefore, $u$ satisfies
$$-\Delta u+\frac{\eps}{4}u\le 0 \quad\text{in $|x|>\rho$}.$$
As in \cite{MS13}*{Propositions 6.3} we conclude that
\begin{equation}\label{e-exprough}
u(x)\le c|x|^{-\frac{N-1}{2}}e^{-\frac{\sqrt{\eps}}{4}|x|}.
\end{equation}
Therefore, $u$ is a solution of
$$-\Delta u+W_\eps(x)u=0 \quad\text{in $|x|>\rho$},$$
where
$$W_\eps(x):=\eps-(I_\alpha*u^p)u^{p-2}(x)+u^{q-2}(x) \simeq \eps-\frac{A_\alpha\|u\|_p^p}{|x|^{N-\alpha}}u^{p-2}(x)+u^{q-2}(x)\quad\text{as $|x|\to\infty$.}
$$
The initial rough upper bound \eqref{e-exprough} implies that the term $u^{q-2}(x)$ in the linearisation potential $W_\eps(x)$ has an exponential decay and is negligible in the subsequent asymptotic analysis of Propositions 6.3 and 6.5 in \cite{MS13}.
We omit the details.

\medskip\noindent
{\bf Case $p< 2$.}
This proof is an adaptation of \cite{MS13}*{Propositions 6.6}.

To derive the upper bound, we note that by Young's inequality,
\[
(I_\alpha \ast \abs{u}^p) u^{p - 1} \le \eps^{-\frac{p-1}{2-p}}(2 - p) (I_\alpha \ast \abs{u}^p)^\frac{1}{2 - p} + \eps(p - 1) u.
\]
By \eqref{eqSublinearAsymptoticsRiesz}, we have
\[
\bigl((I_\alpha \ast \abs{u}^p) (x)\bigr)^\frac{1}{2 - p}
\le I_\alpha (x)^\frac{1}{2 - p} \bigl(\|u\|_p^p + c|x|^{-\delta}\bigr)^\frac{1}{2-p}
\le I_\alpha (x)^\frac{1}{2 - p} \bigl(\|u\|_p^\frac{p}{2 - p} + c|x|^{-\delta}\bigr)
\quad\text{for $|x|>1$.}
\]
Therefore, $u$ satisfies the inequality
\[
-\Delta u  + \eps(2 - p) u \le \eps^{-\frac{p-1}{2-p}}(2 - p) I_\alpha (x)^\frac{1}{2 - p}
\bigl(\|u\|_p^\frac{p}{2 - p} + c|x|^{-\delta}\bigr)\quad\text{if $|x|>1$}.
\]
Define now \(\Bar{u} \in C^2 (\R^N \setminus B_1)\) by
\[
\left\{
\begin{aligned}
-\Delta \Bar{u} + \eps(2 -p)\Bar{u} &= \eps^{-\frac{p-1}{2-p}}(2-p) I_\alpha (x)^\frac{1}{2 - p} \bigl(\|u\|_p^\frac{p}{2 - p} + c|x|^{-\delta}\bigr) & & \text{if $|x|>1$,}\\
\Bar{u} (x) & = u (x)  & & \text{if $|x|=1$},\\
\lim_{\abs{x} \to \infty} \Bar{u} (x) & = 0.
\end{aligned}
\right.
\]
We now apply \cite{MS13}*{Lemma 6.7} twice and use the linearity of the operator $-\Delta+\eps(2-p)$ to obtain
\[
\lim_{\abs{x} \to \infty} \frac{\Bar{u}(x)}{I_\alpha (x)^\frac{1}{2 - p}}
= \eps^{-\frac{1}{2-p}}\|u\|_p^\frac{p}{2 - p}.
\]
By the comparison principle, we have \(u \le \Bar{u}\) in \(\R^N \setminus B_1\).
Thus
\begin{equation}
\label{eqLimSupSublinearAsymptotics}
\limsup_{\abs{x} \to \infty}  \frac{u^{2-p} (x)}{I_\alpha (x)} \le \eps^{-1}\|u||_p^p.
\end{equation}

To deduce the lower bound, note that by the chain rule, \(u^{2 - p} \in C^2 (\R^N)\) and
\[
- \Delta u^{2 - p} = - (2 - p) u^{1 - p} \Delta u + (2 - p) (p - 1) \abs{\nabla u}^2.
\]
Since $p \in (1, 2)$ and $q>2$, by the equation satisfied by $u$ and by \eqref{eqSublinearAsymptoticsRiesz} and \eqref{eqLimSupSublinearAsymptotics},
for some $c_\eps>0$ we have
\begin{multline*}
- \Delta u^{2 - p} + \eps(2 - p) u^{2 - p} \ge(2 - p)I_\alpha (x) \bigl(\|u\|_p^p - \mu|x|^{-\delta}\bigr)-(2-p)u^{q-p}\\
\ge(2 - p)I_\alpha (x) \bigl(\|u\|_p^p - \mu|x|^{-\delta}\bigr)-c_\eps I_\alpha(x)^\frac{q-p}{2-p}\quad\text{for $x \in \R^N$.}
\end{multline*}
Let \(\underline{u} \in C^2(\R^N \setminus B_1)\) be such that
\[
\left\{
\begin{aligned}
-\Delta \underline{u}  + \eps(2 - p) \underline{u} &\le (2 - p) I_\alpha (x) \bigl(\|u\|_p^p - \mu|x|^{-\delta}\bigr)-c_\eps I_\alpha(x)^\frac{q-p}{2-p} & & \text{if $|x|>1$,}\\
\underline{u} (x) & = u (x)^{2 - p}  & &\text{if $|x|=1$,}\\
\lim_{\abs{x} \to \infty} \underline{u} (x) & = 0.
\end{aligned}
\right.
\]
We apply now \cite{MS13}*{Lemma 6.7} to deduce
\[
\lim_{\abs{x} \to \infty}  \frac{\underline{u}(x)}{I_\alpha (x)} = \eps^{-1}\|u\|_p^p.
\]
By the comparison principle,
\(\underline{u} \le u^{2 - p} \) in \(\R^N \setminus B_1\).
We conclude that
\begin{equation}
\label{eqLimInfSublinearAsymptotics}
\liminf_{\abs{x} \to \infty}  \frac{u^{2 - p}(x)}{I_\alpha (x)} \ge \eps^{-1}\|u\|_p^p
\end{equation}
and the assertion follows from the combination of \eqref{eqLimInfSublinearAsymptotics} and \eqref{eqLimSupSublinearAsymptotics}.
\qed

\subsection{Proof of the existence}\label{s43}

Throughout this section, we assume that either $\frac{N+\alpha}{N}<p<\frac{N+\alpha}{N-2}$ and $q>2$ or $p\geq\frac{N+\alpha}{N-2}$ and $q>\frac{2Np}{N+\alpha}$.
We construct a groundstate of \eqref{eqPeps} by minimising over the Poho\v{z}aev manifold of \eqref{eqPeps}. Similar approach for Choquard's equations with different classes of nonlinearities was recently used in \cite{Li-Tang}, \cite{Gao-Tang-Chen}.

Set
$$\mathscr{P}_{\eps}:=\{u\in \H\setminus\{0\}: \mathcal{P}_{\eps}(u)=0\},$$
where $\mathcal{P}_{\eps}:\H\to \R$ is defined by
$$
\mathcal{P}_{\eps}(u)=\frac{N-2}{2}\int_{\R^N}|\nabla u|^2dx+\frac{\eps N}{2}\int_{\R^N}|u|^2dx+\frac{N}{q}\int_{\R^N}|u|^qdx-\frac{N+\alpha}{2p}\int_{\R^N}(I_{\alpha}*|u|^p)|u|^pdx.
$$
For each $u\in \H\setminus\{0\}$, set
\begin{equation}\label{e-ut}
u_t(x):=u(\tfrac{x}{t}).
\end{equation}
Then
\begin{multline}\label{Peps-non0}
f_u(t):=\mathcal{I}_{\eps}(u_t)\\
=\frac{t^{N-2}}{2}\int_{\R^N}|\nabla u|^2dx+\frac{\eps t^N}{2}\int_{\R^N}|u|^2dx+\frac{t^N}{q}\int_{\R^N}|u|^qdx-\frac{t^{N+\alpha}}{2p}\int_{\R^N}(I_{\alpha}*|u|^p)|u|^pdx.
\end{multline}
Clearly, there exists a unique $t_u>0$ such that $f_u(t_u)=\max\{f_u(t): t>0\}$ and $f'_u(t_u)t_u=0$, which means that $u(x/t_u)\in \mathscr{P}_{\eps}$.
Therefore $\mathscr{P}_{\eps}\neq \emptyset$.

Define $M: \H\to \R$ as
\begin{equation}\label{eqM}
M(u):=\|\nabla u\|_2^2+\|u\|_2^2+\|u\|_q^q.
\end{equation}
Then $M(u)=0$ if and only if $u=0$. Taking into account the definition of $M$ and the norm $\|\cdot\|_\H$, we can check that
\begin{equation}\label{eq33}
2^{-\frac{q}{2}}\|u\|^q_\H\leq M(u)\leq C\|u\|^2_\H\quad \text{if either $M(u)\leq1$ or $\|u\|_\H\leq 1$},
\end{equation}
where $C>0$ is independent of $u\in\H$.

\begin{lemma}\label{lem22}
	Assume that either $\frac{N+\alpha}{N}<p<\frac{N+\alpha}{N-2}$ and $q>2$, or $p\geq \frac{N+\alpha}{N-2}$ and $q>\frac{2Np}{N+\alpha}$.
	Then there exists $C>0$ such that for all $u\in \H$,
	$$
	\int_{\R^N}(I_{\alpha}*|u|^p)|u|^pdx\leq C\max\{M(u)^{\frac{N+\alpha}{N}}, M(u)^{\frac{N+\alpha}{N-2}}\}.
	$$
\end{lemma}

\begin{proof}
	For each $u\in \H\setminus\{0\}$, let $u_t$ be defined in \eqref{e-ut}. If $M(u)\leq 1$ then we set $t=M(u)^{-\frac{1}{N}}\geq 1$, and we have $M(u_t)\leq t^NM(u)=1$. Thus it follows from the HLS inequality, the embedding $\H\hookrightarrow L^{\frac{2Np}{N+\alpha}}(\R^{N})$, and \eqref{eq33} that
	\begin{multline*}
	\int_{\R^N}(I_{\alpha}*|u|^p)|u|^pdx=t^{-(N+\alpha)}\int_{\R^N}(I_{\alpha}*|u_t|^p)|u_t|^pdx\\
	\leq Ct^{-(N+\alpha)}\|u_t\|^{2p}_{\frac{2Np}{N+\alpha}}\leq C_1t^{-(N+\alpha)}\|u_t\|^{2p}_{\H}\leq C_2M(u)^{\frac{N+\alpha}{N}}.
	\end{multline*}
To clarify the last inequality, note that $M(u_t)\leq t^N M(u)=1$.
	Then using $(4.17)$, we obtain
	$$Ct^{-(N+\alpha)}\|u_t\|^{2p}_{\H}\le
	CM(u)^{\frac{N+\alpha}{N}}\big(\|u_t\|^q_{\H}\big)^\frac{2p}{q}\le
	CM(u)^{\frac{N+\alpha}{N}}\big(2^\frac{q}{2}M(u_t)\big)^\frac{2p}{q}
	\le 2^pC\,M(u)^{\frac{N+\alpha}{N}}.$$

	If $M(u)> 1$ then we set $t=M(u)^{-\frac{1}{N-2}}<1$, and we have $M(u_t)\leq t^{N-2}M(u)=1$. Similarly as before, we conclude that
	\begin{multline*}
	\int_{\R^N}(I_{\alpha}*|u|^p)|u|^pdx=t^{-(N+\alpha)}\int_{\R^N}(I_{\alpha}*|u_t|^p)|u_t|^pdx\\
	\leq Ct^{-(N+\alpha)}\|u_t\|^{2p}_{\frac{2Np}{N+\alpha}}\leq C_1t^{-(N+\alpha)}\|u_t\|^{2p}_{\H}\leq C_2M(u)^{\frac{N+\alpha}{N-2}},
	\end{multline*}
	which completes the proof.
\end{proof}

To find a groundstate solution of \eqref{eqPeps}, we prove the existence of a spherically symmetric nontrivial nonegative minimizer of the  minimization problem
\begin{equation}\label{eqCe}
c_\eps=\inf_{u\in \mathscr{P}_{\eps}} \mathcal{I}_{\eps}(u),
\end{equation}
and then show that $\mathscr{P}_{\eps}$ is a natural constraint for $\I_\eps$, i.e.~the minimizer $u_0\in \mathscr{P}_{\eps}$ satisfies $\I_\eps^\prime(u_0)=0$.
Such approach for the local equations goes back at least to \cite{Shatah-85} in the local case and to \cite{Ruiz-06} in the case of nonlocal problems.

We divide the proof of the existence of the groundstate into several steps.

\medskip\noindent
{\bf Step 1.}  $0\notin \partial \mathscr{P}_{\eps}$.

\proof
Indeed, for $u\in \mathscr{P}_{\eps}$, we have, by using the HLS and Sobolev inequalities,
$$\aligned
0=\mathcal{P}_{\eps}(u)
\geq& \min\Big\{\frac{N-2}{2}, \frac{\eps N}{2}, \frac{N}{q}\Big\}M(u)-C\|u\|^{2p}_{\frac{2Np}{N+\alpha}}\\
\geq& \min\Big\{\frac{N-2}{2}, \frac{\eps N}{2}, \frac{N}{q}\Big\}M(u)-C\big(M(u)\big)^{\frac{N+\alpha}{N}},
\endaligned
$$
which means that there exists $C>0$ such that $M(u)\geq C$ for all $u\in \mathscr{P}_{\eps}$.
\qed

\medskip\noindent
{\bf Step 2.} $c_{\eps}=\inf_{u\in \mathscr{P}_{\eps}} \mathcal{I}_{\eps}(u)>0$

\proof
For each $u\in \mathscr{P}_{\eps}$, we have
\begin{equation}\label{eq11}
\mathcal{I}_{\eps}(u)=\frac{1}{N}\|\nabla u\|_2^2+\frac{\alpha}{2pN}\int_{\R^N}(I_{\alpha}*|u|^p)|u|^pdx,
\end{equation}
therefore $c_{\eps}\geq 0$.  If $c_{\eps}=0$, then there exists a sequence $\{u_n\}\subset \mathscr{P}_{\eps}$ such that $\mathcal{I}_{\eps}(u_n)\to 0$, which means that $\|\nabla u_n\|_2^2\to 0$ and $\int_{\R^N}(I_{\alpha}*|u_n|^p)|u_n|^pdx\to 0$.  Recall that $\mathcal{P}_{\eps}(u_n)=0$. Then we conclude that $\|u_n\|_2^2\to 0$ and $\|u_n\|^q_q\to 0$. This implies that $\|u_n\|_\H^2\to 0$, which contradicts to $0\notin \partial \mathscr{P}_{\eps}$.
\qed

\medskip\noindent
{\bf Step 3.} {\em There exists $u_0\in \mathscr{P}_{\eps}$ such that $\mathcal{I}_{\eps}(u_0)=c_\eps$.}

\proof
Since $c_\eps$ is well defined, there exists a sequence $\{u_n\}\subset \mathscr{P}_{\eps}$ such that $\mathcal{I}_{\eps}(u_n)\to c_\eps$. It follows from \eqref{eq11} that both $\{\|\nabla u_n\|_2^2\}$ and $\{\int_{\R^N}(I_{\alpha}*|u_n|^p)|u_n|^pdx\}$ are bounded.  Note that $\P_{\eps}(u_n)=0$. Then we see that $\{\| u_n\|_2^2\}$ and $\{\| u_n\|_q^q\}$ are bounded, and therefore $\{u_n\}$ is bounded in $\H$.

Let $u^*_n$ be the Schwartz spherical rearrangement of $|u_n|$. Then $u_n^*\in \Hrad$, the subspace of $\H$ which consists of all spherically symmetric functions in $\H$, and
$$
\|\nabla u_n\|_2^2\geq \|\nabla u^*_n\|_2^2, \ \| u_n\|_2^2=\| u^*_n\|_2^2, \ \| u_n\|_q^q=\| u^*_n\|_q^q, \ \int_{\R^N}(I_{\alpha}*|u_n|^p)|u_n|^pdx\leq \int_{\R^N}(I_{\alpha}*|u_n^*|^p)|u^*_n|^pdx,
$$
cf. \cite{Lieb-Loss}*{Section 3}.
For each $u^*_n$, there exists a unique $t_n\in(0, 1)$ such that $v_n:=u^*_n(\frac{x}{t_n})\in \mathscr{P}_{\eps}$.  Therefore we obtain that
$$\aligned
\mathcal{I}_{\eps}(u_n)\geq \mathcal{I}_{\eps}(u_n(\tfrac{x}{t_n}))\geq \mathcal{I}_{\eps}(v_n)\geq c_\eps,
\endaligned
$$
which implies that $\{v_n\}$ is also a minimizing sequence for $c_{\eps}$, that is $\mathcal{I}_{\eps}(v_n)\to c_\eps$.
(In fact, we can also prove that $t_n\to 1$.)

Clearly $\{v_n\}\subset \Hrad$ is bounded. Then there exists $v\in \Hrad$ such that $v_n\weakto v$ weakly in $\H$ and $v_n(x)\to v(x)$ for a.e. $x\in\R^N$, by the local compactness of the emebedding $\H\hookrightarrow L^2_{loc}(\R^N)$ on bounded domains.
Using Strauss's $L^s$--bounds with $s=2$ and $s=q^*$, we conclude that
	$$v_n(|x|)\leq U(x):=C\min\big\{|x|^{-N/2}, |x|^{-N/{q^*}}\big\}.$$
	Since $U\in L^s(\R^N)$ for $s\in(2,q^*)$, by the Lebesgue dominated convergence we conclude that for $s\in (2, q^*)$,
	$$
	\lim_{n\to\infty}\int_{\R^N}|v_n|^sdx=\int_{\R^N}|v|^sdx.
	$$
	Note that $q^*>\frac{2Np}{N+\alpha}$ and hence we can always choose $s>\frac{2Np}{N+\alpha}>p$ such that $\{v_n\}$ is bounded
	in $L^s(\R^N)$.
	Then by the nonlocal Brezis-Lieb Lemma with high local integrability \cite{MMS16}*{Proposition 4.7} we conclude that
	$$
	\lim_{n\to\infty}\int_{\R^N}(I_{\alpha}*|v_n|^p)|v_n|^pdx=\int_{\R^N}(I_{\alpha}*|v|^p)|v|^pdx.
	$$
This means that $v\neq 0$, since by Lemma \ref{lem22} the sequence  $\{M(v_n)\}$ has a positive lower bound.
Then there exists a unique $t_0>0$ such that $v(\frac{x}{t_0})\in \mathscr{P}_{\eps}$. By the weakly lower semi-continuity of the norm, we see that
$$
\mathcal{I}_{\eps}(v_n)\geq \mathcal{I}_{\eps}(v_n(\tfrac{x}{t_0}))\geq \mathcal{I}_{\eps}(v(\tfrac{x}{t_0}))\geq c_\eps,
$$
which implies that $\mathcal{I}_{\eps}(v(\frac{x}{t_0}))=c_\eps$.
We conclude this step by taking $u_0(x):=v(\frac{x}{t_0})$.\qed

\smallskip\noindent
{\bf Step 4.} {\em $\P_\eps^\prime(u_0)\neq 0$, where $u_0$ is obtained in Step 3.}

\proof
Arguing by contradiction, we assume that $\P_\eps^\prime(u_0)=0$. Then $u_0$ is a weak solution of the following equation,
\begin{equation}\label{eq12}
-(N-2)\Delta u +\eps Nu-(N+\alpha)(I_{\alpha}*|u|^p)|u|^{p-2}u+N|u|^{q-2}=0\quad\text{in $\R^N$.}
\end{equation}
By Propositions \ref{p-Phz} and \ref{p-reg}, $u_0$ satisfies the Poho\v{z}aev identity
$$
\frac{(N-2)^2}{2}\|\nabla u_0\|_2^2+\frac{\eps N^2}{2}\|u_0\|_2^2-\frac{(N+\alpha)^2}{2p}\int_{\R^N}(I_{\alpha}*|u_0|^p)|u_0|^pdx+\frac{N^2}{q}\|u_0\|_q^q=0.
$$
This, together with $\P_{\eps}(u_0)=0$, implies that
$$
(N-2)\|\nabla u_0\|_2^2+\frac{(N+\alpha)\alpha}{2p}\int_{\R^N}(I_{\alpha}*|u_0|^p)|u_0|^pdx=0,
$$
which contradict $u_0\neq 0$.
\qed

\medskip\noindent
{\bf Step 5.} {\em $\mathcal{I}_{\eps}'(u_0)=0$, i.e., $u_0$ is a weak solution of \eqref{eqPeps}.}

\proof
By the Lagrange multiplier rule, there exists $\mu\in \R$ such that $\mathcal{I}_{\eps}'(u_0)=\mu\P_{\eps}'(u_0)$.  We claim that $\mu=0$.  Indeed, since $\mathcal{I}_{\eps}'(u_0)=\mu\P_{\eps}'(u_0)$, then $u_0$ satisfies in the weak sense the following equation,
$$
-(\mu(N-2)-1)\Delta u+(\mu N-1)\eps u-(\mu(N+\alpha)-1)(I_{\alpha}*|u|^p)|u|^{p-2}u+(\mu N-1)|u|^{q-2}u=0\quad\text{in $\R^N$.}
$$
By Propositions \ref{p-Phz} and \ref{p-reg}, $u_0$ satisfies the Poho\v{z}aev identity
\begin{multline*}
\frac{(\mu(N-2)-1)(N-2)}{2}\|\nabla u_0\|_2^2+\frac{\eps (\mu N-1)N}{2}\|u_0\|_2^2\\
-\frac{(\mu(N+\alpha)-1)(N+\alpha)}{2p}\int_{\R^N}(I_{\alpha}*|u_0|^p)|u_0|^pdx+\frac{(\mu N-1)N}{q}\|u_0\|_q^q=0.
\end{multline*}
By using $\P_{\eps}(u_0)=0$ again, we conclude that
$$
\mu(N-2)\|\nabla u_0\|_2^2+\frac{\mu\alpha(N+\alpha)}{2p}\int_{\R^N}(I_{\alpha}*|u_0|^p)|u_0|^pdx=0,
$$
which means that $\mu=0$. Therefore $\mathcal{I}_{\eps}'(u_0)=0$.
\qed

\section{Existence and properties of groundstates for \eqref{eqP0}}\label{s5}

In this section we study the existence and some qualitative properties of groundstate solutions for the equation
\begin{equation}\tag{$P_0$}\label{eq03+}
-\Delta u-(I_{\alpha}*|u|^p)|u|^{p-2}u+|u|^{q-2}u=0\quad\text{in $\R^N$},
\end{equation}
where $N\ge 3$, $\alpha\in(0,N)$, $p>1$ and $q>2$.
Equation \eqref{eqP0} appears as a {\sl formal} limit of \eqref{eqPeps} with $\eps=0$.
The natural domain for the formal limit energy $\I_0$ which corresponds to \eqref{eqP0} is the space
$$
\D:=D^1(\R^N)\cap L^q(\R^N).
$$
Clearly, $\D$ endowed with the norm
$$\|u\|_{\D}:=\|\nabla u\|_{L^2}+\|u\|_{L^q}$$
is a Banach space, and $\D\hookrightarrow L^q\cap L^{2^*}(\R^{N})$. In particular, $\D\hookrightarrow L^{\frac{2Np}{N+\alpha}}(\R^{N})$ for $\frac{N+\alpha}{N}<p<\frac{N+\alpha}{N-2}$ and $2<q<\frac{2Np}{N+\alpha}$, or $p\geq \frac{N+\alpha}{N-2}$ and $q>\frac{2Np}{N+\alpha}$. However, $\H\subsetneq \D$. Hence \eqref{eqPeps} can not be considered as a small perturbation of \eqref{eqP0}, since the domain of $\I_0$ is strictly bigger than the domain of $\I_\eps$.

If $p\in (\frac{N+\alpha}{N}, \frac{N+\alpha}{N-2})$ and $q\in (2, \frac{2Np}{N+\alpha})$ or $p\geq \frac{N+\alpha}{N-2}$ and $q>\frac{2Np}{N+\alpha}$ then the HLS, Sobolev and interpolation inequalities ensure the control of the nonlocal term by the $L^q$ and $D^1$--norms,
\begin{equation}\label{eq15}\aligned
\int_{\R^N}(I_{\alpha}*|u|^p)|u|^pdx&\leq C\|u\|^{2p}_{\frac{2Np}{N+\alpha}}\leq C\|u\|^{2p\theta}_q\|u\|^{2p(1-\theta)}_{2^*}\leq C\|u\|^{2p\theta}_q\|\nabla u\|^{2p(1-\theta)}_{2}\\ &\leq C(\|u\|^{2p}_q+\|\nabla u\|^{2p}_{2})\leq C\|u\|^{2p}_{\D},
\endaligned
\end{equation}
with a $\theta\in (0,1)$. Then it is standard to check that $\I_0\in C^1(\D, \R)$ and the problem \eqref{eqP0} is variationaly well-posed,
in the sense that weak solutions $u\in\D$ of \eqref{eqPeps} are critical points of $\mathcal{I}_\eps$,
i.e.
$$
\langle\mathcal{I}_0^\prime(u),\varphi\rangle_{\D}=\int_{\R^N}\nabla u\cdot\nabla \varphi dx-\int_{\R^N}(I_{\alpha}*|u|^p)|u|^{p-2}\varphi dx+\int_{\R^N}|u|^{q-2}u\varphi dx=0,
$$
for all $\varphi\in \D$. In particular, weak solutions $u\in\D$ of \eqref{eqP0} satisfy the {\em Nehari identity}
\begin{equation}\label{eq05}
\int_{\R^N}|\nabla u|^2 dx+\int_{\R^N}|u|^q dx-\int_{\R^N}(I_{\alpha}*|u|^p)|u|^p dx=0.
\end{equation}
As in Proposition \ref{p-Phz}, we see that weak solutions $u\in\D\cap L^\frac{2Np}{N+\alpha}(\R^N)\cap W^{1,\frac{2Np}{N+\alpha}}_{loc}(\R^N)\cap W^{2,2}_{loc}(\R^N)$ of \eqref{eqP0} also satisfy the {\em Poho\v{z}aev identity}
\begin{equation}\label{eq06}
	\frac{N-2}{2}\int_{\R^N}|\nabla u|^2dx+\frac{N}{q}\int_{\R^N}|u|^qdx-\frac{N+\alpha}{2p}\int_{\R^N}(I_{\alpha}*|u|^p)|u|^pdx=0.
\end{equation}

We are going to prove the existence of a ground state of $(P_0)$ by minimizing over the Poho\v{z}aev manifold $\P_0$.
This requires apriori additional regularity and some decay properties of the weak solutions.

\begin{lemma}[$L^1$--decay]\label{l-L1-0}
	Let $\frac{N+\alpha}{N}<p<\frac{N+\alpha}{N-2}$ and $2<q<\frac{2Np}{N+\alpha}$.
	If $0\le u\in\D$ is a weak solution of \eqref{eqP0} then $u\in L^1(\R^N)$.
\end{lemma}

\proof
Note that $u\ge 0$ weakly satisfies the inequality
\eqref{eq-sniter-super}.
Then, by the HLS and H\"older inequalities, and by the contraction inequality \ref{e-Ponce-sub}, $u\in L^{\underline{s}_n}(\R^N)$ with $0<\frac{2p-1}{\underline{s}_{n}}-\frac{\alpha}{N}<1$ implies $u\in L^{\underline{s}_{n+1}}(\R^N)$, where
$$\frac{q-1}{\underline{s}_{n+1}}=\frac{2p-1}{\underline{s}_{n}}-\frac{\alpha}{N}.$$
We start the $\underline{s}_n$--iteration with $\underline{s}_0=q<\frac{2Np}{N+\alpha}$.
Then we achieve $\underline{s}_{n+1}\le 1$ after a finite number of steps.
\qed

\begin{proposition}[Regularity]\label{p-reg0}
	Let $\frac{N+\alpha}{N}<p<\frac{N+\alpha}{N-2}$ and $2<q<\frac{2Np}{N+\alpha}$, or $p\ge\frac{N+\alpha}{N-2}$
	and $q>\frac{2Np}{N+\alpha}$.
	If $0\le u\in\D$ is a nontrivial weak solution of \eqref{eqP0} then $u\in C^2(\R^N)$
	and $u(x)>0$ for all $x\in\R^N$.
\end{proposition}

\proof
Since $0\le u\in\D$ we know that $u\in L^q\cap L^{2^*}(\R^N)$.
Assume that  $q\le \frac{Np}{\alpha}$, otherwise we conclude that $u\in L^\infty(\R^N)$ by a modification of the comparison argument of Lemma \ref{l-infty}.

If $p>\frac{N+\alpha}{N-2}$ and $q>\frac{2Np}{N+\alpha}$ we can show that
$u\in L^\infty(\R^N)$ by repeating the same iteration argument as in the proof of Proposition \ref{p-reg}, Step 2(B).

If $\frac{N+\alpha}{N}<p<\frac{N+\alpha}{N-2}$ and $2<q<\frac{2Np}{N+\alpha}$ we know additionally that $u\in L^1\cap L^{2^*}(\R^N)$ by Lemma \ref{l-L1-0}.
Then we can conclude that $u\in L^\infty(\R^N)$ by repeating the iteration argument
in the proof of Proposition \ref{p-reg}, Step 2(A).

Finally, $u\in L^{q^*}\cap L^\infty(\R^N)$ implies $u\in C^2(\R^N)$ by the standard H\"older and Schauder estimates, while positivity of $u(x)$ follows via the weak Harnack inequality, as in the proof of Proposition \ref{p-reg}, Steps 3 and 4.
\qed

Unlike in the case $\eps>0$, for $p>\frac{N+\alpha}{N-2}$ we can not conclude that $u\in L^1(\R^N)$ via a regularity type iteration arguments. In fact, the decay of groundstates of \eqref{eqP0} is more complex.

\begin{proposition}[Decay estimates]
	Assume that either $p<\frac{N+\alpha}{N-2}$ and $q<\frac{2Np}{N+\alpha}$ or $p>\frac{N+\alpha}{N-2}$ and $q>\frac{2Np}{N+\alpha}$.
	Let $u\in\Drad$ be a nontrivial nonnegative weak solution of \eqref{eqP0}.
	Then:
	\begin{itemize}
		\item if \(p < \frac{N+\alpha}{N-2}\) then $u\in L^1(\R^N)$,
		\item if \(p >  \frac{N+\alpha}{N-2}\) then
		$$u_0\gtrsim|x|^{-(N-2)}\quad\text{as $|x|\to\infty$},$$
		and if $p>\max\big\{\frac{N+\alpha}{N-2},\frac23\big(1+\frac{N+\alpha}{N-2}\big)\big\}$ then
		\begin{equation}\label{u0-decay}
		u_0\sim|x|^{-(N-2)}\quad\text{as $|x|\to\infty$}.
		\end{equation}
	\end{itemize}
\end{proposition}

\proof
The case $p<\frac{N+\alpha}{N-2}$ is the content of Lemma \ref{l-L1-0}.

Assume $p>\frac{N+\alpha}{N-2}$. Then $u\in L^{2^*}\cap C^2(\R^N)$.
By the Strauss's radial $L^s$--bounds with $s=2^*$,
\begin{equation}\label{eq-Strauss+}
u(|x|)\leq c(1+|x|)^{-\frac{N-2}{2}}\qquad(x\in\R^N).
\end{equation}
Since $q>2^*$, $u>0$ satisfies
$$-\Delta u+V(x)u>0\quad\text{in $\R^N$},$$
where
$$0<V(x):=u^{q-2}(x)\le c(1+|x|)^{-(2+\delta)}\qquad(x\in\R^N),$$
for some $\delta>0$.
By comparing with an explicit subsolution $c|x|^{-(N-2)}(1+|x|^{-\delta/2})$ in $|x|>1$, we conclude that
\begin{equation}\label{e0-lower}
u(x)\ge c(1+|x|)^{-(N-2)}\qquad(x\in\R^N).
\end{equation}

Assume now that $p\ge \max\big\{\frac{N+\alpha}{N-2},2\big\}$. Using again \eqref{eq-Strauss+} and we conclude that
\begin{equation}\label{eq-Strauss+++}
\big((I_\alpha*u^p)u^{p-2}\big)(x)\leq W(x):=c(1+|x|)^{-(N-2)(p-1)+\alpha}\le c(1+|x|)^{-(2+\delta)}\qquad(x\in\R^N),
\end{equation}
for some $\delta\in(0,1)$.
Therefore, $u>0$ satisfies
$$-\Delta u-W(x)u<0\quad\text{in $\R^N$}.$$
By comparing with an explicit supersolution $c|x|^{-(N-2)}(1-|x|^{-\delta/2})$ in $|x|>2$, we conclude that
$$u(x)\le C(1+|x|)^{-(N-2)}\qquad(x\in\R^N).$$

Next we assume  that $\alpha<N-4$ and $\frac{N+\alpha}{N-2}<p<2$.
Using again \eqref{eq-Strauss+}, the lower bound \eqref{e0-lower} and taking into account that $p<2$, we conclude that
\begin{equation}\label{eq-Strauss++++}
\big((I_\alpha*u^p)u^{p-2}\big)(x)\leq W(x):=c(1+|x|)^{-\frac{N-2}{2}p+\alpha+(N-2)(2-p)}\le c(1+|x|)^{-(2+\delta)}\quad(x\in\R^N),
\end{equation}
where $\delta>0$ provided that $p>\frac23\big(1+\frac{N+\alpha}{N-2}\big)$. Then we conclude as before.
\qed

\begin{proof}[Proof of Theorem \ref{thmP0}]
We assume that either $\frac{N+\alpha}{N}<p<\frac{N+\alpha}{N-2}$ and $2<q<\frac{2Np}{N+\alpha}$ or $p>\frac{N+\alpha}{N-2}$ and $q>\frac{2Np}{N+\alpha}$.
Set
$$\mathscr{P}_{0}:=\{u\in \D\setminus\{0\}: \mathcal{P}_{0}(u)=0\},$$
where $\mathcal{P}_{0}:\D\to \R$ is defined by
$$
\mathcal{P}_{0}(u)=\frac{N-2}{2}\int_{\R^N}|\nabla u|^2dx+\frac{N}{q}\int_{\R^N}|u|^qdx-\frac{N+\alpha}{2p}\int_{\R^N}(I_{\alpha}*|u|^p)|u|^pdx.
$$
As in the case $\eps>0$, it is standard to check that $\mathscr{P}_{0}\neq \emptyset$ (see \eqref{Peps-non0}).
To construct a groundstate solution of $(P_0)$, we prove the existence of a spherically symmetric nontrivial nonegative minimizer of the  minimization problem
\begin{equation}\label{eqC0min}
c_0=\inf_{u\in \mathscr{P}_{0}} \mathcal{I}_{0}(u),
\end{equation}
and then show that $\mathscr{P}_{0}$ is a natural constraint for $\I_0$, i.e.~the minimizer $u_0\in \mathscr{P}_{0}$ satisfies $\I_0^\prime(u_0)=0$. The arguments follow closely the proof of Theorem \ref{thm01}, except that instead of the quantity $M$ defined in \eqref{eqM}, we use $\overline{M}: \D\to \R$ defined by
$$
\overline{M}(u):=\|\nabla u\|_2^2+\|u\|_q^q.
$$
It is easy check that
\begin{equation}\label{eq330}
2^{-\frac{q}{2}}\|u\|^q_\D\leq \overline{M}(u)\leq C\|u\|^2_\D\quad \text{if either $\overline{M}(u)\leq1$ or $\|u\|_\D\leq 1$},
\end{equation}
where $C>0$ is independent of $u\in\D$. Similarly to Lemma \ref{lem22}, we also can prove that there exists $C>0$ such that for all $u\in \D$,
	$$
	\int_{\R^N}(I_{\alpha}*|u|^p)|u|^pdx\leq C\max\{\overline{M}(u)^{\frac{N+\alpha}{N}}, \overline{M}(u)^{\frac{N+\alpha}{N-2}}\},
	$$
which allows to control the nonlocal term.
The remaining arguments follow closely Steps 1--5 in the proof of Theorem \ref{thm01}.
We omit further details.
\end{proof}

\begin{remark}
	An equivalent route to construct a groundstate solution of \eqref{eqP0} is to prove the existence of a minimizer of the problem
	\begin{equation}\label{eq04-1}
	a_0=\inf\left\{\frac{1}{2}\|\nabla w\|_2^2+\frac{1}{q}\|w\|_q^q\;:\; w\in\D,\; \int_{\R^N}(I_{\alpha}*|u|^p)|w|^pdx=2p\right\}.
	\end{equation}
	This does not require apriori regularity or decay properties of the weak solutions.
	It is standard but technical to establish the relation
	$$
	c_0=\tfrac{(\alpha+2)q-2(2p+\alpha)}{(N+\alpha)(q-2)-4p}\left(\tfrac{(N-2)q-2N}{(N+\alpha)(q-2)-4p}\right)^{\frac{(N-2)q-2N}{q-2}}a_0^{\frac{(\alpha+2)q-2(2p+\alpha)}{(N+\alpha)(q-2)-4p}},
	$$
	and to prove that the minimization problems for $a_0$ and $c_0$ are equivalent up to a rescaling.
	Moreover, if $w_0\in\D$ is a minimizer for $a_0$ then
	$$
	-\Delta w_0+|w_0|^{q-2}w_0=\mu(I_{\alpha}*|w_0|^p)|w_0|^{p-2}w_0, \quad x\in \R^N,
	$$
	where $\mu>0$ and
	$
	u_0(x)=\mu^{\frac{2}{4(p-1)-(\alpha+2)(q-2)}}w_0\big(\mu^{\frac{q-2}{4(p-1)-(\alpha+2)(q-2)}} x\big)
	$
	is a solution of \eqref{eqP0}.
\end{remark}

\begin{remark}\label{r-Phz0}
	Combining \eqref{eq05} and \eqref{eq06}, we conclude that
	$$
	\frac{(N-2)p-(N+\alpha)}{2p}\|\nabla u\|_2^2=\frac{(N+\alpha)q-2Np}{2pq}\|u\|_q^q,
	$$
	which implies that
	\eqref{eqP0} has no nontrivial solutions $u\in\D\cap L^\frac{2Np}{N+\alpha}(\R^N)\cap W^{1,\frac{2Np}{N+\alpha}}_{loc}(\R^N)\cap W^{2,2}_{loc}(\R^N)$ either if $p<\frac{N+\alpha}{N-2}$ and $q\geq\frac{2Np}{N+\alpha}$ or if $p>\frac{N+\alpha}{N-2}$ and $q\leq  \frac{2Np}{N+\alpha}$.  Moreover, if $p=\frac{N+\alpha}{N-2}$ then \eqref{eqP0} has no nontrivial solution for  $q\neq\frac{2Np}{N+\alpha}$. This confirms that the existence assumptions of Theorem \ref{thmP0} on $p$ and $q$ are optimal, with one exception of the {\em double--critical} case $p=\frac{N+\alpha}{N-2}$ and $q=\frac{2Np}{N+\alpha}$.
\end{remark}

\begin{remark}\label{r-2crit}
In the double critical case $p=\frac{N+\alpha}{N-2}$ and $q=\frac{2Np}{N+\alpha}$ the Poho\v zaev argument does not lead to the nonexistence. In fact, it is not difficult to check (cf.~\cite{Minbo-Du}*{Lemma 1.1}) that the Emden--Fowler solution $U_*$ defined in \eqref{eEF} satisfies
	$$
	-\Delta U_*+U_*^{2^*-1}=2\mathcal{S}_*^{-\frac{\alpha}{2}}\mathcal{C}_\alpha^{-1}\big(I_{\alpha}*U_*^\frac{N+\alpha}{N-2}\big)U_*^\frac{\alpha+2}{N-2}, \quad x\in \R^N.
	$$
	and the ``Lagrange multiplier'' can not be scaled out due to the scale invariance of the equation.
	It is an interesting open problem to show that a rescaling of $U_*$ is a minimizer of the variational problem \eqref{eq04-1} in the double--critical case.
\end{remark}

\section{The Thomas--Fermi groundstate}\label{sTF}
To simplify notation we set in this section $m:=q/2$. Denote
$$E(\rho):=\int_{\R^N}|\rho|^m dx+\int_{\R^N}|\rho|dx,\qquad D_{\alpha}(\rho)=\int_{\R^N}(I_{\alpha}*|\rho|)|\rho|\, dx,$$
and
$$\mathcal{A}_1=\big\{0\le \rho\in L^1\cap L^{m}(\R^N): D_{\alpha}(\rho)=1\big\}.$$
We first establish the following.

\begin{proposition}\label{prpTF}
	Let $m>\frac{2N}{N+\alpha}$.
	Then the minimization problem
	$$s_{T\!F}=\inf_{\mathcal A_1}E$$
	admits a nonnegative spherically--symmetric nonincreasing minimizer $\rho_*\in L^1\cap L^m(\R^N)$.
\end{proposition}

\proof
Let $\{\rho_n\}\subset \mathcal{A}_1$ be a minimizing sequence for $s_{T\!F}$.
Let $\{\rho_n^*\}$ be the sequence of symmetric--decreasing rearrangements of $\{\rho_n\}$. Then
$$E(\rho_n)=E(\rho_n^*),\qquad D_{\alpha}(\rho_n^*)\ge D_{\alpha}(\rho_n)=1,$$
see \cite{Lieb-Loss}*{Section 3}. Set $\bar{\rho}_n^*(x):=\rho_n^*\big((D_{\alpha}(\rho_n^*))^{\frac{1}{N+\alpha}}x\big)$. Then $\bar{\rho}_n^*\in\mathcal{A}_1$ and as $n\to \infty$,
$$
s_{T\!F}\leq E(\bar{\rho}_n^*)=\big(D_{\alpha}(\rho_n^*)\big)^{-\frac{N}{N+\alpha}}E(\rho_n^*)\leq E(\rho_n)\to s_{T\!F},
$$
which means that $\{\bar{\rho}_n^*\} \subset \mathcal{A}_1$ is a minimizing sequence for $s_{T\!F}$. Moreover, $\{\bar{\rho}_n^*\}$ is bounded in $L^1(\R^N)$ and $L^{m}(\R^N)$.
By Strauss's lemma,
\begin{equation}\label{rhoStrauss}
\bar{\rho}_n^*(|x|)\leq U(|x|):=C\min\{|x|^{-N}, |x|^{-\frac{N}{m}}\}\quad\text{for all $|x|>0$}.
\end{equation}
By Helly's selection theorem for monotone functions, there exists a nonnegative spherically--symmetric nonincreasing function $\rho_*(|x|)\le U(|x|)$ such that, up to a subsequence,
$$\bar{\rho}_n^*(x)\to \rho_*(x)\quad\text{a.e. in $\R^N$ as $n\to \infty$}.$$
Since $U\in L^s(\R^N)$ for all $s\in(1,m)$, by the Lebesgue's dominated convergence we conclude that for every $s\in (1, m)$,
$$
\lim_{n\to\infty}\int_{\R^N}|\bar{\rho}^*_n|^sdx=\int_{\R^N}|\rho_*|^sdx.
$$
Note that $\frac{2N}{N+\alpha}\in (1, m)$. Then by the nonlocal Brezis-Lieb lemma \cite{MMS16}*{Proposition 4.7} we conclude that
$$
\lim_{n\to\infty}D_\alpha(\bar{\rho}^*_n)=D_\alpha(\rho_*)=1,
$$
which implies that $0\neq \rho_*\in \mathcal{A}_1$.  Therefore, using the standard Brezis-Lieb lemma,
$$\aligned
s_{T\!F}=\lim_{n\to\infty}E(\bar{\rho}^*_n)&=\lim_{n\to\infty}\big(\|\rho_*\|_1+\|\bar{\rho}_n^*-\rho_*\|_1\big)+\big(\|\rho_*\|^{m}_{m}+\|\bar{\rho}_n^*-\rho_*\|^{m}_{m}\big)\\
&=E(\rho_*)+E(\bar{\rho}_n^*-\rho_*)\geq s_{T\!F}\big(1+D_{\alpha}(\bar{\rho}_n^*-\rho^*)\big)^{\frac{N}{N+\alpha}}
\geq s_{T\!F},
\endaligned
$$
that is, $E(\rho_*)=s_{T\!F}$.
Moreover, $\bar{\rho}_n^*\to \rho_*$ strongly in $L^1(\R^N)$ and $L^{m}(\R^N)$.
\qed

\begin{proof}[Proof of Theorem \ref{thmTF}]
	Let $\rho_*\in L^1\cap L^m(\R^N)$ be a minimizer for $s_{T\!F}$, as constructed in Proposition 6.1.
	It is standard to show that $\rho_*$ satisfies
	\begin{equation}\label{TF-rho+}
	\aligned
	m\rho_*^{m-1}+1&=\lambda I_{\alpha}*\rho_*\quad\text{a.e. in $\mathrm{supp}(\rho)$},\\
	m\rho_*^{m-1}+1&\ge\lambda I_{\alpha}*\rho_*\quad\text{a.e. in $\R^N$}
	\endaligned
	\end{equation}
	for a Lagrange multiplier $\lambda\in\R$. The proof can be adapted from \cite{Auchmuty-Beals} or \cite{Carrillo-NA}*{Proposition 3.6} and we omit it here.
	Since $\rho_*\ge 0$, \eqref{TF-rho+} is equivalent to the Thomas--Fermi equation
	\begin{equation}\label{TF-rho++}
	m\rho_*^{m-1}=(\lambda I_{\alpha}*\rho_*-1)_+\quad\text{a.e. in $\R^N$}.
	\end{equation}
	
	Testing \eqref{TF-rho++} against $\rho_*$ we conclude that
	\begin{equation}\label{TF-EL+}
	m\|\rho_*\|_m^m+\|\rho_*\|_1=\lambda,
	\end{equation}
	or taking into account the definition of $s_{T\!F}$ we conclude that
	\begin{equation}\label{TF-Lagrange}
	\lambda=s_{T\!F}+(m-1)\|\rho_*\|_m^m.
	\end{equation}
	To prove the virial identity \eqref{TF-Virial} consider the rescaling $\rho_t(x)=t^{-\frac{N+\alpha}{2}}\rho_*(x/t)$.
	Then
	$$D_\alpha(\rho_t)=1,\qquad E(\rho_t)=t^{N-m\frac{N+\alpha}{2}}\|\rho_*\|_m^m+t^{\frac{N-\alpha}{2}}\|\rho_*\|_1,$$
	and
	$$\frac{d}{dt}E(\rho_t){\big|_{t=1}}=\big(N-m\frac{N+\alpha}{2}\big)\|\rho_*\|_m^m+\frac{N-\alpha}{2}\|\rho_*\|_1=0,$$
	since $m>\frac{2N}{N+\alpha}$ and hence $t=1$ is the minimum of the differentiable function $E(\rho_t)$, which implies that
	\begin{equation}\label{eq-TF-1}
	\frac{m(N+\alpha)-2N}{N-\alpha}\|\rho_*\|_m^m=\|\rho_*\|_1.
	\end{equation}
	Therefore we have
	$$
	s_{T\!F}=\|\rho_*\|_m^m+\|\rho_*\|_1=\frac{(m-1)(N+\alpha)}{N-\alpha}\|\rho_*\|_m^m,
	$$
	$$
	\lambda=m\|\rho_*\|_m^m+\|\rho_*\|_1=\Big(m+\frac{m(N+\alpha)-2N}{N-\alpha}\Big)\|\rho_*\|_m^m=\frac{2N(m-1)}{N-\alpha}\|\rho_*\|_m^m=\frac{2N}{N+\alpha}s_{T\!F}.
	$$
	It follows from \eqref{TF-rho++} that $\rho_*$ satisfies
	\begin{equation}\label{eqTF123}
	m\rho_*^{m-1}=\Big(\frac{2N}{N+\alpha}s_{T\!F} I_{\alpha}*\rho_*-1\Big)_+\quad\text{a.e. in $\R^N$},
	\end{equation}
	and the virial identity
	$$
	m\|\rho_*\|_m^m+\|\rho_*\|_1=s_{T\!F}\frac{2N}{N+\alpha}.
	$$

To prove the $L^\infty$--bound on $\rho_*$, in view of the Strauss' radial bound \eqref{rhoStrauss} we only need to show that $\rho_*$ is bounded near the origin.
Observe that
\begin{equation}\label{TF-rho+++}
0\le \rho_*^{m-1}\le I_\alpha*\rho_*\quad\text{a.e. in $\R^N$}.
\end{equation}
If $m>\frac{N}{\alpha}$ then $I_\alpha*\rho_*\in C^{0,\alpha-\frac{N}{m}}(\R^N)$ and hence $\rho_*\in L^\infty(\R^N)$ in view of \eqref{TF-rho++}.

If $m<\frac{N}{\alpha}$ we employ an $L^s$--iteration of the same structure as in our proof of Proposition \ref{p-reg}, Step 2(B).
Indeed, by the HLS and H\"older inequalities applied to \eqref{TF-rho+++}, $\rho_*\in L^{\overline{s}_n}(\R^N)$ with $0<\frac{1}{\overline{s}_{n}}-\frac{\alpha}{N}<1$ implies $\rho_*\in L^{\overline{s}_{n+1}}(\R^N)$, where
$$\frac{m-1}{\overline{s}_{n+1}}=\frac{1}{\overline{s}_{n}}-\frac{\alpha}{N}.$$
We start the $\overline{s}_n$--iteration with $\overline{s}_0=m$. If $m\ge 2$ we achieve $\overline{s}_{n+1}>\frac{N}{\alpha}$ after a finite number of steps.
If $m<2$ we note that $\overline{s}_0=m>\frac{2N}{N+\alpha}$  by the assumption. Then we again achieve $\overline{s}_{n+1}>\frac{N}{\alpha}$ after a finite number of steps. (Or if $\overline{s}_{n+1}=\frac{N}{\alpha}$ we readjust $\overline{s}_0$.)
\smallskip

The compact support property is standard (cf. \cite{Carrillo-NA}*{Corollary 3.8}). We sketch the argument for completeness. Indeed, since $\rho_*\in L^1(\R^N)$ is a nonnegative radially nonincreasing function, it is known that for any $\alpha\in(0,N)$,
$$I_\alpha*\rho_*(|x|)=C I_\alpha(|x|)(1+o(1))\quad\text{as $|x|\to\infty$},$$
(and $C=\|\rho_*\|_1$ if $\alpha>1$), cf. \cite{Duoandikoetxea}*{Corollary 2.3}.
This is incompatible with \eqref{TF-rho+}, unless $\rho_*$ has a compact support. Since $\rho_*$ is nonincreasing we also conclude that $\mathrm{supp}(\rho_*)$ is a connected set and hence must be a ball of radius $R_*>0$ (and if $\alpha>1$ then  $R_*\lesssim \|\rho_*\|_1^\frac{1}{N-\alpha}$).
\smallskip

If $\alpha>\big(\frac{m-2}{m-1})_+$ then the H\"older regularity $\rho_*\in C^{0,\gamma}(\R^N)$ with $\gamma=\min\{1,\frac{1}{m-1}\}$ follows exactly as in \cite{Carrillo-CalcVar}*{Theorem 8}. We only note that the iteration steps $(3.26)$, $(3.27)$ in \cite{Carrillo-CalcVar}*{p.127} remain valid for any $m\le 2$, as soon as $\rho_*\in L^\infty(\R^N)$, which is ensured by our assumption $m>\frac{2N}{N+\alpha}$.
 If $m>2$ and $\alpha\le\big(\frac{m-2}{m-1})_+$ then $\rho_*\in C^{0,\gamma}(\R^N)$ for any $\gamma<\frac{2}{m-2}$ by the same argument as in \cite{Carrillo-CalcVar}*{Remark 2}.
 Further, $\rho_*\in C^\infty(B_{R_*})$ can be deduced as in \cite{Carrillo-CalcVar}*{Theorem 10}.

 Finally, keeping in mind that $q=2m$, the function
 \begin{equation*}
 v_0(x)=\Big(\frac{q}{2}\Big)^\frac{1}{q-2}\sqrt{\rho_*\Big((\tfrac{q}{2})^\frac{2}{\alpha(q-2)}\big(\tfrac{2N s_{T\!F}}{N+\alpha}\big)^{-1/\alpha} x\Big)}
 \end{equation*}
 is a groundstate of the Thomas--Fermi equation in the form
 \begin{equation*}\tag{$T\!F$}
 v-(I_{\alpha}*|v|^2)v+|v|^{q-2}v=0\quad\text{in $\R^N$},
 \end{equation*}
 by direct scaling computation and in view of the properties of $\rho_*$.
\end{proof}

\begin{remark}
	In \cite[Proposition 5.16]{Volzone} the authors establish the existence of a unique bounded nonnegative radially nonincreasing solution to the Euler--Lagrange equation \eqref{eqTF123} in the range $\alpha\in(0,2)$ and $m\in\big(\frac{2N}{N+\alpha},m_c\big)$ (for $\alpha=2$ the existence of a radial solution is classical, see e.g. \cite[Theorem 5.1]{Lions-81}, while the uniqueness follows from \cite[Lemma 5]{Flucher}). These existence results do not include an explicit variational characterisation of the solution in terms of $s_{T\!F}$. However once the existence of a minimizer for $s_{T\!F}$ is established (see Proposition \ref{prpTF}), solutions constructed for $\alpha\in(0,2]$ in \citelist{\cite{Volzone}\cite{Lions-81}} coincide with the minimizer for $s_{T\!F}$ in view of the uniqueness. 
\end{remark}

\section{Asymptotic profiles: non-critical regimes}\label{s6}

In this section we prove the convergence of rescaled groundstates $u_\eps$ to the limit profiles in the three {\em noncritical} regimes.

\subsection{Formal limit $(P_0)$}
Throughout this section we assume that $p>\frac{N+\alpha}{N-2}$ and $q>\frac{2Np}{N+\alpha}$,
or $\frac{N+\alpha}{N}<p<\frac{N+\alpha}{N-2}$ and $2<q<\frac{2Np}{N+\alpha}$.

Let $u_\eps$ be the positive spherically symmetric groundstate solution of \eqref{eqPeps} constructed in Theorem \ref{thm01},
and $c_\eps=\I_\eps(u_\eps)>0$ denotes the corresponding energy level, defined in \eqref{eqCe}.
We are going to show that $u_\eps$ converges to the constructed in Theorem \ref{thmP0} positive spherically symmetric groundstate $u_0$ of the formal limit equation \eqref{eqP0}, which has the energy $c_0=\I_0(u_0)>0$, as defined in \eqref{eqC0}.

Below we present the proof only in the supercritical case $p>\frac{N+\alpha}{N-2}$ and $q>\frac{2Np}{N+\alpha}$.
The subcritical case $\frac{N+\alpha}{N}<p<\frac{N+\alpha}{N-2}$ and $2<q<\frac{2Np}{N+\alpha}$ follows the same lines but easier, because in this case $u_0\in L^1(\R^N)$. The proof in the supercritical case relies on the decay estimate \eqref{u0-decay}, which needs an additional restriction $p>\max\big\{\frac{N+\alpha}{N-2},\frac23\big(1+\frac{N+\alpha}{N-2}\big)\big\}$.

\begin{lemma}\label{lem1101}
	Let $p>\max\big\{\frac{N+\alpha}{N-2},\frac23\big(1+\frac{N+\alpha}{N-2}\big)\big\}$ and $q>\frac{2Np}{N+\alpha}$.
	Then $0<c_{\eps}-c_0\to 0$ as $\eps\to 0$.
\end{lemma}

\begin{proof}
	First, we use $u_\eps$ with $\eps>0$ as a test function for $\P_0$. We obtain
    $$\mathcal{P}_0(u_{\eps})=\mathcal{P}_{\eps}(u_{\eps})-\frac{N\eps}{2}\|u_{\eps}\|_2^2=-\frac{N\eps}{2}\|u_{\eps}\|_2^2<0.$$
	Hence there exists a unique $t_{\eps}\in (0, 1)$ such that $u_{\eps}(x/t_{\eps})\in \mathscr{P}_0$, and we have
	\begin{equation}\label{eq1102}
	c_0\leq \mathcal{I}_{0}(u_{\eps}(x/t_{\eps}))=\frac{t_{\eps}^{N-2}}{N}\|\nabla u_{\eps}\|_2^2+\frac{\alpha t_{\eps}^{N+\alpha}}{2Np}\int_{\R^N}(I_{\alpha}*|u_{\eps}|^p)|u_{\eps}|^pdx< \I_{\eps}(u_{\eps}(x))=c_{\eps},
	\end{equation}
	which means $c_0<c_{\eps}$.
	
	To show that $c_{\eps}\to c_0$ as $\eps\to 0$ we shall use $u_0$ as a test function for $\P_\eps$. According to \eqref{u0-decay}, $u_0\in L^2(\R^N)$ iff $N\ge 5$. The dimensions $N=3,4$ require a separate consideration.
	\medskip
	
	\noindent{\bf Case $N\geq 5$.}
	Since $\mathcal{P}_{\eps}(u_0)=\frac{\eps}{2}\|u_0\|_2^2>0$, there exists $\overline{t}_{\eps}>1$ such that $u_0(x/\overline{t}_{\eps})\in \mathscr{P}_{\eps}$, i.e.,
	$$
	\frac{(N-2)\overline{t}_{\eps}^{N-2}}{2}\|\nabla u_0\|_2^2+\frac{N\eps \overline{t}_{\eps}^{N}}{2}\| u_0\|_2^2+\frac{N\overline{t}_{\eps}^{N}}{q}\|u_0\|_q^q=\frac{(N+\alpha) \overline{t}_{\eps}^{N+\alpha}}{2p}\int_{\R^N}(I_{\alpha}*|u_0|^p)|u_{\eps}|^p dx.
	$$
	This, combined with \eqref{eq06} and $u_0\in L^2(\R^N)$, implies that
	$$
	\frac{(N+\alpha) (\overline{t}_{\eps}^{\alpha}-1)}{2p}\int_{\R^N}(I_{\alpha}*|u_0|^p)|u_0|^pdx-\frac{(N-2)(\overline{t}_{\eps}^{-2}-1)}{2}\|\nabla u_0\|_2^2=\frac{N\eps}{2}\|\nabla u_0\|_2^2\to 0.
	$$
	Therefore, $\overline{t}_{\eps}\to 1$ as $\eps\to 0$.  Moreover,
	$$
	\overline{t}_{\eps}\leq 1+C\eps,
	$$
	where $C>0$ is independent of $\eps$.
	Thus we have
	$$\aligned
	c_{\eps}\leq\mathcal{I}_{\eps}(u_0(x/\overline{t}_{\eps}))\leq&\mathcal{I}_{0}(u_0)
	+C(\overline{t}_{\eps}^{N}-1)+\frac{N\eps \overline{t}_{\eps}^{N}}{2}\| u_0\|_2^2\\
	\leq &c_0+C\eps.
	\endaligned
	$$
	This, together with \eqref{eq1102}, means that $c_{\eps}-c_0\to 0$ as $\eps\to 0$.
	\smallskip
	
	To consider the case $N=3, 4$, given $R\gg 1$, we introduce a cut-off function $\eta_{R}\in C_c^{\infty}(\R^N)$ such that
	$\eta_R(r)=1$ for $|r|<R$, $0<\eta_R(r)<1$ for $R<|r|<2R$, $\eta_R(r)=0$ for $|r|\geq 2R$ and $|\eta_R'(r)|\leq R/2$.
	It is standard to compute (cf. \cite{S96}*{Theorem 2.1}),
	\begin{align}\label{eq1103}
	\int_{\R^N}|\nabla(\eta_Ru_0)|^2dx&=\|\nabla u_0\|_2^2+\O(R^{-(N-2)}),\\
		\label{eq1106}
	\|\eta_Ru_0\|_2^2&=
	\left\{\aligned
	&\O(\ln R),\quad N=4,&\\
	&\O(R), \quad N=3, &\\
	\endaligned\right.\\
	\label{eq1105}
	\|\eta_Ru_0\|_q^q&=\|u_0\|_q^q-\O(R^{N-q(N-2)}),\\
	\label{eq1104}
	\int_{\R^N}(I_{\alpha}*|\eta_Ru_0|^{p})|\eta_Ru_0|^{p}dx&=
	\int_{\R^N}(I_{\alpha}*|u_0|^{p})|u_0|^{p}dx+\O(R^{\alpha-p(N-2)}).
	\end{align}
	We will use $\eta_R u_0$ with a suitable choice of $R=R(\eps)$ as a family of test functions for $\P_\eps$.
	\medskip
	
	\noindent{\bf Case $N=4$.}  By \eqref{eq1103}, \eqref{eq1104}, \eqref{eq1105} and \eqref{eq1106}, for $R\gg 1$ we have
	$$
	\mathcal{P}_{\eps}(\eta_{R}u_0)=\mathcal{P}_{0}(u_0)+\O(R^{-(N-2)})+\frac{N\eps}{2}\O(\ln R)-\O(R^{N-q(N-2)})-\O(R^{\alpha-p(N-2)})>0.
	$$
	Set $R=\eps^{-1}$. Then for each $\eps>0$ small, there exists $\widetilde{t}_{\eps}>1$ such that $\mathcal{P}_{\eps}(\eta_{R}(x/\widetilde{t}_{\eps})u_0(x/\widetilde{t}_{\eps}))=0$.
	Similarly to the case $N\geq 5$, we can show that
	$\widetilde{t}_{\eps}\to 1$ as $\eps\to 0$ and
	$$
	\widetilde{t}_{\eps}\leq 1+C\eps.
	$$
	We conclude that $c_{\eps}\leq c_0+C\eps$.
	\medskip
	
	\noindent{\bf Case $N=3$.}  By \eqref{eq1103}, \eqref{eq1104}, \eqref{eq1105} and \eqref{eq1106}, for $R\gg 1$ we have
	$$
	\mathcal{P}_{\eps}(\eta_{R}u_0)=\mathcal{P}_{0}(u_0)+\O(R^{-(N-2)})+\frac{N\eps}{2}\O(R)-\O(R^{N-q(N-2)})-\O(R^{\alpha-p(N-2)})>0.
	$$
	Set $R=\eps^{-1/2}$. Then for each $\eps>0$ small, there exists $\widehat{t}_{\eps}>1$ such that $\mathcal{P}_{\eps}(\eta_{R}(x/\widehat{t}_{\eps})u_0(x/\widehat{t}_{\eps}))=0$ and
	$\widehat{t}_{\eps}\to 1$ as $\eps\to 0$. We conclude as in the previous case.
	\end{proof}

\begin{corollary}\label{cor-bound0}
	Let $p>\max\big\{\frac{N+\alpha}{N-2},\frac23\big(1+\frac{N+\alpha}{N-2}\big)\big\}$ and $q>\frac{2Np}{N+\alpha}$.
	Then the quantities
	$$\|\nabla u_\eps\|_2^2,\quad \eps\|u_\eps\|_2^2,\quad \|u_{\eps}\|_q^q,
	\quad \int_{\R^N}(I_{\alpha}*|u_\eps|^p)|u_\eps|^p dx,$$
	are uniformly bounded as $\eps\to 0$.
\end{corollary}

\begin{proof}
	From Poho\v zaev and Nehari identities for $\mathcal{I}_{\eps}$ and Lemma \ref{lem1101} we have
	$$
	c_0+o(1)=c_{\eps}=\mathcal{I}_{\eps}(u_\eps(x))=\frac{1}{N}\|\nabla u_\eps\|_2^2+\frac{\alpha}{2Np}\int_{\R^N}(I_{\alpha}*|u_\eps|^p)|u_\eps|^pdx,
	$$
	$$
	\frac{N(p-1)}{2Np}\int_{\R^N}(I_{\alpha}*|u_{\eps}|^p)|u_{\eps}|^{p}dx=c_{\eps}+\frac{q-2}{2q}\|u_{\eps}\|_q^q,
	$$
	$$
	\frac{2N-q(N-2)}{2q}\|\nabla u_{\eps}\|_2^2+\frac{N(2-q)}{2q}\eps\|u_{\eps}\|_2^2+\frac{(N+\alpha)q-2Np}{2pq}\int_{\R^N}(I_{\alpha}*|u_{\eps}|^p)|u_{\eps}|^{p}dx=0.
	\qedhere$$
\end{proof}

\begin{lemma}
	Let $p>\max\big\{\frac{N+\alpha}{N-2},\frac23\big(1+\frac{N+\alpha}{N-2}\big)\big\}$ and $q>\frac{2Np}{N+\alpha}$.
	Then $\eps\|u_{\eps}\|_2^2\to 0$ as $\eps\to 0$.
\end{lemma}

\begin{proof}
	Lemma~\ref{lem1101} implies that there exists a unique $t_{\eps}\in (0, 1)$ such that $u_{\eps}(x/t_{\eps})\in \mathscr{P}_0$.  Indeed, assume that $t_{\eps}\to t_0<1$ as $\eps\to 0$. Then by \eqref{eq1102} we have, as $\eps\to 0$,
	\begin{multline}\label{teto0}
	c_0\leq \mathcal{I}_{0}(u_{\eps}(x/t_{\eps}))=\frac{t_{\eps}^{N-2}}{2}\|\nabla u_{\eps}\|_2^2+\frac{\alpha t_{\eps}^{N+\alpha}}{2Np}\int_{\R^N}(I_{\alpha}*|u_{\eps}|^p)|u_{\eps}|^pdx\\
	\leq  t_{\eps}^{N-2}\mathcal{I}_{\eps}(u_{\eps})=t_{\eps}^{N-2}c_{\eps}
	\to t_0^{N-2} c_0<c_0,
	\end{multline}
	which is a contradiction.  Therefore $t_{\eps}\to 1$ as $\eps\to 0$. Using $\mathcal{P}_0(u_{\eps}(x/t_{\eps}))=0$ again, we see that
	$$\aligned
	0=&\frac{(N-2)t^{N-2}_{\eps}}{2}\|\nabla u_{\eps}\|_2^2+\frac{Nt^{N}_{\eps}}{q}\|u_{\eps}\|_q^q-
	\frac{(N+\alpha)t^{N+\alpha}_{\eps}}{2p}\int_{\R^N}(I_{\alpha}*|u_{\eps}|^p)|u_{\eps}|^pdx\\
	=&\mathcal{P}_{\eps}(u_{\eps})-\frac{N\eps t^{N}_{\eps}}{2}\|u_{\eps}\|_2^2+\frac{(N-2)(t^{N-2}_{\eps}-1)}{2}\|\nabla u_{\eps}\|_2^2+\frac{N(t^{N}_{\eps}-1)}{q}\|u_{\eps}\|_q^q\\
	&-\frac{(N+\alpha)(t^{N+\alpha}_{\eps}-1)}{2p}\int_{\R^N}(I_{\alpha}*|u_{\eps}|^p)|u_{\eps}|^pdx.
	\endaligned
	$$
	This implies that $\eps\|u_{\eps}\|_2^2\to 0$ as $\eps\to 0$, since $\mathcal{P}_{\eps}(u_{\eps})=0$.
\end{proof}

\begin{proof}[Proof of Theorem~\ref{thm02-lim} (case $p>\max\big\{\frac{N+\alpha}{N-2},\frac23\big(1+\frac{N+\alpha}{N-2}\big)\big\}$ and $q>\frac{2Np}{N+\alpha}$)]

	From Corollary \ref{cor-bound0} and \eqref{teto0}, we see that $\{u_{\eps}(x/t_{\eps})\}$ is a minimizing sequence for $c_0$
which is bounded in $D^1_{rad}\cap L^q(\R^N)$.
Then there exists $\overline{w}_0\in D^{1}_{rad}\cap L^q(\R^N)$ such that
$$
u_{\eps}(x/t_{\eps})\rightharpoonup \overline{w}_0 \quad \text{in}\ D^{1}_{rad}(\R^N) \quad\text{and}\quad u_{\eps}(x/t_{\eps})\to \overline{w}_0 \quad \text{a.e. in}\ \R^N,
$$
by the local compactness of the emebedding $D^1(\R^N)\hookrightarrow L^2_{loc}(\R^N)$ on bounded domains.
Using Strauss' radial $L^s$--bounds with $s=2^*$ and $s=q$, we conclude that
$$u_{\eps}(x/t_{\eps})\leq U(x):=C\min\big\{|x|^{-N/2^*}, |x|^{-N/{q}}\big\}.$$
Similarly to Step 3 in Section~\ref{s4}, using Lebesgue dominated convergence and nonlocal Brezis-Lieb Lemma \cite{MMS16}*{Proposition 4.7},
we can show that $u_{\eps}(x/t_{\eps})\to \overline{w}_0$ in $D^{1}\cap L^q(\R^N)$ and $\overline{w}_0$ is a groundstate  solution of $(P_0)$.
\end{proof}

\subsection{1st rescaling: Choquard limit}
Throughout this section we assume that $\frac{N+\alpha}{N}<p<\frac{N+\alpha}{N-2}$ and $q>2\frac{2p+\alpha}{2+\alpha}$.
The energy corresponding to the 1st rescaling
$$
v(x):=\eps^{-\frac{\alpha+2}{4(p-1)}}u\Big(\frac{x}{\sqrt{\eps}}\Big),
$$
is given by
\begin{equation}\label{eq26}
\mathcal{I}_{\eps}^{(1)}(v):=\frac{1}{2}\int_{\R^N}|\nabla v|^2+|v|^2dx-\frac{1}{2p}\int_{\R^N}(I_{\alpha}*|v|^p)|v|^pdx+\frac{\eps^\frac{q(2+\alpha)-2(2p+\alpha)}{4(p-1)}}{q}\int_{\R^N}|v|^q dx,
\end{equation}
and the corresponding Poho\v zaev functional is
\begin{multline}\label{eq28}
\mathcal{P}_{\eps}^{(1)}(v):=
\frac{N-2}{2}\int_{\R^N}|\nabla v|^2dx+\frac{ N}{2}\int_{\R^N}|v|^2dx-\frac{N+\alpha}{2p}\int_{\R^N}(I_{\alpha}*|v|^p)|v|^pdx\\
+\frac{N\eps^{\frac{q(2+\alpha)-2(2p+\alpha)}{4(p-1)}}}{q}\int_{\R^N}|v|^qdx.
\end{multline}
Note that $\I_\eps(u)=\eps^\frac{(N+\alpha)-p(N-2)}{2(p-1)}\I_\eps^{(1)}(v)$ and consider the rescaled minimization problem
\begin{equation}\label{eq30}
c^{(1)}_\eps:=\inf_{ v\in \mathscr{P}_{\eps}^{(1)}}\mathcal{I}_{\eps}^{(1)}(v),
\end{equation}
where $\mathscr{P}_{\eps}^{(1)}:=\{v\in H^1(\R^N)\setminus\{0\}: \mathcal{P}_{\eps}^{(1)}(v)=0\}$ is the Poh\v zaev manifold of \eqref{eqCeps}. When $\eps=0$, we formally obtain
\begin{equation}\label{eq3001}
c^{(1)}_0:=\inf_{ v\in \mathcal{\mathscr{P}}_{0}^{(1)}}\mathcal{I}_{0}^{(1)}(v).
\end{equation}
It is known that $c_0^{(1)}>0$ and $c_0^{(1)}$ admits a minimizer $v_0\in H^1(\R^N)$, which is a groundstate of the Choquard equation \eqref{eqC} characterised in Theorem \ref{thmC} (see \cite{MS13}*{Propositions 2.1 and 2.2}).

Let $u_\eps$ be the positive spherically symmetric groundstate solution of \eqref{eqPeps} constructed in Theorem \ref{thm01}.
Then the rescaled groundstate
$$
v_{\eps}(x):=\eps^{-\frac{\alpha+2}{4(p-1)}}u_{\eps}\Big(\frac{x}{\sqrt{\eps}}\Big),
$$
is a groundstate solution of \eqref{eqCeps}, i.e. $\I_\eps^{(1)}(v_\eps)=c_\eps^{(1)}$.
We are going to show that $v_\eps$ converges to the groundstate $v_0$ of the Choquard equation \eqref{eqC}.

\begin{lemma}\label{lem1102}
$0<c_{\eps}^{(1)}-c_0^{(1)}\to 0$ as $\eps\to 0$.
\end{lemma}

\proof
 Clearly, $\mathcal{P}_{\eps}^{(1)}(v_{\eps})=0$ implies that $\mathcal{P}_{0}^{(1)}(v_{\eps})<0$.  Let $w_{\eps, t}(x)=v_{\eps}(\frac{x}{t})$, then
$$
\mathcal{P}_{0}^{(1)}(w_{\eps,t})=\frac{(N-2)t^{N-2}}{2}\|\nabla v_{\eps}\|_2^2+\frac{ Nt^N}{2}\|v_{\eps}\|^2_2-\frac{(N+\alpha)t^{N+\alpha}}{2p}\int_{\R^N}(I_{\alpha}*|v_{\eps}|^p)|v_{\eps}|^pdx
$$
has a unique maximum and $\mathcal{P}_{0}^{(1)}(w_{\eps, 1})<0$, thus  there exists $t_{v_{\eps}}\in (0,1)$ such that $w_{\eps, t_{v_{\eps}}}\in \mathscr{P}_{0}^{(1)}$.  Therefore we have
$$\aligned
c_0^{(1)}\leq \mathcal{I}_{0}^{(1)}(w_{\eps, t_{v_{\eps}}})=&\frac{t^{N-2}_{v_{\eps}}}{N}\|\nabla v_{\eps}\|_2^2+\frac{\alpha}{2Np}t^{N+\alpha}_{v_{\eps}}\int_{\R^N}(I_{\alpha}*|v_{\eps}|^p)|v_{\eps}|^{p}\\
<&
\frac{1}{N}\|\nabla v_{\eps}\|_2^2+\frac{\alpha}{2Np}\int_{\R^N}(I_{\alpha}*|v_{\eps}|^p)|v_{\eps}|^{p}\\
=&\mathcal{I}_{\eps}^{(1)}(v_{\eps})=c_{\eps}^{(1)}.
\endaligned
$$

On the other hand, let $v_0\in\mathscr{P}^{(1)}_0$ be a radially symmetric ground state of \eqref{eqC}, that is $\I_0^{(1)}(v_0)=c_0^{(1)}$.
Then $\mathcal{P}^{(1)}_{\eps}(v_0)>0$ and there exists $t_{v_0}(\eps)>1$ such that $v_0(x/t_{v_0}(\eps))\in \mathscr{P}^{(1)}_{\eps}$.
This implies that $t_{v_0}(\eps)$ is bounded, up to subsequence, we assume that $t_{v_0}(\eps)\to t_{v_0}(0)$.  Recall that $\mathcal{P}^{(1)}_0(v_0)=0$, we can conclude that $t_{v_0}(0)=1$.
Set $\kappa:=\frac{q(2+\alpha)-2(2p+\alpha)}{4(p-1)}$.
We obtain
$$\aligned
\frac{N+\alpha}{2p}t^{\alpha}_{v_0}(\eps)\int_{\R^N}(I_{\alpha}*|v_0|^p)|v_0|^{p}dx=& \frac{N-2}{2}\|\nabla v_0\|_2^2t^{-2}_{v_0}(\eps)+\frac{N}{2}\|v_0\|_2^2+\frac{N}{q}\|v_0\|_q^q\eps^\kappa\\
\leq& \frac{N-2}{2}\|\nabla v_0\|_2^2+\frac{N}{2}\|v_0\|_2^2+\frac{N}{q}\|v_0\|_q^q\eps^\kappa\\
=& \frac{N+\alpha}{2p}\int_{\R^N}(I_{\alpha}*|v_0|^p)|v_0|^{p}dx+ \frac{N}{q}\|v_0\|_q^q\eps^\kappa,
\endaligned
$$
which means that $t^{\alpha}_{v_0}(\eps)\leq 1+C\eps^{\kappa}$, or equivalently
$$
t_{v_0}(\eps)\leq (1+C\eps^{\kappa})^{1/\alpha}\leq 1+C\eps^{\kappa}.
$$
Therefore
$$\aligned
c_{\eps}^{(1)}\leq \mathcal{I}^{(1)}_{\eps}(v_0(\frac{x}{t_{v_0}(\eps)}))=&\frac{t^{N-2}_{v_0}}{N}\|\nabla v_0\|_2^2(\eps)+\frac{\alpha t^{N+\alpha}_{v_0}(\eps)}{2Np}\int_{\R^N}(I_{\alpha}*|v_{0}|^p)|v_{0}|^{p}dx \\
=&c_0^{(1)}+\frac{t^{N-2}_{v_0}(\eps)-1}{N}\|\nabla v_0\|_2^2+\frac{\alpha(t^{N+\alpha}_{v_0}(\eps)-1)}{2Np}\int_{\R^N}(I_{\alpha}*|u|^p)|u|^{p}dx \\
\leq &c_0^{(1)}+C\eps^{\kappa}.
\endaligned
$$
It follows that $c_{\eps}^{(1)}\to c_0^{(1)}$ as $\eps\to 0$.
\qed

\begin{corollary}\label{cor-bound1}
The quantities
$$\|\nabla v_\eps\|_2^2,\quad \|v_\eps\|_2^2,\quad \eps^{\kappa}\|v_{\eps}\|_q^q,
\quad \int_{\R^N}(I_{\alpha}*|v_\eps|^p)|v_\eps|^p dx,$$
are uniformly bounded as $\eps\to 0$.
\end{corollary}

\begin{proof}
Follows from Poho\v zaev and Nehari identities for $\mathcal{I}_{\eps}^{(1)}$ and Lemma \ref{lem1102}, as in Corollary \ref{cor-bound0}.
\end{proof}

\begin{proof}[Proof of Theorem~\ref{thmC-lim}]
Since
$$
\mathcal{P}^{(1)}_{0}(v_\varepsilon)=\mathcal{P}^{(1)}_{\varepsilon}(v_\varepsilon)-\frac{N}{q}\varepsilon^{\kappa}\|v_{\varepsilon}\|_q^q=-\frac{N}{q}\varepsilon^{\kappa}\|v_{\varepsilon}\|_q^q<0,
$$
there exists $t_{\eps}\in (0,1)$ such that $\mathcal{P}^{(1)}_{0}(v_\eps(x/t_{\eps}))=0$. Using Lemma~\ref{lem1102}, as $\eps\to 0$ we obtain
\begin{equation*}\label{teto1}
c_0^{(1)}\leq \mathcal{I}^{(1)}_{0}(v_{\eps}(x/t_{\eps}))=\frac{t^{N-2}_{\eps}}{2}\|\nabla v_{\eps}\|_2^2+
\frac{\alpha}{2Np}t^{N+\alpha}_{\eps}\int_{\R^N}(I_{\alpha}*|v_{\eps}|^p)|v_{\eps}|^{p}\leq \mathcal{I}^{(1)}_{\eps}(v_{\eps})=c_{\eps}^{(1)}\to c_0^{(1)}.
\end{equation*}
Thus $\{v_\eps(x/t_{\eps})\}\subset \mathscr{P}^{(1)}_{0}$ is a minimizing sequence for $c_0^{(1)}$ and we may assume that $t_{\eps}\to t_0\in (0, 1]$.
(If $t_{\eps}\to 0$ then $0<c_0\leq \I^{(1)}_0(v_{\varepsilon}(x/t_{\varepsilon}))\to 0$, which is a contradiction.) From Corollary \ref{cor-bound1} we see that $\{v_{\eps}(x/t_{\eps})\}$ is bounded in $H^1_{rad}\cap L^q(\R^N)$.
Then similarly to the argument in Step 3 of Section~\ref{s4},
using Strauss' radial bounds, Lebesgue dominated convergence and nonlocal Brezis-Lieb Lemma \cite{MMS16}*{Proposition 4.7}, we conclude that
$v_\eps(x)\to \overline{v}_0$ in $D^1\cap L^q(\R^N)$ and $\overline{v}_0$ is a groundstate solution of \eqref{eqCeps}.
\end{proof}

\begin{remark}
We claim that in fact $t_0=1$. Indeed, if $t_0<1$ then there exists $\delta>0$ such that for all $\eps\in (0, \delta)$, it holds $t_{\eps}<\frac{t_0+1}{2}$. Therefore,
$$\aligned
c_0^{(1)}\leq \mathcal{I}^{(1)}_{0}(v_{\eps}(x/t_{\eps}))=&\frac{t^{N-2}_{\eps}}{2}\|\nabla v_{\eps}\|_2^2+
\frac{\alpha}{2Np}t^{N+\alpha}_{\eps}\int_{\R^N}(I_{\alpha}*|v_{\eps}|^p)|v_{\eps}|^{p}\\
\leq&\Big(\frac{t_0+1}{2}\Big)^{N-2}\mathcal{I}^{(1)}_{\eps}(v_{\eps})
=\Big(\frac{t_0+1}{2}\Big)^{N-2}c_{\eps}^{(1)}
\to \Big(\frac{t_0+1}{2}\Big)^{N-2}c_0^{(1)},
\endaligned
$$
as $\eps\to0$, which is a contradiction.
In particular, since
$$
c_0^{(1)}\leq \mathcal{I}^{(1)}_{0}(v_{\varepsilon}(x/t_{\varepsilon}))=\mathcal{I}^{(1)}_{\varepsilon}(v_{\varepsilon}(x/t_{\varepsilon}))-\frac{t^N_{\varepsilon}}{q}\varepsilon^{\kappa}\|v_{\varepsilon}\|_q^q
\leq c_{\varepsilon}^{(1)}-\frac{t^N_{\varepsilon}}{q}\varepsilon^{\kappa}\|v_{\varepsilon}\|_q^q,
$$
we conclude that $\varepsilon^{\kappa}\|v_{\varepsilon}\|_q^q\to 0$ as $\varepsilon\to 0$.
\end{remark}

\subsection{2nd rescaling: Thomas--Fermi limit for $\alpha=2$}\label{s-TF}
Throughout this section we assume that $N\le 5$, $\alpha=2$, $p=2$ and $\frac{4N}{N+2}<q<3$.
The energy corresponding to the 2nd rescaling
$$
v(x):=\eps^{-\frac{1}{q-2}}u\big(\eps^{-\frac{4-q}{2(q-2)}}x\big)
$$
is given by
\begin{multline}\label{eq27}
\mathcal{I}_{\eps}^{(2)}(v):=\frac{\eps^{\frac{2(3-q)}{q-2}}}{2}\int_{\R^N}|\nabla v|^2dx+\frac12\int_{\R^N}|v|^2dx-\frac{1}{4}\int_{\R^N}(I_{2}*|v|^2)|v|^2dx+\frac{1}{q}\int_{\R^N}|v|^qdx,
\end{multline}
and the corresponding Poho\v zaev functional is
\begin{multline}\label{eq29}
\mathcal{P}_{\eps}^{(2)}(v):=\frac{(N-2)\eps^{\frac{2(3-q)}{q-2}}}{2}\int_{\R^N}|\nabla v|^2dx+\\
\frac{ N}{2}\int_{\R^N}|v|^2dx-\frac{N+2}{4}\int_{\R^N}(I_{2}*|v|^2)|v|^2 dx+\frac{N}{q}\int_{\R^N}|v|^qdx.
\end{multline}
Note that $\I_\eps(u)=\eps^{\frac{q(N+2)-4N}{2(q-2)}}\I_\eps^{(2)}(v)$
and consider the rescaled minimization problem
\begin{equation}\label{eq31}
c^{(2)}_\eps=\inf_{v\in \mathscr{P}_{\eps}^{(2)}} \mathcal{I}_{\eps}^{(2)}(v),
\end{equation}
where $\mathscr{P}_{\eps}^{(2)}:=\big\{u\neq 0: \mathcal{P}_{\eps}^{(2)}(v)=0\big\}$
is the Poho\v zaev manifold of \eqref{eqTFeps}.
When $\eps=0$, we formally obtain
\begin{equation}\label{eq3002}
c^{(2)}_0:=\inf_{ v\in \mathscr{P}_{0}^{(2)}}\mathcal{I}_{0}^{(2)}(v).
\end{equation}
By using an appropriate rescaling, it is standard to see that the minimization problem for $c_0^{(2)}$ is equivalent to the minimization problem $s_{T\!F}$, defined in \eqref{eq-sTF}.
In particular, $c_0^{(2)}>0$ and $c_0^{(2)}$ admits a minimizer $v_0\in L^2\cap L^q(\R^N)$, which is a groundstate of the Thomas--Fermi equation \eqref{eqTF} characterised in \eqref{TF-resc-min-rho-1} of Theorem \ref{thmTF}.

Let $u_\eps$ be the positive spherically symmetric groundstate solution of \eqref{eqPeps} constructed in Theorem \ref{thm01}.
It is clear that the rescaled groundstate
$$
v_{\eps}(x):=\eps^{-\frac{1}{q-2}}u_\eps\big(\eps^{-\frac{4-q}{2(q-2)}}x\big)
$$
is a groundstate solution of \eqref{eqTFeps}, i.e. $\I_\eps^{(2)}(v_\eps)=c_\eps^{(2)}$.
We are going to show  that $v_\eps$ converges to a groundstate of the Thomas--Fermi equation \eqref{eqTF}, characterised in Theorem \ref{thmTF}.

Before we do this, we deduce a two--sided estimate on the ``boundary behaviour'' of the
the nonnegative radially symmetric Thomas--Fermi minimizer $\rho_*$, constructed in Theorem \ref{thmTF}.
Recall that $\mathrm{supp}(\rho_*)=\bar B_{R_*}$ for some $R_*>0$ and $\rho_*$ is $C^\infty$ inside the support. Moreover, since we assume that $p=2$, $\alpha=2$ (and denote $m=q/2$), we see that $\rho_*^{m-1}\in C^{0,1}(\R^N)$ and
$\rho_*\in C^{0, \gamma}(\R^N)$, where $\gamma=\min\{1,\frac{1}{m-1}\}$. In particular, $\rho_*\in H^1_0(B_{R_*})$ and we can
apply $-\Delta$ to the Euler--Lagrange equation \eqref{TF-rho} considered in $B_{R_*}$, to obtain
\begin{equation}\label{TF-rho-2}
	-\Delta\big(m\rho_*^{m-1}\big)=-\Delta\Big(s_{T\!F}\tfrac{2N}{N+2} I_{2}*\rho_*-1\Big)=s_{T\!F}\tfrac{2N}{N+2}\rho_*\ge 0\quad\text{in $B_{R_*}$}.
\end{equation}
We conclude that $\rho_*^{m-1}$ is superharmonic in $B_{R_*}$ and by the boundary Hopf lemma
\begin{equation}\label{rho-0-low}
	\rho_*^{m-1}\ge c(R_*-|x|)\quad\text{in $B_{R_*}$.}
\end{equation}
Hence we deduce a two--sided bound
\begin{equation}\label{rho-0-lowup}
	c(R_*-|x|)^\frac{1}{m-1}\le \rho_*\le C(R_*-|x|)^\gamma \quad\text{in $B_{R_*}$.}
\end{equation}
Similar estimates should be available for $\alpha\neq 0$, at least under the assumption $\alpha<2$.
We will study this in the forthcoming paper \cite{TF}.

\begin{lemma}\label{lem1201}
	$0<c^{(2)}_{\eps}-c^{(2)}_0\to 0$ as $\eps\to 0$.
\end{lemma}
\begin{proof}
	Let $v_{\eps}\in \mathscr{P}^{(2)}_{\eps}$ be a solution of $(P_{\eps})$ with $\mathcal{I}^{(2)}_{\eps}(v_{\eps})=c^{(2)}_{\eps}$ and $\eps>0$. Then $$\mathcal{P}^{(2)}_0(v_{\eps})=\mathcal{P}^{(2)}_{\eps}(v_{\eps})-\frac{(N-2)\eps^{\nu}}{2}\|\nabla v_{\eps}\|_2^2=-\frac{(N-2)\eps^{\nu}}{2}\|\nabla v_{\eps}\|_2^2<0,$$
	where we denoted $\nu:=\frac{2(3-q)}{q-2}$.
	Let $w_{\eps, t}(x)=v_{\eps}(\frac{x}{t})$, then we obtain that
	$$
	\mathcal{P}_{0}^{(2)}(w_{\eps,t})=\frac{Nt^N}{2}\|v_{\eps}\|^2_2+\frac{Nt^N}{q}\|v_{\eps}\|^q_q
	-\frac{(N+2)t^{N+2}}{4}\int_{\R^N}(I_{2}*|v_{\eps}|^2)|v_{\eps}|^2 dx
	$$
	has a unique maximum and $\mathcal{P}_{0}^{(2)}(w_{\eps, 1})<0$. Thus there exists $t_{\eps}\in (0,1)$ such that $w_{\eps, t_{\eps}}\in \mathscr{P}_{0}^{(2)}$.  Therefore we have
	\begin{equation}\label{eq1202}\aligned
	c_0^{(2)}\leq \mathcal{I}_{0}^{(2)}(w_{\eps, t_{\eps}})=&\frac{1}{2N}t^{N+2}_{\eps}\int_{\R^N}(I_{2}*|v_{\eps}|^2)|v_{\eps}|^{2}\\
	<&\frac{1}{2N}\int_{\R^N}(I_{2}*|v_{\eps}|^2)|v_{\eps}|^{2}
	=\mathcal{I}_{\eps}^{(2)}(v_{\eps})
	=c_{\eps}^{(2)},
	\endaligned
	\end{equation}
	so $c^{(2)}_0<c^{(2)}_{\eps}$. We are going to prove that $c^{(2)}_{\eps}\to c^{(2)}_0$ as $\eps\to 0$.
	
	Let $v_0\in \mathscr{P}^{(2)}_{0}$ be a groundstate \eqref{TF-resc-min-rho-1} of the Thomas--Fermi equation \eqref{eqTF}, as constructed in Theorem \ref{thmTF}. From \eqref{TF-resc-min-rho-1} and \eqref{rho-0-lowup}
	we conclude that $v_0^{q-2}\in C^{0,1}(\R^N)$ and
	\begin{equation}\label{rho-0-lowup-v0}
	c(R_*-|x|)^\frac{1}{q-2}\le v_0\le C(R_*-|x|)^{\gamma_0} \quad\text{in $B_{R_*}$,}
	\end{equation}
	where $R_*$ is the support radius of $v_0$ and
	$$\gamma_0:=\min\big\{\tfrac12,\tfrac1{q-2}\big\}.$$
	Note that if $q\ge 4$ then $v_0\not\in D^1(\R^N)$ because of the singularity of the gradient on the boundary of the support, even if $v_0^2$ is Lipschitz.

	Given $n\gg 1$, we introduce the cut-off function $\eta_{n}\in C_c^{\infty}(\R^N)$ such that
	$\eta_n(x)=1$ for $|x|\leq R_*-\frac{1}{n}$, $0<\eta_n(x)<1$ for $R_*-\frac{1}{n}<|x|\leq R_*-\frac{1}{2n}$, $\eta_n(x)=0$ for $|x|\geq R_*-\frac{1}{2n}$. Furthermore, $|\eta_n'(x)|\leq 4n$ and $|\eta_n'(x)|\geq \frac{n}{2}$ for $R_*-\frac{4}{5n}<|x|<R_*-\frac{3}{5n}$.  Set
	$$v_n(x):=\eta_n(x)v_0(x).$$
	It is elementary to obtain the estimates
	\begin{equation}\label{eq1204}
	\int_{\R^N}(I_{2}*|v_n|^{2})|v_n|^{2}dx=
	\int_{\R^N}(I_{2}*|v_0|^{2})|v_0|^{2}dx+\O(\tfrac{1}{n}),
	\end{equation}
	\begin{equation}\label{eq1205}
	\|v_n\|_q^q=\|v_0\|_q^q+\O(\tfrac{1}{n}),
	\end{equation}
	\begin{equation}\label{eq1206}
	\|v_n\|_2^2=\|v_0\|_2^2+\O(\tfrac{1}{n}).
	\end{equation}

To estimate the gradient term, note that since $v_0^{q-2}$ is Lipschitz on $\R^N$ and smooth inside the support, we have $v_0^{q-3}|\nabla v_0|\in L^\infty(\R^N)$ and then it follows from \eqref{rho-0-lowup-v0} that
$$
|\nabla v_0|\le C(R_*-|x|)^{\gamma_0(3-q)}.
$$
Then we have
\begin{equation}\label{eq030}
\aligned
\int_{\R^N}\eta_n^2|\nabla v_0|^2dx\leq&\int_{|x|\leq R_*-\frac{1}{2n}}|\nabla v_0|^2dx
\leq C\int_{0}^{R_*-\frac{1}{2n}}(R_*-r)^{2\gamma_0(3-q)}dr\\ \leq& \left\{\aligned
	&C,&\quad 2<q<4,\\
	&C(1+\ln n), &\quad q=4,\\
    &C(1+n^{\frac{q-4}{q-2}}), &\quad q>4,\\
	\endaligned\right.\\
\endaligned
\end{equation}
On the other hand, using the right hand side of \eqref{rho-0-lowup-v0}, we have
\begin{equation}\label{eq04}\aligned
	\int_{\R^N}|\nabla\eta_n|^2v_0^2 dx=&\int_{R_*-\frac{1}{n}\leq|x|\leq R_*-\frac{1}{2n}}|\nabla\eta_n|^2v_0^2 dx
	\leq C n^2\int_{R_*-\frac{1}{n}}^{R_*-\frac{1}{2n}}(R_*-r)^{2\gamma_0}r^{N-1}dr\\
\leq &Cn^{1-2\gamma_0}
\leq \left\{\aligned
	&C,&\quad 2<q\leq 4,\\
    &Cn^{\frac{q-4}{q-2}},&\quad q>4.\\
	\endaligned\right.\\
\endaligned
\end{equation}
Therefore
\begin{equation}\label{eq1203}
\int_{\R^N}|\nabla v_n|^2dx\leq 2\int_{\R^N}|\nabla\eta_n|^2v_0^2 dx+2\int_{\R^N}\eta_n^2|\nabla v_0|^2dx
\leq \left\{\aligned
	&C,&\quad 2<q<4,\\
	&C(1+\ln n), &\quad q=4,\\
    &C(1+n^{\frac{q-4}{q-2}}), &\quad q>4,\\
	\endaligned\right.\\
\end{equation}
Recall that $N\le 5$, $\alpha=2$, $p=2$ and $\frac{4N}{N+2}<q<3$ in this section, and hence $\gamma_0=1/2$
($q\ge 3$ in \eqref{eq1203} is needed to study the case $\eps\to\infty$).
Then 
$$
\int_{\R^N}|\nabla v_n|^2dx\leq 2\int_{\R^N}|\nabla\eta_n|^2v_0^2 dx+2\int_{\R^N}\eta_n^2|\nabla v_0|^2dx\leq C.
$$

Set $n=\varepsilon^{-\frac{3}{2}\nu}$. Then for $\varepsilon>0$ small enough, we have
	$$\aligned
	\mathcal{P}^{(2)}_{\varepsilon}(v_{\varepsilon})=&\mathcal{P}^{(2)}_{0}(v_{\varepsilon})+\frac{(N-2)\varepsilon^{\nu}}{2}\|\nabla v_{\varepsilon}\|_2^2\\
	\geq&
	\mathcal{P}^{(2)}_{0}(v_{0})+\frac{(N-2)\varepsilon^{\nu}}{2}\Big(\int_{|x|\leq\frac{3}{4}R_*}|\nabla v_0|^2dx\Big)-C\varepsilon^{\frac{3}{2}\nu}>0,
	\endaligned
	$$
	and there exists $t_{\varepsilon}>1$ such that $\mathcal{P}^{(2)}_{\varepsilon}(v_{\varepsilon}(x/t_{\varepsilon}))=0$. This implies that
	$$\aligned
	\frac{N+2}{4}t^{2}_{\varepsilon}\int_{\R^N}(I_{2}*|v_{\varepsilon}|^2)|v_{\varepsilon}|^{2}dx=& \frac{N-2}{2}\varepsilon^{\nu}\|\nabla v_{\varepsilon}\|_2^2t^{-2}_{\varepsilon}(\varepsilon)+\frac{N}{2}\|v_{\varepsilon}\|_2^2+\frac{N}{q}\|v_{\varepsilon}\|_q^q\\
	\leq& \frac{N-2}{2}\varepsilon^{\nu}\|\nabla v_{\varepsilon}\|_2^2+\frac{N}{2}\|v_{\varepsilon}\|_2^2+\frac{N}{q}\|v_{\varepsilon}\|_q^q\\
	=& \frac{N-2}{2}\varepsilon^{\nu}\|\nabla v_{\varepsilon}\|_2^2+\frac{N+2}{4}\int_{\R^N}(I_{2}*|v_0|^2)|v_0|^{2}dx+C\varepsilon^{\frac{3}{2}\nu}.
	\endaligned
	$$	
	It follows from \eqref{eq1203} and \eqref{eq1204} that $t_{\varepsilon}\to 1$ as $\varepsilon\to 0$. Moreover, we also have
	$$
	t_{\varepsilon}\leq 1+C\varepsilon^{\nu}.
	$$
	Therefore, we have
	$$
	c^{(2)}_{\varepsilon}\leq \mathcal{I}^{(2)}_{\varepsilon}(v_\varepsilon(x/t_{\varepsilon}))=\mathcal{I}^{(2)}_{0}(v_0)
	+\frac{t^{N-2}_{\varepsilon}\varepsilon^{\nu}}{2}\|\nabla v_{\varepsilon}\|_2^2+C(t^{N}_{\varepsilon}-1)-C(t^{N+2}_{\varepsilon}-1)\leq c_0^{(2)}+C\varepsilon^{\nu},
	$$
	which means that $c^{(2)}_{\varepsilon}\to c_0^{(2)}$ as $\varepsilon\to 0$.
\end{proof}

\begin{lemma}\label{lem1202}
	$\eps^{\frac{2(3-q)}{q-2}}\|\nabla v_{\eps}\|_2^2\to 0$ as $\eps\to 0$.
\end{lemma}
\begin{proof}
	According to Lemma~\ref{lem1201}, there exists a unique $t_{\eps}\in (0, 1)$ such that $v_{\eps}(x/t_{\eps})\in \mathscr{P}^{(2)}_0$.  Now we claim that $t_{\eps}\to 1$ as $\eps\to 0$.  If not, we assume that $t_{\eps}\to t_0<1$ as as $\eps\to 0$, then by \eqref{eq1102}, we have, as $\eps\to 0$,
	$$\aligned
	c^{(2)}_0\leq \mathcal{I}^{(2)}_{0}(v_{\eps}(x/t_{\eps}))=&\frac{t_{\eps}^{N+2}}{2N}\int_{\R^N}(I_{2}*|v_{\eps}|^2)|v_{\eps}|^2 dx\\
	\leq & t_{\eps}^{N+2}\mathcal{I}^{(2)}_{\eps}(v_{\eps})=t_{\eps}^{N+2}c^{(2)}_{\eps}
	\to t_0^{N+2} c^{(2)}_0<c_0,
	\endaligned
	$$
	this is a contradiction.  Therefore our claim holds, i.e., $t_{\eps}\to 1$ as $\eps\to 0$.  Using $\mathcal{P}^{(2)}_0(v_{\eps}(x/t_{\eps}))=0$ again, we see that
	$$\aligned
	0=&\frac{Nt^{N}_{\eps}}{2}\|v_{\eps}\|_2^2+\frac{Nt^{N}_{\eps}}{q}\|v_{\eps}\|_q^q-
	\frac{(N+2)t^{N+2}_{\eps}}{4}\int_{\R^N}(I_{2}*|v_{\eps}|^2)|v_{\eps}|^2 dx\\
	=&\mathcal{P}^{(2)}_{\eps}(v_{\eps})-\frac{(N-2)\eps^{\nu}}{2}\|\nabla v_{\eps}\|_2^2+\frac{N (t^{N}_{\eps}-1)}{2}\|v_{\eps}\|_2^2+\frac{N(t^{N}_{\eps}-1)}{q}\|v_{\eps}\|_q^q\\
	&-\frac{(N+2)(t^{N+2}_{\eps}-1)}{4}\int_{\R^N}(I_{2}*|v_{\eps}|^2)|v_{\eps}|^2 dx,
	\endaligned
	$$
	which implies that $\eps^{\nu}\|\nabla v_{\eps}\|_2^2\to 0$ as $\eps\to 0$ since $\mathcal{P}^{(2)}_{\eps}(v_{\eps})=0$.
\end{proof}

\noindent{\bf Proof of Theorem \ref{t-TF-0}.}  By Lemmas~\ref{lem1201} and \ref{lem1202}, we see that $\{v_{\eps}(x/t_{\eps})\}$ is a bounded minimizing sequence of $c^{(2)}_0$.
Then similarly to the arguments in the proof of Proposition \ref{prpTF} and Theorem \ref{thmTF}, there exists $\overline{v}_0\in L^{2}(\R^N)\cap L^{q}(\R^N)$ such that
$$
v_{\eps}(x/t_{\eps})\to \overline{v}_0 \quad \text{in}\ L^{2}(\R^N)\cap L^{q}(\R^N)
$$
and $\overline{v}_0$ is a weak solution of \eqref{eqTF}.

\section{Asymptotic profiles: critical regimes}\label{s7}

\subsection{Critical Choquard case}
\label{sect-C-crit}

Throughout this sub-section we assume that $p=\frac{N+\alpha}{N-2}$ and $q>2^*=\frac{2N}{N-2}$.
Consider the minimization problem
$$
\mathcal{S}_{HL}=\inf_{u\in D^{1}(\R^N)\setminus\{0\}}\frac{\|\nabla u\|_2^2}{\left(\int_{\R^N}(I_{\alpha}*|u|^{\frac{N+\alpha}{N-2}})|u|^{\frac{N+\alpha}{N-2}}dx\right)^\frac{N-2}{N+\alpha}}
$$
defined in \eqref{eSHL}. Combining Sobolev inequality \eqref{Sobolev} and HLS inequality \eqref{HLS},
\begin{equation}\label{HLSS}
\|\nabla u\|_2^2\ge\mathcal{S}_*\|u\|^2_{2^*}\ge \mathcal{S}_*\mathcal{C}_\alpha^{-\frac{N-2}{N+\alpha}}
\Big(\int_{\R^N}(I_{\alpha}*|u|^\frac{N+\alpha}{N-2})|u|^\frac{N+\alpha}{N-2}dx\Big)^{\frac{N-2}{N+\alpha}},
\end{equation}
hence $\mathcal{S}_{HL}\ge \mathcal{S}_*\mathcal{C}_\alpha^{-\frac{N-2}{N+\alpha}}$.
It is not difficult to check (cf.~\cite{Minbo-Du}*{Lemma 1.1}) that $\mathcal{S}_{HL}=\mathcal{S}_*\mathcal{C}_\alpha^{-\frac{N-2}{N+\alpha}}$ and $\mathcal{S}_{HL}$ is achieved by the function
\begin{equation}\label{eqV1}
V(x)=\big(\mathcal{S}_*^\alpha\mathcal{C}_\alpha^2\big)^{\frac{2-N}{4(\alpha+2)}}U_*(x),
\end{equation}
and the family of rescalings
$$V_{\lambda}(x):=\lambda^{-\frac{N-2}{2}}V(x/\lambda)\qquad(\lambda>0),$$
here $U_*$ is the Emden--Fowler solution in \eqref{eEF}.
Up to a rescaling, $V_\lambda$ is a solution of the critical Choquard equation \eqref{eqC0}
and satisfies
$$\|\nabla V_{\lambda}\|_2^2=\|\nabla V\|_2^2=\int_{\R^N}(I_{\alpha}*|V|^{\frac{N+\alpha}{N-2}})|V|^{\frac{N+\alpha}{N-2}}dx=\int_{\R^N}(I_{\alpha}*|V_{\lambda}|^{\frac{N+\alpha}{N-2}})|V_{\lambda}|^{\frac{N+\alpha}{N-2}}dx=\mathcal{S}_{HL}^{\frac{N+\alpha}{\alpha+2}}.
$$
The energy functional associated to \eqref{eqC0} is
$$
J(u)=\frac{1}{2}\|\nabla u\|_2^2-\frac{N-2}{2(N+\alpha)}\int_{\R^N}(I_{\alpha}*|u|^{\frac{N+\alpha}{N-2}})|u|^{\frac{N+\alpha}{N-2}}dx.
$$
We define
$$
c_{HL}=\inf_{u\in \mathscr{P}_{HL}}J(u)=\inf_{u\in D^{1}(\R^N)\setminus \{0\}}\max_{t>0}J(u(x/t)),
$$
where
$$
\mathscr{P}_{HL}=\Big\{u\in D^{1}(\R^N)\setminus \{0\}: \mathcal{P}_{HL}(u):=\|\nabla u\|_2^2-\int_{\R^N}(I_{\alpha}*|u|^{\frac{N+\alpha}{N-2}})|u|^{\frac{N+\alpha}{N-2}}dx=0\Big\}.
$$
By a simple calculation, we see that $c_{HL}=\frac{\alpha+2}{2(N+\alpha)}\mathcal{S}_{HL}^{\frac{N+\alpha}{\alpha+2}}$.

\begin{lemma}\label{lem631}
Let $\sigma_{\eps}=c_{\eps}-c_{HL}$. Then as $\eps\to 0$, we have
$$
0<\sigma_{\eps}\lesssim
\left\{\aligned &\eps^{\frac{q(N-2)-2N}{(q-2)(N-2)}},\quad
N\geq 5,&\\
&\Big(\eps\ln\frac{1}{\eps}\Big)^{\frac{q-4}{q-2}},\quad
N=4,&\\
&\eps^{\frac{q-6}{2(q-4)}}, \quad N=3. &\\
\endaligned\right.
$$
\end{lemma}
\begin{proof}
Note that $u_{\eps}\in \mathscr{P}_{\eps}$ is a solution of \eqref{eqPeps} with $\mathcal{I}_{\eps}(u_{\eps})=c_{\eps}$, then $\mathcal{P}_{HL}(u_{\eps})<0$, thus there exists $t_{\eps}\in (0, 1)$ such that $u_{\eps}(x/t_{\eps})\in \mathscr{P}_{HL}$ and we have
$$
c_{HL}\leq J(u_{\eps}(\frac{x}{t_{\eps}}))=\frac{(\alpha+2)}{2(N+\alpha)}t^{N-2}_{\eps}\|\nabla u_{\eps}\|_2^2<\mathcal{I}_{\eps}(u_{\eps})=c_{\eps}.
$$
Therefore, $\sigma_{\eps}=c_{\eps}-c_{HL}>0$.

\medskip
\paragraph{\bf Case $N\geq 5$.}  Note that for $N\geq 5$, $V_{\lambda}(x)\in L^2(\R^N)$ for each $\lambda>0$, thus we see that
$\mathcal{P}_{\eps}(V_{\lambda}(x))>0$.  Then for each $\eps>0$ and $\lambda>0$, there exists a unique $s_{\eps,\lambda}>1$ such that $\mathcal{P}_{\eps}(V(x/s_{\eps, \lambda}))=0$, which means that
$$
\big(s_{\eps, \lambda}^{\alpha}-s_{\eps, \lambda}^{-2}\big)\frac{N-2}{2}\|\nabla V\|_2^2=\frac{N\eps\lambda^2}{2}\|V\|_2^2+\frac{N}{q}\lambda^{\frac{2N-q(N-2)}{2}}\|V\|_q^q:=\psi_{\eps}(\lambda).
$$
Then there exists $\lambda_{\eps}>0$ such that
$$
\psi_{\eps}(\lambda_{\eps})=\min_{\lambda>0}\psi_{\eps}(\lambda)=
\frac{N\|V\|_q^q}{4q}\Big(\frac{N\|V\|_q^q(q(N-2)-2N)}{2qN\|V\|_2^2}\Big)^{\frac{2N-q(N-2)}{(q-2)(N-2)}}\eps^{\frac{q(N-2)-2N}{(q-2)(N-2)}}.
$$
This implies that $s_{\eps}:=s_{\eps, \lambda_{\eps}}\to 1$ as $\eps\to 0$.  Furthermore, we have
$$
s_{\eps}\leq 1+C\eps^{\frac{q(N-2)-2N}{(q-2)(N-2)}}.
$$
Therefore, we obtain that
$$\aligned
c_{\eps}\leq \mathcal{I}_{\eps}(V_{\lambda}(\frac{x}{s_{\eps}}))=&\frac{\alpha+2}{2(N+\alpha)}\|\nabla V\|_2^2s_{\eps}^{N-2}+
\frac{\alpha}{N+\alpha}\Big(\frac{\eps\lambda_{\eps}^2}{2}\|V\|_2^2+\frac{\lambda_{\eps}^{\frac{2N-q(N-2)}{2}}}{q}\|V\|_q^q\Big)s_{\eps}^N\\
\leq& c_{HL}+\frac{\alpha+2}{2(N+\alpha)}\|\nabla V\|_2^2(s_{\eps}^{N-2}-1)+C\eps^{\frac{q(N-2)-2N}{(q-2)(N-2)}}\\
\leq& c_{HL}+C\eps^{\frac{q(N-2)-2N}{(q-2)(N-2)}},
\endaligned
$$
which means that $\sigma_\eps\leq C\eps^{\frac{q(N-2)-2N}{(q-2)(N-2)}}$.
\smallskip

To consider the case $N=3, 4$, given $R\gg \lambda$, we introduce a cut-off function $\eta_{R}\in C_c^{\infty}(\R^N)$ such that
$\eta_R(r)=1$ for $|r|<R$, $0<\eta_R(r)<1$ for $R<|r|<2R$, $\eta_R(r)=0$ for $|r|\geq 2R$ and $|\eta_R'(r)|\leq R/2$.  Similarly, e.g. to \cite{S96}*{Section III, Theorem 2.1}, we compute
\begin{equation}\label{eq639}
\int_{\R^N}|\nabla(\eta_RV_{\lambda})|^2dx=\mathcal{S}_{HL}^{\frac{N+\alpha}{\alpha+2}}+\O((R/\lambda)^{-(N-2)}).
\end{equation}
\begin{equation}\label{eq6310}
\int_{\R^N}(I_{\alpha}*|\eta_RV_{\lambda}|^{\frac{N+\alpha}{N-2}})|\eta_RV_{\lambda}|^{\frac{N+\alpha}{N-2}}dx=
\mathcal{S}_{HL}^{\frac{N+\alpha}{\alpha+2}}+\O((R/\lambda)^{-N}).
\end{equation}
\begin{equation}\label{eq6311}
\|\eta_RV_{\lambda}\|_q^q=\lambda^{\frac{2N-q(N-2)}{2}}\|V\|_q^q(1+\O((R/\lambda)^{N-q(N-2)})).
\end{equation}
\begin{equation}\label{eq6312}
\|\eta_RV_{\lambda}\|_2^2=\lambda^{2}\|\eta_{R/\lambda}V\|_2^2=\left\{\aligned &\O(\lambda^2\ln\frac{R}{\lambda}),\quad
N=4,&\\
&\O(R\lambda), \quad N=3. &\\
\endaligned\right.
\end{equation}
Here we only sketch the proof of \eqref{eq6310}.   We may assume that $R/\lambda\gg1$.  Note that
$$
\sup_{x\in\R^N} |V(x)|^{\frac{N+\alpha}{N-2}}(1+|x|)^{N+\alpha}<\infty.
$$
It follows from \cite{MS13}*{Lemma 6.2} that
\begin{equation}\label{eq6313}
\lim_{|x|\to \infty}\frac{\big(I_{\alpha}*|V|^{\frac{N+\alpha}{N-2}}\big)(x)}{I_{\alpha}(x)}=\int_{\R^N}|V|^{\frac{N+\alpha}{N-2}}dx.
\end{equation}
Therefore, by \eqref{eq6313}, we estimate
\begin{multline*}
\Big|\int_{\R^N}(I_{\alpha}*|\eta_RV_{\lambda}|^{\frac{N+\alpha}{N-2}})|\eta_RV_{\lambda}|^{\frac{N+\alpha}{N-2}}dx-
\int_{\R^N}(I_{\alpha}*|V_{\lambda}|^{\frac{N+\alpha}{N-2}})|V_{\lambda}|^{\frac{N+\alpha}{N-2}}dx\Big|\\
\leq 2\int_{B_R^c}(I_{\alpha}*|V_{\lambda}|^{\frac{N+\alpha}{N-2}})|V_{\lambda}|^{\frac{N+\alpha}{N-2}}dx
= 2\int_{B_{R/\lambda}^c}(I_{\alpha}*|V|^{\frac{N+\alpha}{N-2}})|V|^{\frac{N+\alpha}{N-2}}dx\\
\leq C\int_{B_{R/\lambda}^c}\frac{1}{|x|^{N-\alpha}(1+|x|^2)^{\frac{N+\alpha}{2}}}dx
\leq C\int_{R/\lambda}^{\infty}r^{-N-1}dr
=C(R/\lambda)^{-N}.
\end{multline*}
\begin{multline*}
\Big|\int_{\R^N}(I_{\alpha}*|\eta_RV_{\lambda}|^{\frac{N+\alpha}{N-2}})|\eta_RV_{\lambda}|^{\frac{N+\alpha}{N-2}}dx-
\int_{\R^N}(I_{\alpha}*|V_{\lambda}|^{\frac{N+\alpha}{N-2}})|V_{\lambda}|^{\frac{N+\alpha}{N-2}}dx\Big|\\
\geq \int_{B_{2R}^c}(I_{\alpha}*|V_{\lambda}|^{\frac{N+\alpha}{N-2}})|V_{\lambda}|^{\frac{N+\alpha}{N-2}}dx
=\int_{B_{2R/\lambda}^c}(I_{\alpha}*|V|^{\frac{N+\alpha}{N-2}})|V|^{\frac{N+\alpha}{N-2}}dx\\
\geq C\int_{B_{2R/\lambda}^c}\frac{1}{|x|^{N-\alpha}(1+|x|^2)^{\frac{N+\alpha}{2}}}dx
\geq C\int_{2R/\lambda}^{\infty}r^{-N-1}dr
=C(R/\lambda)^{-N}.
\end{multline*}
Combining the above inequalities, we obtain \eqref{eq6310}.

When $R\gg\lambda$, by the above estimates we conclude that $\mathcal{P}_{\eps}(\eta_RV_{\lambda})>0$.  Then there exists a unique $t_{\eps}:=t_{\eps}(R, \lambda)>1$ such that $\mathcal{P}_{\eps}(\eta_R(x/t_{\eps})V_{\lambda}(x/t_{\eps}))=0$, which implies
\begin{multline}\label{eq638}
\frac{N-2}{2}(t_{\eps}^{\alpha}\int_{\R^N}(I_{\alpha}*|\eta_RV_{\lambda}|^{\frac{N+\alpha}{N-2}})|\eta_RV_{\lambda}|^{\frac{N+\alpha}{N-2}}dx-t_{\eps}^{-2}\|\nabla(\eta_RV_{\lambda})\|_2^2)
\\=\frac{N\eps}{2}\|\eta_RV_{\lambda}\|_2^2+\frac{N}{q}\|\eta_RV_{\lambda}\|_q^q
:=\psi_{\eps}(R, \lambda).
\end{multline}

\paragraph{\bf Case $N=4$.}
For each $\eps>0$ small, by \eqref{eq6311}, \eqref{eq6312} and choosing
$$
\lambda_{\eps}=\Big(\eps\ln\frac{1}{\eps}\Big)^{-\frac{1}{q-2}}, \quad R_{\eps}=\frac{\lambda_{\eps}}{\eps},
$$
we have
$$
\psi_{\eps}(R_{\eps}, \lambda_{\eps})\leq 2\eps\O\Big(\lambda_{\eps}^2\ln\frac{R_\eps}{\lambda_{\eps}}\Big)+\frac{4}{q}\lambda_{\eps}^{4-q}\|V\|_q^q\Big(1+\O\Big(\frac{R_\eps}{\lambda_{\eps}}\Big)^{4-2q}\Big)
\leq C\Big(\eps\ln\frac{1}{\eps}\Big)^{\frac{q-4}{q-2}}.
$$
Then it follows from \eqref{eq638} that $t_{\eps}\to 1$ as $\eps\to 0$.  Furthermore, we have
$
t_{\eps}\leq 1+C\Big(\eps\ln\frac{1}{\eps}\Big)^{\frac{q-4}{q-2}}.
$

\medskip
\paragraph{\bf Case $N=3$.}
For each $\eps>0$ small, by \eqref{eq6311}, \eqref{eq6312} and choosing
$$
\lambda_{\eps}=\eps^{-\frac{1}{q-4}}, \quad R_{\eps}=\eps^{-\frac{1}{2}},
$$
we have
$$
\psi_{\eps}(R_{\eps}, \lambda_{\eps})\leq \frac{3}{2}\eps\O(R_\eps\lambda_{\eps})+\frac{3}{q}\lambda_{\eps}^{\frac{6-q}{2}}\|V\|_q^q
\Big(1+\O\Big(\frac{R_\eps}{\lambda_{\eps}}\Big)^{3-q}\Big)
\leq C\eps^{\frac{q-6}{2(q-4)}}.
$$
Then it follows from \eqref{eq638} that $t_{\eps}\to 1$ as $\eps\to 0$.  Furthermore, we have
$
t_{\eps}\leq 1+C\eps^{\frac{q-6}{2(q-4)}}.
$

\medskip
\paragraph{\bf Conclusion of the proof for $N=3,4$.}
Combining previous estimates together, we obtain
\begin{eqnarray*}
c_{\eps}
&\leq& \mathcal{I}_{\eps}(\eta_{R_\eps}(x/t_{\eps})V_{\lambda_{\eps}}(x/t_{\eps}))\\
&=&\frac{\alpha+2}{2(N+\alpha)}\|\nabla (\eta_{R_\eps}V_{\lambda_{\eps}})\|_2^2t_{\eps}^{N-2}+
\frac{\alpha}{N+\alpha}\Big(\frac{\eps}{2}\|\eta_{R_\eps}V_{\lambda_{\eps}}\|_2^2+\frac{1}{q}\|\eta_{R_\eps}V_{\lambda_{\eps}}\|_q^q\Big)t_{\eps}^N\\
&\leq& c_{HL}+\O((R_{\eps}/\lambda_{\eps})^{-(N-2)})+C\psi_{\eps}(R_{\eps}, \lambda_{\eps})
= c_{HL}+\left\{\aligned &C\Big(\eps\ln\frac{1}{\eps}\Big)^{\frac{q-4}{q-2}},\quad
N=4,&\\
&C\eps^{\frac{q-6}{2(q-4)}}, \quad N=3, &\\
\endaligned\right.
\end{eqnarray*}
which completes the proof of this lemma.
\end{proof}

\begin{lemma}\label{lem632}
$$
\|u_{\eps}\|_q^q=\frac{2q\eps}{q(N-2)-2N}\|u_{\eps}\|_2^2, \quad \|\nabla u_{\eps}\|_2^2+\frac{(q-2)N\eps}{q(N-2)-2N}\|u_{\eps}\|_2^2=\int_{\R^N}(I_{\alpha}*|u_\eps|^{\frac{N+\alpha}{N-2}})|u_\eps|^{\frac{N+\alpha}{N-2}}dx.
$$
\end{lemma}
\begin{proof}
Follows from Nehari and Poho\v zaev identities for \eqref{eqPeps}.
\end{proof}

\begin{lemma}\label{lem633}
$\frac{(q-2)N\eps}{q(N-2)-2N}\|u_{\eps}\|_2^2\leq C\sigma_{\eps}$.  Moreover, we have
$$
\lim_{\eps\to 0}\eps\|u_{\eps}\|_2^2=0, \quad \lim_{\eps\to 0}\|u_{\eps}\|_q^q=0, \quad \lim_{\eps\to 0}\|u_{\eps}\|_{2^*}^2=\mathcal{S}_{HL}^{\frac{N-2}{\alpha+2}}\mathcal{C}_\alpha^{-\frac{N-2}{N+\alpha}}=(\mathcal{S}_*\mathcal{C}_\alpha^{-1})^\frac{N-2}{\alpha+2},
$$
$$
\lim_{\eps\to 0}\|\nabla u_{\eps}\|_2^2=\lim_{\eps\to 0}
\int_{\R^N}(I_{\alpha}*|u_\eps|^{\frac{N+\alpha}{N-2}})|u_\eps|^{\frac{N+\alpha}{N-2}}dx= \frac{2(N+\alpha)}{\alpha+2}c_{HL}=\mathcal{S}_{HL}^{\frac{N+\alpha}{\alpha+2}}.
$$
\end{lemma}
\begin{proof}
By Lemma~\ref{lem631}, we see that
\begin{equation}\label{eq632}
c_{HL}+o(1)=c_{\eps}=\mathcal{I}_{\eps}(u_{\eps})=\frac{1}{N}\|\nabla u_{\eps}\|_2^2+\frac{\alpha(N-2)}{2N(N+\alpha)}\int_{\R^N}(I_{\alpha}*|u_\eps|^{\frac{N+\alpha}{N-2}})|u_\eps|^{\frac{N+\alpha}{N-2}}dx.
\end{equation}
This, together with Lemma~\ref{lem632}, implies that there exists $C>0$ such that $$\int_{\R^N}(I_{\alpha}*|u_\eps|^{\frac{N+\alpha}{N-2}})|u_\eps|^{\frac{N+\alpha}{N-2}}dx\geq C$$ for all small $\eps>0$.
Note that $\mathcal{P}_{HL}(u_{\eps})<0$, then there exists $t_{\eps}\in (0, 1)$ such that $\mathcal{P}_{HL}(u_{\eps}(x/t_{\eps}))=0$, which means that
$$
\|\nabla u_{\eps}\|_2^2
=t_{\eps}^{\alpha+2}\int_{\R^N}(I_{\alpha}*|u_\eps|^{\frac{N+\alpha}{N-2}})|u_\eps|^{\frac{N+\alpha}{N-2}}dx.
$$
If $t_{\eps}\to 0$ as $\eps\to 0$, then we must have $\|\nabla u_{\eps}\|_2^2\to 0$ as $\eps\to 0$, this contradicts
$$
0<c_{HL}\leq J(u_{\eps}(x/t_{\eps}))=\frac{(2+\alpha)t_{\eps}^{N-2}}{2(N+\alpha)}\|\nabla u_{\eps}\|_2^2.
$$
Therefore, there exists $C>0$ such that $t_{\eps}>C$ for $\eps>0$ small.  Thus we have
$$\aligned
c_{HL}\leq J(u_{\eps}(x/t_{\eps}))=&\mathcal{I}_{\eps}(u_{\eps}(x/t_{\eps}))-t_{\eps}^N(\frac{\eps}{2}\| u_{\eps}\|_2^2+\frac{1}{q}\| u_{\eps}\|_q^q)\\
\leq& \mathcal{I}_{\eps}(u_{\eps})-\frac{(q-2)N\eps}{q(N-2)-2N}\| u_{\eps}\|_2^2t_{\eps}^N\\
=& c_{\eps}-\frac{(q-2)N\eps}{q(N-2)-2N}\| u_{\eps}\|_2^2t_{\eps}^N.
\endaligned
$$
This means that
$$
\frac{(q-2)N\eps}{q(N-2)-2N}\| u_{\eps}\|_2^2\leq t_{\eps}^{-N}(c_{\eps}-c_{HL})\leq C\sigma_{\eps}.
$$
Moreover, we have $\eps\| u_{\eps}\|_2^2\to 0$ as $\eps\to 0$.
It follows from Lemma~\ref{lem632} that
$$
\| u_{\eps}\|_q^q\to 0, \quad \|\nabla u_{\eps}\|_2^2+o(1)=\int_{\R^N}(I_{\alpha}*|u_\eps|^{\frac{N+\alpha}{N-2}})|u_\eps|^{\frac{N+\alpha}{N-2}}dx\to \frac{2+\alpha}{2(N+\alpha)}c_{HL}=\mathcal{S}_{HL}^{\frac{N+\alpha}{\alpha+2}}
$$
as $\eps\to 0$.

Set $w_{\eps}(x)=u_{\eps}(\mathcal{S}_{HL}^{\frac{1}{\alpha+2}}x)$, then we have, as $\eps\to 0$,
$$
\|\nabla w_{\eps}\|_2^2=\mathcal{S}_{HL}^{\frac{2-N}{\alpha+2}}\|\nabla u_{\eps}\|_2^2\to \mathcal{S}_{HL},
$$
$$
\int_{\R^N}(I_{\alpha}*|w_\eps|^{\frac{N+\alpha}{N-2}})|w_\eps|^{\frac{N+\alpha}{N-2}}dx=\mathcal{S}_{HL}^{\frac{-(N+\alpha)}{\alpha+2}}
\int_{\R^N}(I_{\alpha}*|u_\eps|^{\frac{N+\alpha}{N-2}})|u_\eps|^{\frac{N+\alpha}{N-2}}dx\to 1.
$$
On the other hand, by the HLS inequality, we have
\begin{equation*}
\mathcal{C}_\alpha^{\frac{N-2}{N+\alpha}}\mathcal{S}_{HL}=\mathcal{S}_*\leq \frac{\|\nabla w_{\eps}\|_2^2}{\|w_{\eps}\|_{2^*}^2}
\leq
\mathcal{C}_\alpha^{\frac{N-2}{N+\alpha}}\frac{\|\nabla w_{\eps}\|_2^2}{\left(\int_{\R^N}(I_{\alpha}*|u_\eps|^{\frac{N+\alpha}{N-2}})|u_\eps|^{\frac{N+\alpha}{N-2}}dx\right)^{\frac{N-2}{N+\alpha}}}
\to \mathcal{C}_\alpha^{\frac{N-2}{N+\alpha}}\mathcal{S}_{HL}.
\end{equation*}
This means that
$$
\lim_{\eps\to 0}\|w_{\eps}\|_{2^*}^2=\mathcal{C}_\alpha^{-\frac{N-2}{N+\alpha}}, \quad \lim_{\eps\to 0}\|u_{\eps}\|_{2^*}^2=\mathcal{S}_{HL}^{\frac{N-2}{\alpha+2}}\mathcal{C}_\alpha^{-\frac{N-2}{N+\alpha}},
$$
which completes the proof.
\end{proof}

Set
$$\overline{w}_{\eps}(x):=\frac{w_{\eps}(x)}{\|w_{\eps}\|_{2^*}},\quad\overline{V}(x):=\mathcal{C}_\alpha^{\frac{N-2}{2(N+\alpha)}}V(\mathcal{S}_{HL}^{\frac{1}{\alpha+2}}x).$$
Then $\|\overline{w}_{\eps}\|_{2^*}=\|\overline{V}\|_{2^*}=1$, $\|\nabla \overline{V}\|_2^2=\mathcal{S}_*$ and $\{\overline{w}_{\eps}\}$ is a minimizing sequence for the critical Sobolev constant $\mathcal{S}_*$.
Similarly to the arguments in \cite{MM-14}*{p.1094}, we conclude that for $\eps>0$ small there exists $\lambda_{\eps}>0$ such that
$$
\int_{B(0, \lambda_{\eps})}|\overline{w}_{\eps}(x)|^{2^*}dx=\int_{B(0, 1)}|\overline{V}(x)|^{2^*}dx.
$$
We define the rescaled family
$$
v_{\eps}(x):=\lambda_{\eps}^{\frac{N-2}{2}}\overline{w}_{\eps}(\lambda_{\eps}x).
$$
Then
$$
\|v_{\eps}\|_{2^*}=1,\quad \|\nabla v_{\eps}\|^2_2=\mathcal{S}_*+o(1),
$$
i.e., $\{v_{\eps}\}$ is a minimizing sequence for $\mathcal{S}_*$.  Furthermore,
$$
\int_{B(0, 1)}|v_{\eps}|^{2^*}dx=\int_{B(0, 1)}|\overline{V}(x)|^{2^*}dx.
$$

\begin{lemma}\label{lem635}
$\lim_{\eps\to 0}\|\nabla (\overline{v}_{\eps}-V)\|_2=\lim_{\eps\to 0}\|\overline{v}_{\eps}-V\|_{2^*}=0$, where $\overline{v}_{\eps}(x):=\lambda_{\eps}^{\frac{N-2}{2}}u_{\eps}(\lambda_{\eps}x)$.
\end{lemma}
\begin{proof}
Since $\{v_{\eps}\}$ is a minimizing sequence of $\mathcal{S}_*$, it follows from the Concentration-Compactness Principle of P.L.Lions \cite{S96}*{Chapter 1, Theorem 4.9} that
$$
\lim_{\eps\to 0}\|\nabla (v_{\eps}-\overline{V})\|_2=\lim_{\eps\to 0}\|v_{\eps}-\overline{V}\|_{2^*}=0.
$$
which, together with the definitions of $w_{\eps}$ and $\overline{w}_{\eps}$, implies that
$$
\lim_{\eps\to 0}\|\nabla (\overline{v}_{\eps}-V)\|_2=\lim_{\eps\to 0}\|\overline{v}_{\eps}-V\|_{2^*}=0.
\qedhere$$
\end{proof}

By a simple calculation, we see that $v_{\eps}$ solves the equation
\begin{multline}\label{eq6314}
-\Delta v_\eps+\lambda_{\eps}^2\mathcal{S}_{HL}^{\frac{2}{\alpha+2}}\eps v_{\eps}\\
=\|w_\eps\|_{2^*}^{\frac{2(\alpha+2)}{N-2}}\lambda_{\eps}^{-2\alpha}
\mathcal{S}_{HL}^{\frac{2-\alpha}{\alpha+2}}(I_{\alpha}*|v_\eps|^{\frac{N+\alpha}{N-2}})|v_\eps|^{\frac{\alpha+4-N}{N-2}}v_\eps
-\|w_\eps\|_{2^*}^{q-2}\lambda_{\eps}^{\frac{2N-q(N-2)}{2}}
\mathcal{S}_{HL}^{\frac{2}{\alpha+2}}|v_\eps|^{q-2}v_\eps.
\end{multline}
By the definition of $v_{\eps}$ and $w_{\eps}$, we obtain
\begin{equation}\label{eq636}
\|v_{\eps}\|_q^q=\|w_{\eps}\|_{2^*}^{-q}\lambda_{\eps}^{\frac{q(N-2)-2N}{2}}\mathcal{S}_{HL}^{-\frac{N}{\alpha+2}}\|u_{\eps}\|_q^q,
\quad \|v_{\eps}\|_2^2=\|w_{\eps}\|_{2^*}^{-2}\lambda_{\eps}^{-2}\mathcal{S}_{HL}^{-\frac{N}{\alpha+2}}\|u_{\eps}\|_2^2.
\end{equation}
It follows from Lemma~\ref{lem632} and Lemma~\ref{lem633} that
\begin{equation}\label{eq633}
\|w_{\eps}\|_{2^*}^{q}\lambda_{\eps}^{\frac{2N-q(N-2)}{2}}\mathcal{S}_{HL}^{\frac{N}{\alpha+2}}\|v_{\eps}\|_q^q=
\frac{2q\eps}{q(N-2)-2N}\|w_{\eps}\|_{2^*}^{2}\lambda_{\eps}^{2}\mathcal{S}_{HL}^{\frac{N}{\alpha+2}}\|v_{\eps}\|_2^2\leq C\sigma_{\eps}.
\end{equation}
Therefore we can deduce the following estimates on $\lambda_{\eps}$.
\begin{lemma}\label{lem636}
$\sigma_{\eps}^{-\frac{2}{(N-2)q-2N}}\lesssim \lambda_{\eps}\lesssim \eps^{-\frac12}\sigma_{\eps}^{\frac{1}{2}}$,
as $\eps\to 0$.
\end{lemma}
\begin{proof}
Since $v_{\eps}\to \overline{V}$ strongly in $D^{1}(\R^N)$ and $L^{2^*}(\R^N)$, we have $v_{\eps}\to \overline{V}$ strongly in $L^{s}_{loc}(\R^N)$ for $s\in [2, 2^*)$, thus we get that
\begin{equation}\label{eq634}
\|v_{\eps}\|_q^q\geq \int_{B(0, 1)}|v_{\eps}|^qdx\geq C\int_{B(0, 1)}|v_{\eps}|^{2^*}dx\to \int_{B(0, 1)}|\overline{V}|^{2^*}dx>0
\end{equation}
and
\begin{equation}\label{eq635}
\|v_{\eps}\|_2^2\geq \int_{B(0, 1)}|v_{\eps}|^2dx\to \int_{B(0, 1)}|\overline{V}|^{2}dx>0.
\end{equation}
Therefore, the assertion follows by \eqref{eq633}, \eqref{eq634} and \eqref{eq635}.
\end{proof}

Set
$$Q_{\eps}(x):=\|w_\eps\|_{2^*}^{q-2}\lambda_{\eps}^{\frac{2N-q(N-2)}{2}}
\mathcal{S}_{HL}^{\frac{2}{\alpha+2}}|v_\eps|^{q-2}.$$
By Lemma~\ref{lem636} and since $u_{\eps}\leq C|x|^{-(N-2)/2}\|u_{\eps}\|_{2^*}$, we conclude that
$$
Q_{\eps}(x)\leq Q_0(x):=C|x|^{-\frac{(N-2)(q-2)}{2}},
$$
where $C>0$ does not depend on $\eps>0$ or $x\in\R^N$.  Therefore, for small $\eps>0$, it follows from \eqref{eq6314} that solutions $v_{\eps}>0$ satisfy the linear inequality
$$
-\Delta v_\eps+\lambda_{\eps}^2\mathcal{S}_{HL}^{\frac{2}{\alpha+2}}\eps v_{\eps}+Q_0(x)v_{\eps}\geq 0, \quad x\in \R^N.
$$
By \cite{MM-14}*{Lemma 4.8} we have the following lower estimates of $v_{\eps}$.

\begin{proposition}\label{pro631}
There exists $R>0$ and $c>0$ such that for all small $\eps>0$,
  $$v_{\eps}(x)\geq c|x|^{-(N-2)}e^{-\sqrt{\eps}\mathcal{S}_{HL}^{\frac{1}{\alpha+2}}\lambda_{\eps}|x|}, \quad |x|\geq R.$$
\end{proposition}

The proof of the next result is nearly identical to the proofs in \cite{MM-14}*{Lemma 4.6 and pp.1097-1098}. We outline the arguments for reader's convenience.

\begin{lemma}\label{lem637}
For $\eps\to 0$, we have $\|v_{\eps}\|_{q}^q\sim 1$,
$$
\lambda_{\eps}\sim\left\{\aligned & \eps^{-\frac{1}{q-4}},
&N=3,&\\
&\Big(\eps\ln\frac{1}{\eps}\Big)^{-\frac{1}{q-2}},
&N=4,&\\
&\eps^{-\frac{2}{(q-2)(N-2)}},  &N\geq 5, &\\
\endaligned\right.
\qquad
\|v_{\eps}\|_2^2\sim\left\{\aligned & \eps^{-\frac{q-6}{2(q-4)}}, &
N=3,&\\
&\ln\frac{1}{\eps},&
N=4,&\\
&1, & N\geq 5. &\\
\endaligned\right.
$$
\end{lemma}
\begin{proof}
For all $N\geq 3$, \eqref{eq634} implies that $\|v_{\eps}\|_q^q\geq C$ for some $C>0$.

\noindent
\paragraph{\bf Case $N\geq 5$.}
By Lemmas~\ref{lem631} and \ref{lem636}, we see that for $N\geq 5$
$$
\eps^{-\frac{2}{(q-2)(N-2)}}\lesssim\lambda_{\eps}\lesssim \eps^{-\frac{2}{(q-2)(N-2)}}.
$$
Moreover, $\|v_{\eps}\|_2^2\sim 1$ and $\|v_{\eps}\|_q^q\sim 1$ by Lemma~\ref{lem631}, \eqref{eq633}, \eqref{eq634} and \eqref{eq635}.

\noindent
\paragraph{\bf Case $N=4$.}
By Proposition~\ref{pro631}, we obtain
$$
\|v_{\eps}\|_2^2\geq\int_{\R^N\setminus B_{R}}|v_{\eps}|^2dx\geq \int_{R}^{\infty}Cr^{-2}e^{-2\sqrt{\eps}\mathcal{S}_{HL}^{\frac{1}{\alpha+2}}\lambda_{\eps}r}dr=C\Gamma(0, 2\sqrt{\eps}\mathcal{S}_{HL}^{\frac{1}{\alpha+2}}\lambda_{\eps}),
$$
where
$\Gamma(0,t)=-\ln t-\gamma+\O(t)$ as $t\searrow 0$ is the incomplete Gamma function and $\gamma\approx 0.57$ is the Euler constant.  Hence, by Lemmas~\ref{lem631} and \ref{lem636},  we obtain that
$$
\|v_{\eps}\|_2^2\geq C(-\ln (2\sqrt{\eps}\mathcal{S}_{HL}^{\frac{1}{\alpha+2}}\lambda_{\eps})-\gamma)\geq C\ln{(\frac{1}{\sqrt{\eps}\lambda_{\eps}})}\geq C\ln{\frac{1}{\eps}}.
$$
This, together with Lemmas~\ref{lem631}, \ref{lem636} and \eqref{eq633}, implies that
$$
\Big(\eps\ln\frac{1}{\eps}\Big)^{-\frac{1}{q-2}}\lesssim \lambda_{\eps}\leq \frac{C\sigma_{\eps}^{\frac12}}{\eps^{\frac12}\|v_{\eps}\|_2}\lesssim \Big(\eps\ln\frac{1}{\eps}\Big)^{-\frac{1}{q-2}}.
$$
Moreover,
$$
\|v_{\eps}\|^2_2\leq \frac{C\sigma_{\eps}}{\eps\lambda_{\eps}^2}\leq C\ln{\frac{1}{\eps}}, \quad
\|v_{\eps}\|^q_q\leq C\sigma_{\eps}\lambda_{\eps}^{q-4}\leq C.
$$

\noindent
\paragraph{\bf Case $N=3$.}
By Proposition~\ref{pro631}, we obtain
\begin{equation}\label{eq6315}
\|v_{\eps}\|_2^2\geq\int_{\R^N\setminus B_{R}}|v_{\eps}|^2dx\geq \int_{R}^{\infty}Ce^{-2\sqrt{\eps}\mathcal{S}_{HL}^{\frac{1}{\alpha+2}}\lambda_{\eps}r}dr\geq \frac{C}{\sqrt{\eps}\lambda_{\eps}}
\end{equation}
This, together with Lemma~\ref{lem636} and \eqref{eq633}, implies that
$$
\lambda_{\eps}\leq \frac{C\sigma_{\eps}^{\frac12}}{\eps^{\frac12}\|v_{\eps}\|_2}\leq \eps^{-\frac14}\sigma_{\eps}^{\frac12}\lambda_{\eps}^{\frac12}\quad \Longleftrightarrow \quad
\lambda_{\eps}\leq\eps^{-\frac12}\sigma_{\eps}.
$$
By Lemmas~\ref{lem631} and \ref{lem636}, we obtain that
\begin{equation}\label{eq6316}
\eps^{-\frac{1}{q-4}}\lesssim \lambda_{\eps}\lesssim \eps^{-\frac{1}{q-4}}.
\end{equation}
Furthermore, we deduce from Lemma~\ref{lem631}, \eqref{eq6315} and \eqref{eq6316} that
$$
C\eps^{-\frac{q-6}{2(q-4)}}\leq\frac{C}{\sqrt{\eps}\lambda_{\eps}}\leq\|v_{\eps}\|^2_2\leq \frac{C\sigma_{\eps}}{\eps\lambda_{\eps}^2}\leq C\eps^{-\frac{q-6}{2(q-4)}}, \quad
\|v_{\eps}\|^q_q\leq C\sigma_{\eps}\lambda_{\eps}^{\frac{q-6}{2}}\leq C.
$$
This complete the proof.
\end{proof}

\subsection{Critical Thomas--Fermi case}
Throughout this section we assume that $\frac{N+\alpha}{N}<p<\frac{N+\alpha}{N-2}$ and $q=\frac{2Np}{N+\alpha}$.
Let
	$$
	\mathcal{S}_{T\!F}=\inf_{w\in L^{\frac{2Np}{N+\alpha}}(\R^N)\setminus \{0\}}\frac{\int_{\R^N}|w|^\frac{2Np}{N+\alpha}dx}{\Big\{ \int_{\R^N}(I_{\alpha}*|w|^p)|w|^p dx\Big\}^{\frac{N}{N+\alpha}}}=\mathcal C_{\alpha}^{-\frac{N}{N+\alpha}},
	$$
	where $\mathcal C_{\alpha}$ is the optimal constant in \eqref{HLS}, as described in \eqref{eqSTF}.

It is well-known \cite{Lieb}*{Theorem 4.3} that $\mathcal{S}_{T\!F}$ is achieved by the function
\begin{equation}\label{eqV1tilde}
\widetilde{V}(x)=\widetilde{U}(\mathcal{S}_{T\!F}^{\frac{1}{\alpha}}x),
\end{equation}
and the family of rescalings
\begin{equation}\label{eq-R-HLS-V}
\widetilde{V}_\lambda(x):=\lambda^{-\frac{N+\alpha}{2p}}\widetilde{V}(x/\lambda)=\lambda^{-\frac{N}{q}}\widetilde{V}(x/\lambda)\qquad(\lambda>0),
\end{equation}
here $\widetilde{U}$ is the groundstate solution of \eqref{eqTF0} defined in \eqref{eqUtilde}.
It is clear that
$$\|\widetilde{V}_{\lambda}\|_q^q=\| \widetilde{V}\|_q^q=\int_{\R^N}(I_{\alpha}*|\widetilde{V}|^{p})|\widetilde{V}|^{p}dx=\int_{\R^N}(I_{\alpha}*|\widetilde{V}_{\lambda}|^{p})|\widetilde{V}_{\lambda}|^{p}dx=\mathcal{S}_{T\!F}^{\frac{N+\alpha}{\alpha}}.
$$
The energy functional which corresponds to \eqref{eqTF0} is
$$
H(u)=\frac{1}{q}\|u\|_q^q-\frac{1}{2p}\int_{\R^N}(I_{\alpha}*|u|^{p})|u|^{p}dx.
$$
We define,
$$
c_{T\!F}=\inf_{u\in \mathscr{P}_{T\!F}}H(u)=\inf_{u\in L^q(\R^N)\setminus \{0\}}\max_{t>0}H(u(x/t)),
$$
where
$$
\mathscr{P}_{T\!F}=\Big\{u\in L^{q}(\R^N)\setminus \{0\}: \mathcal{P}_{T\!F}(u):=\| u\|_q^q-\int_{\R^N}(I_{\alpha}*|u|^{p})|u|^{p}dx=0\Big\}.
$$
By a simple calculation, we see that $c_{T\!F}=\frac{\alpha}{2Np}\mathcal{S}_{T\!F}^{\frac{N+\alpha}{\alpha}}$.

Note that $\widetilde{V}_\lambda\in L^2(\R^N)$ for $N\ge 4$, and $\widetilde{V}_\lambda\in D^{1}(\R^N)$ if $N\geq 3$.

\begin{lemma}\label{lem641}
Let $\overline{\sigma}_{\eps}=c_{\eps}-c_{T\!F}$ and $\eps\to 0$. If $N\geq 4$ then $$0<\overline{\sigma}_{\eps}\lesssim \eps^{\frac{2N-q(N-2)}{2q}},$$
while if $N=3$ then
$$
0<\overline{\sigma}_{\eps}\lesssim \left\{
\begin{array}{ll}
\eps^{\frac{3+\alpha-p}{2p}},\quad &p\in \left(\tfrac13(3+\alpha), \tfrac23(3+\alpha)\right),\smallskip\\
\big(\eps\ln\frac{1}{\eps}\big)^{\frac{1}{4}},\quad &p=\tfrac23(3+\alpha),\smallskip\\
\eps^{(\frac{3+\alpha-p}{p})^2},\quad &p\in \left(\tfrac23(3+\alpha), (3+\alpha)\right).
\end{array}
\right.
$$
\end{lemma}

\begin{proof}
Note that $u_{\eps}\in \mathscr{P}_{\eps}$ is a solution of \eqref{eqPeps} with $\mathcal{I}_{\eps}(u_{\eps})=c_{\eps}$, then $\P_{T\!F}(u_{\eps})<0$, thus there exists $t_{\eps}\in (0, 1)$ such that $u_{\eps}(x/t_{\eps})\in \mathscr{P}_{T\!F}$ and we have
$$
c_{T\!F}\leq H(u_{\eps}(\frac{x}{t_{\eps}}))=\frac{2p-q}{2pq}t^{N}_{\eps}\| u_{\eps}\|_q^q<\mathcal{I}_{\eps}(u_{\eps})=c_{\eps}.
$$
Therefore, $\overline{\sigma}_{\eps}=c_{\eps}-c_{T\!F}>0$.

\medskip
\noindent
\paragraph{\bf Case $N\geq 4$.}  Note that for $N\geq 4$, $\tV_{\lambda}(x)\in L^2(\R^N)$ for each $\lambda>0$ and therefore, $\P_{\eps}(\tV_{\lambda}(x))>0$. Then for each $\eps>0$ and $\lambda>0$, there exists a unique $s_{\eps,\lambda}>1$ such that $\P_{\eps}(\tV(x/s_{\eps, \lambda}))=0$, which means that
$$\aligned
(s_{\eps, \lambda}^{\alpha}-1)\frac{N+\alpha}{2p}\| \tV\|_q^q=&\frac{(N-2)}{2s_{\eps, \lambda}^2}\lambda^{\frac{(N-2)q-2N}{q}}\|\nabla \tV\|_2^2+\frac{N\eps}{2}\lambda^{\frac{Nq-2N}{q}}\|\tV\|_2^2\\ \leq& \frac{(N-2)}{2}\lambda^{\frac{(N-2)q-2N}{q}}\|\nabla \tV\|_2^2+\frac{N\eps}{2}\lambda^{\frac{Nq-2N}{q}}\|\tV\|_2^2 :=\phi_{\eps}(\lambda)
\endaligned
$$
then there exists $\lambda_{\eps}>0$ such that
$$
\phi_{\eps}(\lambda_{\eps})=\min_{\lambda>0}\phi_{\eps}(\lambda)\leq
C\eps^{\frac{2N-q(N-2)}{2q}}.
$$
Therefore $s_{\eps}:=s_{\eps, \lambda_{\eps}}\to 1$ as $\eps\to 0$.  Furthermore, we have
$$
s_{\eps}\leq 1+C\eps^{\frac{2N-q(N-2)}{2q}}.
$$
Therefore, we obtain that
$$\aligned
c_{\eps}\leq \mathcal{I}_{\eps}(\tV_{\lambda_\eps}(\frac{x}{s_{\eps}}))=&\frac{s_{\eps}^{N-2}}{N}\|\nabla \tV_{\lambda_\eps}\|_2^2+
\frac{\alpha s_{\eps}^{N+\alpha}}{2Np}\int_{\R^N}(I_{\alpha}*|\tV|^{p})|\tV|^{p}dx\\
\leq& c_{T\!F}+c_{T\!F}(s_{\eps}^{N+\alpha}-1)+C\eps^{\frac{2N-q(N-2)}{2q}}\\
\leq& c_{T\!F}+C\eps^{\frac{2N-q(N-2)}{2q}},
\endaligned
$$
which means that $\sigma_\eps\leq C\eps^{\frac{2N-q(N-2)}{2q}}$.

\medskip
\noindent
\paragraph{\bf Case $N=3$.}
To consider the case $N=3$, given $R\gg \lambda$, we introduce a cut-off function $\eta_{R}\in C_c^{\infty}(\R^N)$ such that
$\eta_R(r)=1$ for $|r|<R$, $0<\eta_R(r)<1$ for $R<|r|<2R$, $\eta_R(r)=0$ for $|r|\geq 2R$ and $|\eta_R'(r)|\leq R/2$.  We then compute as in, e.g. \cite{S96}*{Theorem 2.1}:
\begin{equation}\label{eq649}
\int_{\R^N}|\nabla(\eta_R\tV_{\lambda})|^2dx=\lambda^{\frac{p(N-2)-(N+\alpha)}{p}}\|\nabla \tV\|_2^2(1+\O((R/\lambda)^{\frac{p(N-2)-2(N+\alpha)}{p}}).
\end{equation}
\begin{equation}\label{eq6410}
\int_{\R^N}(I_{\alpha}*|\eta_R\tV_{\lambda}|^{p})|\eta_R\tV_{\lambda}|^{p}dx=
\mathcal{S}_{T\!F}^{\frac{N+\alpha}{\alpha}}+\O((R/\lambda)^{-N}).
\end{equation}
\begin{equation}\label{eq6411}
\|\eta_R\tV_{\lambda}\|_q^q=\mathcal{S}_{T\!F}^{\frac{N+\alpha}{\alpha}}+\O((R/\lambda)^{-N}).
\end{equation}
\begin{equation}\label{eq6412}
\|\eta_R\tV_{\lambda}\|_2^2=\lambda^{\frac{3p-(3+\alpha)}{p}}\|\eta_{R/\lambda}\tV\|_2^2\sim
\left\{
\begin{array}{ll} \lambda^{\frac{3p-(3+\alpha)}{p}}\big(1-(\frac{R}{\lambda})^{\frac{3p-2(3+\alpha)}{p}}\big),\quad
&p\in \left(\tfrac13(3+\alpha), \tfrac23(3+\alpha)\right),\smallskip\\
\lambda^{\frac{3}{2}}\ln\frac{R}{\lambda}, \quad
&p=\tfrac23(3+\alpha),\smallskip\\
\lambda^{\frac{3p-(3+\alpha)}{p}}(\frac{R}{\lambda})^{\frac{3p-2(3+\alpha)}{p}},\quad
&p\in \left(\tfrac23(3+\alpha), (3+\alpha)\right).
\end{array}
\right.
\end{equation}

 When $R\gg\lambda$, by the above estimates, we see that $\mathcal{P}_{\eps}(\eta_R\tV_{\lambda})>0$,  then there exists a unique $t_{\eps}:=t_{\eps}(R, \lambda)>1$ such that $\mathcal{P}_{\eps}(\eta_R(x/t_{\eps})\tV_{\lambda}(x/t_{\eps}))=0$, which implies that
\begin{multline}\label{eq6414}
\frac{N+\alpha}{2p}(t_{\eps}^{\alpha}\int_{\R^N}(I_{\alpha}*|\eta_R\tV_{\lambda}|^{p})|\eta_R\tV_{\lambda}|^{p}dx-\|\eta_R\tV_{\lambda}\|_q^q)\\
\leq\frac{N\eps}{2}\|\eta_R\tV_{\lambda}\|_2^2+\frac{N-2}{2}\|\nabla(\eta_R\tV_{\lambda})\|_2^2
:=\psi_{\eps}(R, \lambda).
\end{multline}
To estimate $\psi_{\eps}(R, \lambda)$ we consider three cases.

\medskip
\paragraph{\sl $(i)$ Case $p\in \left(\tfrac13(3+\alpha), \tfrac23(3+\alpha)\right)$.}
For small $\eps>0$, set
$$
\lambda_{\eps}= \eps^{-\frac{1}{2}}, \quad R_{\eps}= \eps^{-\frac{3}{2}}.
$$
Then
$$\aligned
\psi_{\eps}(R_{\eps}, \lambda_{\eps})=& \frac{1}{2}\lambda_{\eps}^{\frac{p-(3+\alpha)}{p}}\|\nabla \tV\|_2^2(1+\O((R_{\eps}/\lambda_{\eps})^{\frac{p-2(3+\alpha)}{p}})+
\frac{N\eps}{2}\O(\lambda_{\eps}^{\frac{3p-(3+\alpha)}{p}}(1-(R_{\eps}/\lambda_{\eps})^{\frac{3p-2(3+\alpha)}{p}}))\\
\leq& C\eps^{\frac{3+\alpha-p}{2p}}.
\endaligned
$$
Therefore it follows from \eqref{eq6414} that $t_{\eps}\to 1$ as $\eps\to 0$ and
$
t_{\eps}\leq 1+C\eps^{\frac{3+\alpha-p}{2p}}.
$

\medskip
\paragraph{\sl $(ii)$ Case $p=\tfrac23(3+\alpha)$.}
For small $\eps>0$, set
$$
\lambda_{\eps}= \big(\eps\ln\frac{1}{\eps}\big)^{-\frac{1}{2}}, \quad R_{\eps}= \frac{\lambda_{\eps}}{\eps}.
$$
Then
$$
\psi_{\eps}(R_{\eps}, \lambda_{\eps})=\frac{1}{2}\lambda_{\eps}^{-\frac{1}{2}}\|\nabla \tV\|_2^2(1+\O((R_{\eps}/\lambda_{\eps})^{-2})+
\frac{N\eps}{2}\O(\lambda_{\eps}^{\frac{3}{2}}\ln(R_\eps/\lambda_{\eps})
\leq C(\eps\ln\tfrac{1}{\eps})^{\frac{1}{4}}.
$$
Therefore it follows from \eqref{eq6414} that $t_{\eps}\to 1$ as $\eps\to 0$ and
$
t_{\eps}\leq 1+C(\eps\ln\frac{1}{\eps})^{\frac{1}{4}}.
$

\medskip
\paragraph{\sl $(iii)$ Case $p\in \left(\tfrac23(3+\alpha), 3+\alpha\right)$.}
For small $\eps>0$, set
$$
\lambda_{\eps}= \eps^{-\frac{3+\alpha-p}{p}}, \quad R_{\eps}= \frac{\lambda_{\eps}}{\eps}.
$$
Then
$$\aligned
\psi_{\eps}(R_{\eps}, \lambda_{\eps})=& \frac{1}{2}\lambda_{\eps}^{\frac{p-(3+\alpha)}{p}}\|\nabla \tV\|_2^2(1+\O\big((R_{\eps}/\lambda_{\eps})^{\frac{p-2(3+\alpha)}{p}}\big)+
\frac{N\eps}{2}\O\big(\lambda^{\frac{3p-(3+\alpha)}{p}}(R_\eps/\lambda_\eps)^{\frac{3p-2(3+\alpha)}{p}}\big)\\
\leq& C\eps^{(\frac{3+\alpha-p}{p})^2}.
\endaligned
$$
Therefore it follows from \eqref{eq6414} that $t_{\eps}\to 1$ as $\eps\to 0$ and
$
t_{\eps}\leq 1+C\eps^{(\frac{3+\alpha-p}{p})^2}.
$

\medskip
\paragraph{\bf Conclusion of the proof for $N=3$.}
From $(i)-(iii)$ we deduce that
$$\aligned
c_{\eps}\leq& \mathcal{I}_{\eps}(\eta_{R_\eps}(x/t_{\eps})\tV_{\lambda_{\eps}}(x/t_{\eps}))\\
=&\frac{\alpha+2}{2(3+\alpha)}\|\nabla (\eta_{R_\eps}\tV_{\lambda_{\eps}})\|_2^2t_{\eps}+
\frac{\alpha}{3+\alpha}\Big(\frac{\eps}{2}\|\eta_{R_\eps}\tV_{\lambda_{\eps}}\|_2^2+\frac{1}{q}\|\eta_{R_\eps}\tV_{\lambda_{\eps}}\|_q^q\Big)t_{\eps}^3\\
\leq& c_{T\!F}+c_{T\!F}(t_{\eps}^{3}-1)+\O(\eps^3)+C\psi_{\eps}(R_{\eps}, \lambda_{\eps})\\
\leq& c_{T\!F}+\O(\eps^3)+
\left\{
\begin{array}{ll}
\eps^{\frac{3+\alpha-p}{2p}},\quad &p\in \left(\tfrac13(3+\alpha), \tfrac23(3+\alpha)\right),\\
\big(\eps\ln\frac{1}{\eps}\big)^{\frac{1}{4}},\quad &p=\tfrac23(3+\alpha),\smallskip\\
\eps^{(\frac{3+\alpha-p}{p})^2},\quad &p\in \left(\tfrac23(3+\alpha), (3+\alpha)\right).
\end{array}
\right.
\endaligned
$$
so the assertion follows.
\end{proof}

\begin{lemma}\label{lem642}
$$
\|\nabla u_{\eps}\|_2^2=\frac{(Np-N-\alpha)\eps}{(N+\alpha)-p(N-2)}\|u_{\eps}\|_2^2, \quad\|u_{\eps}\|_q^q+\frac{2p\eps}{(N+\alpha)-p(N-2)}\|u_{\eps}\|_2^2=\int_{\R^N}(I_{\alpha}*|u_\eps|^{p})|u_\eps|^{p}dx.
$$
\end{lemma}
\begin{proof}
	Follows from Nehari and Poho\v zaev identities for \eqref{eqPeps}.
\end{proof}

\begin{lemma}\label{lem643}
$\frac{(q-2)N\eps}{q(N-2)-2N}\|u_{\eps}\|_2^2\leq C\overline{\sigma}_{\eps}$.  Moreover, we have
$$
\lim_{\eps\to 0}\eps\|u_{\eps}\|_2^2=0, \quad \lim_{\eps\to 0}\|\nabla u_{\eps}\|_2^2=0, \quad \lim_{\eps\to 0}\|u_{\eps}\|_{2^*}^2=0.
$$
$$
\lim_{\eps\to 0}\|u_{\eps}\|_q^q=\lim_{\eps\to 0}
\int_{\R^N}(I_{\alpha}*|u_\eps|^{p})|u_\eps|^{p}dx= \frac{2Np}{\alpha}c_{T\!F}=\mathcal{S}_{T\!F}^{\frac{N+\alpha}{\alpha}}.
$$
\end{lemma}
\begin{proof}
By Lemma~\ref{lem641}, we see that
\begin{equation}\label{eq642}
c_{T\!F}+o(1)=c_{\eps}=\mathcal{I}_{\eps}(u_{\eps})=\frac{1}{N}\|\nabla u_{\eps}\|_2^2+\frac{\alpha}{2Np}\int_{\R^N}(I_{\alpha}*|u_\eps|^{p})|u_\eps|^{p}dx.
\end{equation}
This, together with Lemma~\ref{lem642}, implies that there exists $C>0$ such that $\int_{\R^N}(I_{\alpha}*|u_\eps|^{p})|u_\eps|^{p}dx\geq C$ for all $\eps$ small.
Note that $\mathcal{P}_{T\!F}(u_{\eps})<0$, then there exists $t_{\eps}\in (0, 1)$ such that $\mathcal{P}_{T\!F}(u_{\eps}(x/t_{\eps}))=0$, which means that
$$
\|u_{\eps}\|_q^q
=t_{\eps}^{\alpha}\int_{\R^N}(I_{\alpha}*|u_\eps|^{p})|u_\eps|^{p}dx.
$$
If $t_{\eps}\to 0$ as $\eps\to 0$, then we must have $\|u_{\eps}\|_q^q\to 0$ as $\eps\to 0$, this contradicts
$$
0<c_{T\!F}\leq H(u_{\eps}(x/t_{\eps}))=\frac{\alpha t_{\eps}^{N+\alpha}}{2Np}\int_{\R^N}(I_{\alpha}*|u_\eps|^{p})|u_\eps|^{p}dx.
$$
Therefore, there exists $C>0$ such that $t_{\eps}>C$ for $\eps>0$ small.  Then we have
$$\aligned
c_{T\!F}\leq H(u_{\eps}(x/t_{\eps}))=&\mathcal{I}_{\eps}(u_{\eps}(x/t_{\eps}))-\Big(\frac{\eps t_{\eps}^N}{2}\| u_{\eps}\|_2^2+\frac{t_{\eps}^{N-2}}{2}\| \nabla u_{\eps}\|_2^2\Big)\\
\leq& \mathcal{I}_{\eps}(u_{\eps})-\frac{2p\eps}{(N+\alpha)-p(N-2)}\| u_{\eps}\|_2^2t_{\eps}^N\\
=& c_{\eps}-\frac{2p\eps}{(N+\alpha)-p(N-2)}\| u_{\eps}\|_2^2t_{\eps}^N,
\endaligned
$$
this means that
$$
\frac{2p\eps}{(N+\alpha)-p(N-2)}\| u_{\eps}\|_2^2\leq t_{\eps}^{-N}(c_{\eps}-c_{T\!F})\leq C\overline{\sigma}_{\eps}.
$$
Moreover, we have $\eps\| u_{\eps}\|_2^2\to 0$ as $\eps\to 0$.
It follows from Lemma~\ref{lem642} that
$$
\|\nabla u_{\eps}\|_2^2\to 0, \quad\| u_{\eps}\|_q^q+o(1)=\int_{\R^N}(I_{\alpha}*|u_\eps|^{p})|u_\eps|^{p}dx\to \frac{2Np}{\alpha}c_{T\!F}=\mathcal{S}_{T\!F}^{\frac{N+\alpha}{\alpha}}
$$
as $\eps\to 0$.
\end{proof}

Set $w_{\eps}(x)=u_{\eps}(\mathcal{S}_{T\!F}^{\frac{1}{\alpha}}x)$. Then we have, as $\eps\to 0$,
$$
\|w_{\eps}\|_q^q=\mathcal{S}_{T\!F}^{-\frac{N}{\alpha}}\|u_{\eps}\|_q^q\to \mathcal{S}_{T\!F},
$$
$$
\int_{\R^N}(I_{\alpha}*|w_\eps|^{p})|w_\eps|^{p}dx=\mathcal{S}_{T\!F}^{-\frac{N+\alpha}{\alpha}}
\int_{\R^N}(I_{\alpha}*|u_\eps|^{p})|u_\eps|^{p}dx\to 1.
$$
Let $\overline{w}_{\eps}(x)=w_{\eps}(x)/\|w_{\eps}(x)\|_q$ and $\overline{V}(x):=\mathcal{S}_{T\!F}^{-\frac{1}{q}}\tV(\mathcal{S}_{T\!F}^{\frac{1}{\alpha}}x)$, then we see that $\|\overline{w}_{\eps}\|_q^q=\|\overline{V}\|_q^q=1$ and
$$
\int_{\R^N}(I_{\alpha}*|\overline{w}_\eps|^{p})|\overline{w}_\eps|^{p}dx=\|w_{\eps}\|_q^{-\frac{2p}{q}}
\int_{\R^N}(I_{\alpha}*|w_\eps|^{p})|w_\eps|^{p}dx\to \mathcal{S}_{T\!F}^{-\frac{N+\alpha}{N}}=\mathcal{C}_\alpha=\int_{\R^N}(I_{\alpha}*|\overline{V}|^{p})|\overline{V}|^{p}dx.
$$
Thus
$\{\overline{w}_{\eps}\}$ is an optimizing sequence  for $\mathcal{C}_\alpha$.  Then similarly to the arguments in \cite[Section 4.4]{MM-14} it follows that for $\eps>0$ small there exists $\lambda_{\eps}>0$ such that
$$
\int_{B(0, \lambda_{\eps})}|\overline{w}_{\eps}(x)|^{q}dx=\int_{B(0, 1)}|\overline{V}(x)|^{q}dx.
$$
We define the rescaled family
$$
v_{\eps}(x):=\lambda_{\eps}^{\frac{N+\alpha}{2p}}\overline{w}_{\eps}(\lambda_{\eps}x),
$$
then
$$
\|v_{\eps}\|_{q}=1,\quad \int_{\R^N}(I_{\alpha}*|v_\eps|^{p})|v_\eps|^{p}dx=\mathcal{C}_{\alpha}+o(1),
$$
i.e., $\{v_{\eps}\}$ is a maximizing sequence for $\mathcal{C}_\alpha$.  Moreover,
$$
\int_{B(0, 1)}|v_{\eps}|^{q}dx=\int_{B(0, 1)}|\overline{V}(x)|^{q}dx.
$$

\begin{lemma}\label{lem645}
$\lim_{\eps\to 0}\|\overline{v}_{\eps}-\tV\|_{q}=0$, where $\overline{v}_{\eps}(x):=\lambda_{\eps}^{\frac{N+\alpha}{2p}}u_{\eps}(\lambda_{\eps}x)$.
\end{lemma}
\begin{proof}
It follows from Concentration-Compactness Principle of P.L.Lions \cite{L85II}*{Theorem 2.1} that 
$$
\lim_{\eps\to 0}\|v_{\eps}-\overline{V}\|_{q}=0,
$$
which, together with the definitions of $w_{\eps}(x)$ and $\overline{w}_{\eps}(x)$, implies that
$$
\lim_{\eps\to 0}\|\overline{v}_{\eps}-\tV\|_{q}=0.\qedhere
$$
\end{proof}

By a simple calculation, we see that
$$
v_{\eps}(x)=\frac{\lambda_{\eps}^{\frac{N+\alpha}{2p}}}{\|w_{\eps}\|_q}u_{\eps}(\lambda_{\eps}\mathcal{S}_{T\!F}^{\frac{1}{\alpha}}x)
$$
solves the equation
\begin{equation}
-\Delta v_\eps+\lambda_{\eps}^2\mathcal{S}_{T\!F}^{\frac{2}{\alpha}}\eps v_{\eps}=\|w_\eps\|_{q}^{2p-2}\lambda_{\eps}^{\frac{(N+\alpha)-p(N-2)}{p}}
\mathcal{S}_{T\!F}^{\frac{2+\alpha}{\alpha}}\Big((I_{\alpha}*|v_\eps|^{p})|v_\eps|^{p}v_\eps
-\|w_\eps\|_{q}^{q-2p}|v_\eps|^{q-2}v_\eps\Big).
\end{equation}
By the definition of $v_{\eps}$ and $w_{\eps}$, we obtain
\begin{equation}\label{eq646}
\|\nabla v_{\eps}\|_2^2=\|w_{\eps}\|_{q}^{-2}\lambda_{\eps}^{\frac{(N+\alpha)-p(N-2)}{p}}\mathcal{S}_{T\!F}^{\frac{2-N}{\alpha}}\|\nabla u_{\eps}\|_2^2,
\quad \|v_{\eps}\|_2^2=\|w_{\eps}\|_{q}^{-2}\lambda_{\eps}^{\frac{(N+\alpha)-Np}{p}}\mathcal{S}_{T\!F}^{\frac{-N}{\alpha}}\|u_{\eps}\|_2^2.
\end{equation}
It follows from Lemma~\ref{lem642} and Lemma~\ref{lem643} that
\begin{equation}\label{eq643}
\|w_{\eps}\|_{q}^{2}\lambda_{\eps}^{\frac{p(N-2)-(N+\alpha)}{p}}\mathcal{S}_{T\!F}^{\frac{N-2}{\alpha}}\|\nabla v_{\eps}\|_2^2=
\frac{(Np-N-\alpha)\eps}{(N+\alpha)-p(N-2)}\|w_{\eps}\|_{q}^{2}\lambda_{\eps}^{\frac{Np-(N+\alpha)}{p}}\mathcal{S}_{T\!F}^{\frac{N}{\alpha}}\|v_{\eps}\|_2^2\leq \overline{\sigma}_{\eps}.
\end{equation}

\begin{lemma}\label{lem646}
Let $\eps\to 0$. If $N\geq 4$ then
\begin{equation}\label{eq-HLS-0asy4+}
\lambda_{\eps}\sim\eps^{-\frac{1}{2}},
\end{equation}
while if $N=3$ then
\begin{equation}\label{eq-HLS-0asy+}
\left\{\begin{array}{rcll}
&\lambda_{\eps}&\sim\eps^{-\frac{1}{2}},&\qquad p\in \left(\tfrac13(3+\alpha), \tfrac23(3+\alpha)\right),\medskip\\
\eps^{-\frac{1}{2}}(\ln\tfrac{1}{\eps})^{-\frac{1}{2}}\lesssim &\lambda_{\eps}&\lesssim \eps^{-\frac{1}{2}}(\ln\tfrac{1}{\eps})^{\frac{1}{6}},&\qquad p=\tfrac23(3+\alpha),\medskip\\
\eps^{\frac{p-(3+\alpha)}{p}}\lesssim &\lambda_{\eps}&\lesssim\eps^{\frac{(3+\alpha)(3+\alpha-2p)}{p(3p-(3+\alpha))}},&\qquad p\in \left(\tfrac23(3+\alpha), 3+\alpha\right).
\end{array}
\right.
\end{equation}
\end{lemma}

\begin{proof}
Since $v_{\eps}\to \overline{V}$ strongly in $L^{q}(\R^N)$, we have $v_{\eps}\to \overline{V}$ strongly in $L^{s}_{loc}(\R^N)$ for $s\in [2, q)$, thus we get that
\begin{equation}\label{eq644}
\|\nabla v_{\eps}\|_2^2\geq C\|v_{\eps}\|_{2^*}^2\geq C\int_{B(0, 1)}|v_{\eps}|^{q}dx\to \int_{B(0, 1)}|\overline{V}|^{q}dx>0
\end{equation}
and
\begin{equation}\label{eq645}
\|v_{\eps}\|_2^2\geq \int_{B(0, 1)}|v_{\eps}|^2dx\to \int_{B(0, 1)}|\overline{V}|^{2}dx>0.
\end{equation}
Therefore, by \eqref{eq643}, \eqref{eq644} and \eqref{eq645}, we have
$$
C\overline{\sigma}_{\eps}^{\frac{p}{(N-2)p-(N+\alpha)}}\leq \lambda_{\eps}\leq C\eps^{\frac{p}{(N+\alpha)-Np}}\overline{\sigma}_{\eps}^{\frac{p}{Np-(N+\alpha)}}.
$$
Then \eqref{eq-HLS-0asy4+} and \eqref{eq-HLS-0asy+} follow directly from Lemma~\ref{lem641}.
\end{proof}

\begin{lemma}\label{lem647}
If either $N\geq 4$, or $N=3$ and $p\in (\frac{3+\alpha}{3}, \frac{2(3+\alpha)}{3})$, then
$\|\nabla v_{\eps}\|_2^2\sim 1$, $\|v_{\eps}\|_2^2\sim 1$.
\end{lemma}
\begin{proof}
Follows from \eqref{eq643}, \eqref{eq644}, \eqref{eq645} and Lemma~\ref{lem646}.
\end{proof}

\appendix
\section{A contraction inequality}

Consider the equation
\begin{equation}\label{e-AP}
-\Delta u+|u|^{q-2}u=f\quad\text{in $\R^N$},
\end{equation}
where $N\ge 2$ and $q>2$.
The existence for any $f\in L^1_{loc}(\R^N)$ of the unique distributional solution $u_f\in L^1_{loc}(\R^N)$ of \eqref{e-AP} is the result in \cite{Brezis-1984}*{Theorem 1}. The following remarkable contraction estimate on subsolutions for \eqref{e-AP} was communicated to us by Augusto Ponce.

\begin{theorem}\label{t-Ponce}
	Let	$0\le f\in L^s(\R^N)$ for some $s\ge 1$ and let $v\in L^1_{loc}(\R^N)$ be a nonnegative distributional sub-solution of \eqref{e-AP},
	i.e.
	\begin{equation}\label{e-APsub}
	-\Delta v+v^{q-1}\leq f\quad\text{in $\mathscr D'(\R^N)$}.
	\end{equation}
	Then $v^{q-1}\in L^s(\R^N)$ and
	\begin{equation}\label{e-Ponce-sub}
	\|v\|_{(q-1)s}\le\|f\|_s.
	\end{equation}
\end{theorem}

The inequality \eqref{e-Ponce-sub} is an extension to unbounded domains of the result in \cite{AP-book}*{Proposition 4.24},  see also \cite{AP-book}*{Exercise 4.15 and a hint on p.402}. We outline the arguments for completeness.

\begin{lemma}
Let	$0\le f\in L^s(\R^N)$ for some $s\ge 1$. For $m\in\N$, let $u_m\in L^1_{loc}(B_m)$ be the unique nonnegative distributional solution of
\begin{equation}\label{eq1201}
\left\{
\begin{array}{rl}
-\Delta u + u^{q-1}=f\quad &\text{in $B_m$},\\
u=0 \quad &\text{on $\partial B_m$}.
\end{array}
\right.
\end{equation}
Then $u_m\in L^s(B_m)$ and $\|u_m^{q-1}\|_s\leq \|f\|_s$.
\end{lemma}

\proof
Denote $B_{m,\rho}:=B_m\cap \{|u_m|^{q-1}<\rho\}$.
By the Cavalieri Principle (see e.g. \cite{AP-book}*{Proposition 1.7}), applied
with the measurable function $|u_m|^{q-1}$ and measure $d\nu:=|u_m|^{q-1}_{\chi\{|u|^{q-1}<\rho\}}dx$,
for every $\rho>0$ we have
$$\aligned
\int_{B_{m,\rho}}|u_m|^{(q-1)s}dx&=\int_{B_{m,\rho}}|u_m|^{(q-1)(s-1)}|u_m|^{q-1}_{\chi\{|u_m|^{q-1}<\rho\}}dx\\
&=\int_0^\rho \nu(\{|u_m|^{(q-1)(s-1)}>t\})dt\\
&=(s-1)\int_0^{\rho^{\frac{1}{s-1}}} \nu(\{|u_m|^{q-1}>\tau\})\tau^{s-2}d\tau \\
&=(s-1)\int_0^{\rho^{\frac{1}{s-1}}}\tau^{s-2}\Big(\int_{\{|u_m|^{q-1}\geq \tau\}}|u_m|^{q-1}_{\chi\{|u_m|^{q-1}<\rho\}}dx\Big)d\tau \\
&\leq (s-1)\int_0^{\rho^{\frac{1}{s-1}}}\tau^{s-2}\Big(\int_{\{\rho>|u_m|^{q-1}\geq \tau\}}|f|dx\Big)d\tau,
\endaligned
$$
where in the last line we used the following key inequality proved in \cite{AP-book}*{(4.12) on p.67},
$$\int_{\{|u_m|>k\}}|u_m|^{q-1}dx\leq \int_{\{|u_m|>k\}}|f|dx,\qquad\forall k>0.$$
Applying once again Cavalieri's Principle, this time with the measurable function $|u_m|^{q-1}$ and measure $d\bar\nu:=|f|_{\chi\{|u_m|^{q-1}<\rho\}}dx$, and using H\"older's inequality, we have
$$\aligned
\int_{B_{m,\rho}}|u_m|^{(q-1)s}dx&\leq (s-1)\int_0^{\rho^{\frac{1}{s-1}}} \tau^{s-2}\Big(\int_{\{|u_m|^{q-1}\geq \tau\}}|f|_{\chi\{|u_m|^{q-1}<\rho\}}dx\Big)d\tau \\
&=(s-1)\int_0^{\rho^{\frac{1}{s-1}}} \overline{\nu}(\{|u_m|^{q-1}>\tau\})\tau^{s-2}d\tau\\
&=\int_0^\rho \overline{\nu}(\{|u_m|^{(q-1)(s-1)}>t\})dt \\
&=\int_{B_{m,\rho}}|u_m|^{(q-1)(s-1)}|f|_{\chi\{|u_m|^{q-1}<\rho\}}dx \\
&\leq \Big(\int_{B_{m,\rho}}|u_m|^{(q-1)s}dx\Big)^{\frac{s-1}{s}}\|f\|_s,
\endaligned
$$
which implies
\begin{equation}\label{eqap01}
\int_{B_{m,\rho}}|u_m|^{(q-1)s}dx\leq \int_{B_{m,\rho}}|f|^{s}dx\leq \int_{\R^N}|f|^{s}dx.
\end{equation}
Since the bound \eqref{eqap01} holds uniformly for all $m\in\N$ and $\rho\to\infty$, the assertion follows.
\qed

\begin{proof}[Proof of Theorem \ref{t-Ponce}]
Let $m\in\N$ and $\{u_m\}$ be the sequence of solutions to \eqref{eq1201}.
Observe that $u_{m+1}$ also solves the equation on $B_m$ and $u_{m+1}\geq 0$ on $\partial B_m$. Thus, by the maximum principle on $B_m$ one has
$u_{m+1}\geq u_m$, so $\{u_m\}$ is an increasing sequence.  Moreover, it follows from \eqref{eqap01} that $\{u_m\}$ is locally bounded in $L^s(\R^N)$.  Then $\{u_m\}$ converges pointwise to a function $u_\infty$ that satisfies, by Fatou's lemma,
$$
\|u_\infty\|_{L^s(\R^N)}\leq \|f\|_{L^s(\R^N)}.
$$
By the monotone convergence theorem, $u_m^{q-1}\to u_\infty^{q-1}$
in $L^s(B_m)$, for every $m\in\N$.  Hence $u_\infty$ satisfies, for all $\varphi\in C_c^{\infty}(\R^N)$,
$$
-\int_{\R^N} u_\infty \Delta\varphi dx +\int_{\R^N}u_\infty^{q-1} \varphi dx = \int_{\R^N}f\varphi dx,
$$
i.e. $u_\infty$ is a distributional solution of \eqref{e-AP}.
Then $u_\infty=u_f$ is the unique solution of \eqref{e-AP}, by the Brezis's uniqueness result \cite{Brezis-1984}*{Theorem 1}.

Now let $v\in L^1_{loc}(\R^N)$ be a nonnegative distributional sub-solution of \eqref{e-AP} and set  $w=(v-u_f)^+$. By Kato's inequality \cite{Brezis-1984}*{Lemma A.1},
	\begin{align*}
	-\Delta w + (v^{q-1} - u_f^{q-1})\mathrm{sign}^+(w)\le 0  \quad\text{in $\mathscr D'(\R^N)$}.
	\end{align*}
	On the other hand, there is a $\delta>0$ such that $\delta(a-b)^{q-1}\le a^{q-1}-b^{q-1}$ for all $a\ge b\ge 0$. Then
	\begin{align*}
	-\Delta w + w^{q-1}\le 0  \quad\text{in $\mathscr D'(\R^N)$}
	\end{align*}
	and $w=0$ by \cite{Brezis-1984}*{Lemma 2}.
\end{proof}

\vspace{5pt}
{\small
\noindent{\bf Acknowledgements.}
The authors are grateful to Augusto Ponce for communicating Theorem \ref{t-Ponce}, to Cyrill Muratov for helpful comments on the draft of the paper, and to Bruno Volzone for bringing to our attention \cite{Volzone} and for helpful comments.
This work was initiated during a visit of Z.L.\ at Swansea University.
Part of this work was conducted when V.M.\ visited Suzhou University of Science and Technology.
The support and hospitality of both institutions are gratefully acknowledged.
}

\vspace{5pt}

{\small
	\noindent{\bf Data availability statement.}
Data sharing not applicable to this article as no datasets were generated or analysed during the current study.}

\end{document}